\documentclass[11pt]{article}

\usepackage{exscale,enumerate}
\usepackage{amsfonts,amssymb,latexsym,a4wide}
\usepackage{amsmath}
\usepackage{amsthm}
\usepackage{hyperref}
\usepackage[svgnames]{xcolor} 
\usepackage{verbatim}

\usepackage{graphicx}
\usepackage[all]{xy}
\SelectTips{cm}{}

\newcommand{\seq}{{\rm Seq}}

\newcommand{\Pol}{\cal{P}o\ell}
\newcommand{\ii}{ \textbf{\textit{i}}}
\newcommand{\jj}{ \textbf{\textit{j}}}
\newcommand{\kk}{ \textbf{\textit{k}}}

\newcommand{\Rn}[1]{R(#1)}
\newcommand{\RnL}[1]{R^{\Lambda}(#1)}
\newcommand{\Rnu}{R(\nu)}
\newcommand{\RnuL}{R^{\Lambda}(\nu)}

\DeclareMathOperator{\prD}{pr}
\DeclareMathOperator{\infD}{infl}

\newcommand{\et}[1]{\tilde{e_{#1}}}
\newcommand{\ft}[1]{\tilde{f_{#1}}}

\newcommand{\eL}[1]{e_{#1}^{\Lambda}}

\newcommand{\ets}[1]{\tilde{e_{#1}}^{\vee}}
\newcommand{\es}[1]{e_{#1}^{\vee}}
\newcommand{\fts}[1]{\tilde{f_{#1}}^{\vee}}
\newcommand{\eph}{\epsilon_i^{\vee}}
\newcommand{\ephj}{\epsilon_j^{\vee}}
\newcommand{\ephl}{\epsilon_{\ell}^{\vee}}
\newcommand{\ep}[1]{\epsilon_i(#1)}
\newcommand{\phiL}{\phi_i^{\Lambda}}
\newcommand{\phiO}{\phi_i^{\Omega}}
\newcommand{\jump}{{\rm jump}}
\newcommand{\Lc}[1]{\cal{L}(#1)}
\newcommand{\Lj}[1]{\cal{L}(#1)}
\newcommand{\prL}{\prD_{\Lambda}}
\newcommand{\prO}{\prD_{\Omega}}
\newcommand{\infL}{\infD_{\Lambda}}
\newcommand{\barM}{\overline{M}}
\newcommand{\unu}{\underline{\nu}}

\newcommand{\refequal}[1]{\xy {\ar@{=}^{#1}
(-1,0)*{};(1,0)*{}};
\endxy}

\newcommand{\xsum}[2]{
  \xy
  (0,.4)*{\sum};
  (0,3.7)*{\scs #2};
  (0,-2.9)*{\scs #1};
  \endxy
}

\usepackage{fancyheadings}
\newcommand{\To}{\Rightarrow}

\newcommand{\Hom}{{\rm Hom}}
\newcommand{\HOM}{{\rm HOM}}
\renewcommand{\to}{\rightarrow}
\newcommand{\maps}{\colon}

\newcommand{\iso}{\cong}
\newcommand{\id}{{\rm id}}

\newcommand{\wt}{{\rm wt}}
\newcommand{\soc}{{\rm soc\ }}
\newcommand{\cosoc}{{\rm cosoc\ }}
\newcommand{\chr}{{\rm ch}}

\newcommand{\scs}{\scriptstyle}

\theoremstyle{definition}
\newtheorem{thm}{Theorem}[section]
\newtheorem{cor}[thm]{Corollary}

\newtheorem{lem}[thm]{Lemma}
\newtheorem{rem}[thm]{Remark}
\newtheorem{prop}[thm]{Proposition}
\newtheorem{defn}[thm]{Definition}
\newtheorem{example}[thm]{Example}

\numberwithin{equation}{section}

\newcommand{\0}{{\mathbf 0}}

\let\hat=\widehat
\let\tilde=\widetilde

\let\phi=\varphi
\let\epsilon=\varepsilon

\usepackage{bbm}
\def\C{{\mathbbm C}}
\def\N{{\mathbbm N}}

\def\Z{{\mathbbm Z}}
\def\Q{{\mathbbm Q}}

\def\cal#1{\mathcal{#1}}%
\def\1{\mathbbm{1}}%
\def\nn{\notag}

\DeclareMathOperator{\Ind}{Ind}
\DeclareMathOperator{\Indc}{coInd}
\DeclareMathOperator{\Res}{Res}
\def\lra{{\longrightarrow}}
\def\dmod{{\mathrm{-mod}}}   
\def\fmod{{\mathrm{-fmod}}}   
\def\pmod{{\mathrm{-pmod}}}  
\def\gdim{{\mathrm{gdim}}}
\def\Ext{{\mathrm{Ext}}}

\def\mc{\mathcal}
\def\mf{\mathfrak}
\def\Af{{_{\mc{A}}\mathbf{f}}}    
\def\shuffle{\,\raise 1pt\hbox{$\scriptscriptstyle\cup{\mskip
               -4mu}\cup$}\,}
\newcommand{\define}{\stackrel{\mbox{\scriptsize{def}}}{=}}
\def\UmA{{_{\mc{A}}\mathbf{U}^-_q}}
\def\U{\mathbf{U}_q(\mf{g})}
\def\Um{\mathbf{U}_q^-}

\newcommand{\Gnu}{{G_0(\Rnu)^\ast}}
\newcommand{\Gdual}{G_0^\ast(R)}
\newcommand{\GnuL}{{G_0(\RnuL)^\ast}}
\newcommand{\GL}{G_0^\ast(R^\Lambda)}
\def\UpA{{_{\mc{A}}\mathbf{U}^+_q}}
\def\Up{{\mathbf{U}^+_q}}
\newcommand{\delM}[1]{\delta_{{#1}}}
\newcommand{\deLM}[1]{\mathfrak{d}_{{#1}}}
\newcommand{\deltriv}{\delM{\mathbbm 1}}
\newcommand{\deLtriv}{\deLM{\mathbbm 1}}
\newcommand{\epany}[2]{\epsilon_{#1}(#2)}
\newcommand{\rank}{\mathrm{rank}\/}
\newcommand{\VL}{V(\Lambda)}
\newcommand{\AV}{{}_{\mc{A}} V(\Lambda)}
\newcommand{\Vdual}{V(\Lambda)^\ast}
\newcommand{\AVdual}{{}_{\mc{A}} V^\ast(\Lambda)}
\newcommand{\mydot}{\cdot}

\newcommand{\up}[1]{\xybox{
   (-3,-13)*{};
  (3,8)*{};
 (0,0)*{\includegraphics[scale=0.5]{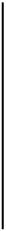}};
 (-.1,-12)*{\scs #1};
 }}

 \newcommand{\updot}[1]{\xybox{
   (-3,-13)*{};
  (3,8)*{};
 (0,0)*{\includegraphics[scale=0.5]{up.eps}};
 (-.1,-12)*{\scs #1};(0,0)*{\bullet};
 }}

\renewcommand{\sup}[1]{\xybox{
   (-3,-7)*{};
  (3,6)*{};
 (0,0)*{\includegraphics[scale=0.5]{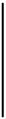}};
 (-.1,-7)*{\scs #1};
 }}

 \newcommand{\supdot}[1]{\xybox{
   (-3,-7)*{};
  (3,6)*{};
 (0,0)*{\includegraphics[scale=0.5]{short_up.eps}};
 (-.1,-7)*{\scs #1}; (0,0)*{\bullet};
 }}

\newcommand{\dcross}[2]{\xybox{
 (-6,-7)*{};
 (6,6)*{};
 (0,0)*{\includegraphics[scale=0.5]{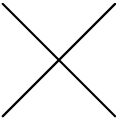}};
 (-5.1,-7)*{\scs #1};
 (4.7,-7)*{\scs #2};
}}

\newcommand{\dcrossul}[2]{\xybox{
(-2.5,2.5)*{\bullet};
 (-6,-7)*{};
 (6,6)*{};
 (0,0)*{\includegraphics[scale=0.5]{dcross.eps}};
 (-5.1,-7)*{\scs #1};
 (4.7,-7)*{\scs #2};
}}

\newcommand{\dcrossur}[2]{\xybox{
(2.5,2.5)*{\bullet};
 (-6,-7)*{};
 (6,6)*{};
 (0,0)*{\includegraphics[scale=0.5]{dcross.eps}};
 (-5.1,-7)*{\scs #1};
 (4.7,-7)*{\scs #2};
}}

\newcommand{\dcrossdl}[2]{\xybox{
(-2.5,-2.5)*{\bullet};
 (-6,-7)*{};
 (6,6)*{};
 (0,0)*{\includegraphics[scale=0.5]{dcross.eps}};
 (-5.1,-7)*{\scs #1};
 (4.7,-7)*{\scs #2};
}}

\newcommand{\dcrossdr}[2]{\xybox{
(2.5,-2.5)*{\bullet};
 (-6,-7)*{};
 (6,6)*{};
 (0,0)*{\includegraphics[scale=0.5]{dcross.eps}};
 (-5.1,-7)*{\scs #1};
 (4.7,-7)*{\scs #2};
}}

\newcommand{\twocross}[2]{\xybox{
 (-6,-13)*{};
 (6,8)*{};
 (0,0)*{\includegraphics[scale=0.5]{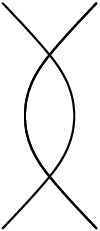}};
 (-4.1,-12)*{\scs #1};
 (3.7,-12)*{\scs #2};
}}

\newcommand{\linecrossL}[3]{\xybox{
 (-6,-13)*{};
 (6,8)*{};
 (0,0)*{\includegraphics[scale=0.5]{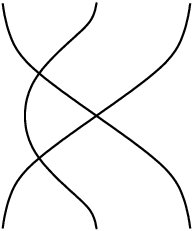}};
 (-8,-12)*{\scs #1};
 (0,-12)*{\scs #2};
 (8,-12)*{\scs #3};
}}

\newcommand{\linecrossR}[3]{\xybox{
 (-6,-13)*{};
 (6,8)*{};
 (0,0)*{\includegraphics[angle=180, scale=0.5]{line_crossL.eps}};
 (-8,-12)*{\scs #1};
 (0,-12)*{\scs #2};
 (8,-12)*{\scs #3};
}}

\title{Crystals from categorified quantum groups}
      \author{Aaron D. Lauda and Monica Vazirani}
      \date{September 10, 2009}

\begin{document}
%

\maketitle

\begin{abstract}

We study the crystal structure on categories of graded modules over
algebras which categorify the negative half of the quantum Kac-Moody
algebra associated to a symmetrizable Cartan data.   We identify
this crystal with Kashiwara's crystal for the corresponding negative
half of the quantum Kac-Moody algebra. As a consequence, we show the
simple graded modules for certain cyclotomic quotients carry the
structure of highest weight crystals, and hence compute the rank of
the corresponding Grothendieck group.
\end{abstract}

\tableofcontents

%
\section{Introduction}
\label{sec_intro}
%

In \cite{KL, KL2,Ro}
a family $R$ of graded algebras was introduced
that categorifies the integral form $\UmA := \UmA(\mf{g})$ of the
negative half of the quantum enveloping algebra $\U$ associated to a
symmetrizable Kac-Moody algebra $\mf{g}$.    The grading on these
algebras equips the Grothendieck group $K_0(R\pmod)$ of the category
of finitely-generated graded projective $R$-modules with the
structure of a $\Z[q,q^{-1}]$-module, where $q^r[M] := [M\{r\}]$,
and $M\{r\}$ denotes a graded module $M$ with its grading shifted up
by $r$. Natural parabolic induction and restriction functors give
$K_0(R\pmod)$ the structure of a (twisted) $\Z[q,q^{-1}]$-bialgebra.
In \cite{KL,KL2} an explicit isomorphism of twisted bialgebras was
given between $\UmA$ and $K_0(R\pmod)$.  The crystal-theoretic methods in this paper provide a new proof of this result.

Several conjectures were also made in \cite{KL, KL2}.  One conjecture that was unproven at the time this article first appeared is the so called cyclotomic
quotient conjecture which suggests a close connection between
certain finite dimensional quotients of the algebras $R$ and the
integrable representation theory of quantum Kac-Moody algebras.
At that time, the conjecture had been proven in finite and affine type $A$
by Brundan and Kleshchev~\cite{BK2}, but very little was known in the
case of an arbitrary symmetrizable Cartan datum.  By obtaining
new results on the fine structure of simple $R$-modules, here we show
that simple graded modules for these cyclotomic quotients carry the
structure of highest weight crystals. Hence we identify the rank of
the corresponding Grothendieck group with the rank of the integral
highest weight representation, proving a major component of the cyclotomic quotient conjecture.  Before this article went to press, proofs of the full conjecture appeared independently in work of Webster~\cite{Web} and Kang and Kashiwara~\cite{KK}.

To explain these results more precisely, suppose we are given a symmetrizable Cartan datum where $I$ is the index set of simple roots. The algebras $R$ have a diagrammatic description and are determined by the symmetrizable Cartan datum of $\mathfrak{g}$ together with some extra parameters. In the literature these
algebras are sometimes called Khovanov-Lauda-Rouquier algebras and quiver Hecke algebras.

For each $\nu \in \N[I]$ the block $\Rn{\nu}$ of the algebra $R$
admits a finite dimensional quotient $\RnuL$ associated to  a
highest weight $\Lambda$, called a cyclotomic quotient. These
quotients were conjectured in \cite{KL,KL2} to categorify the
$\nu$-weight space of the integral version of the irreducible
representation $V(\Lambda)$ of highest weight $\Lambda$ for
$\mathbf{U}_q(\mf{g})$, in the sense that there should be an
isomorphism
\[
 \xy {\ar^-{\iso}(-22,0)*+{V(\Lambda)_{\C}};
 (20,0)*+{\bigoplus_{\nu \in \N[I]} K_0(\RnuL\pmod)_{\C},}}; \endxy
\]
where $K_0(\RnuL\pmod)_{\C}$ denotes the complexified Grothendieck
group of the category of graded finitely generated projective
$\RnuL$-modules. A special case of this conjecture was proven in
type $A$ by Brundan and Stroppel~\cite{BS}. The more general
conjecture was proven in finite and affine type $A$ by Brundan and
Kleshchev~\cite{BK1,BK2}. They constructed an isomorphism
\[
 \xy {\ar^{\iso}(-8,0)*+{ \RnuL}; (8,0)*+{H_{\nu}^{\Lambda}}}; \endxy,
\]
where $H_{\nu}^{\Lambda}$ is a block of the cyclotomic affine Hecke
algebra $H_{m}^{\Lambda}$ as defined in \cite{AK,BM,Ch}.  This
isomorphism induces a new grading on blocks of the cyclotomic affine
Hecke algebra. This has led to the definition of graded Specht
modules for cyclotomic Hecke algebras~\cite{BKW}, the construction
of a homogeneous cellular basis for the cyclotomic quotients $\RnuL$
in type $A$~\cite{HM}, the introduction of gradings in the study of
$q$-Schur algebras~\cite{Ari4}, and an extension of the generalized
LLT conjecture to the graded setting~\cite{BK2}.

Ariki's categorification theorem gave a geometric proof that the sum
of complexified Grothendieck groups of cyclotomic Hecke algebras
$H^{\Lambda}_m$ at an $N$th root of unity over $\C$, taken over all
$m\geq 0$, was isomorphic to the highest weight representation
$V(\Lambda)$ of $U(\hat{\mathfrak{sl}}_N)$ ~\cite{Ari1}, see
\cite{Ari2,Ari3,AM,Mathas} and also \cite{Groj2,LLT}.  Grojnowski
gave a purely algebraic proof of this result, parameterizing the
simple $H_m^{\Lambda}$-modules in terms of crystal data of highest
weight crystals~\cite{Groj}.

Brundan and Kleshchev's proof of the cyclotomic quotient conjecture
in type $A$ utilized the isomorphism between the graded algebras
$\RnuL$ and blocks of the cyclotomic affine Hecke algebra, allowing
them to extend Grojnowski's crystal theoretic classification of
simples of the ungraded affine Hecke algebra to the graded setting.
By keeping careful track of the gradings, they were
able to extend Ariki's theorem to the graded setting, thereby
proving the cyclotomic quotient conjecture in type $A$, as well as
identifying the indecomposable projective modules for
$R^{\Lambda}_{\nu}$ with the  
canonical basis for
$V(\Lambda)$. Indeed, the algebras $\RnuL$ were originally called
cyclotomic quotients in \cite{KL} because they were expected to categorify irreducible highest weight representations of quantum Kac-Moody algebras analogous to the way that cyclotomic Hecke algebras categorify irreducible highest weight representations for type $A$ in the non-quantum setting.  In this way, these diagrammatically defined cyclotomic quotients can be viewed as graded extensions of the cyclotomic Hecke algebras to all types.

While there are natural extensions of cyclotomic Hecke algebras of type $A$, namely quotients of affine Hecke algebras of crystallographic type, they do not provide analogous categorification results. However, categorification results of a different flavour do exist in types $B$ and $D$, see \cite{VV2,SVV,EK,KM}.

In type $A$ homogeneous cellular bases were constructed~\cite{KR2,BS,HM}. However, the study of cyclotomic quotients outside of type $A$ has been
hindered by the lack of explicit bases for the algebras $\RnuL$.  Some explicit calculations of cyclotomic quotients $\RnuL$ in other type were made
for level one and two representations~\cite{RTG}, but it is not
clear how to extend these results to all representations. The
algebras $\Rnu$ have a PBW basis that aid in computations. No such
basis is known for the algebras $\RnuL$.
\medskip

In the symmetric case the algebras $R$ are related to Lusztig's
geometric categorification using perverse sheaves. Following
Ringel~\cite{Ringel}, Lusztig gave a geometric interpretation of
$\Um=\Um(\mf{g})$~\cite{Lus1,Lus2,Lus3}, see also \cite{Lus4,Lus6}.
This gave rise to Lusztig's canonical basis for $\Um$. Kashiwara defined a
crystal basis of $\Um$ for certain simple Lie algebras~\cite{Kas1}
and later proved its existence for all symmetrizable Kac-Moody
algebras~\cite{Kas2,Kas5}; the affine type $A$ case was proven by
Misra and Miwa~\cite{MT}. Kashiwara also constructed the so-called
global crystal basis of $\Um$~\cite{Kas2,Kas3,Kas5}. Grojnowski and
Lusztig~\cite{GL2} proved that the global crystal basis and the
canonical basis are the same. The canonical basis
of $\Um$ is a basis with remarkable positivity and integrality
properties, and gives rise to bases in all irreducible integrable
$\mathbf{U}_q(\mf{g})$-representations.

Varagnolo and Vasserot constructed an isomorphism between
$\Ext$-algebras of certain simple perverse sheaves on Lusztig quiver
varieties~\cite{VV} and the algebras $R(\nu)$ in the symmetric case,
proving a conjecture from \cite{KL}. Consequently, one can identify
indecomposable projectives for the algebras $R$ with simple perverse
sheaves on Lusztig quiver varieties and the
canonical basis for $\UmA$. Rouquier has also announced a similar
result.

One should be able to deduce a classification of graded simple
modules for the algebras $\RnuL$ in the symmetric case using results
of \cite{KL} and \cite{VV} together with Kashiwara and Saito's
geometric construction of crystals~\cite{KS}, but the details of
this argument have not appeared.  We expect that cyclotomic
quotients $\RnuL$ should also have a geometric interpretation in
terms of Nakajima quiver varieties~\cite{Na}. \medskip

In this paper we determine the size of the Grothendieck group for
arbitrary cyclotomic quotients $\RnuL$ associated to a symmetrizable
Cartan datum.  Rather than working geometrically, our methods are
based strongly on the algebraic treatment of the affine Hecke
algebra and its cyclotomic quotients introduced by
Grojnowski~\cite{Groj}. This approach extended Kleshchev's results
for the symmetric groups~\cite{Kle1,Kle2,Kle3}, and utilizes earlier
results of Vazirani~\cite{Vaz1,Vaz2} and
Grojnowski-Vazirani~\cite{GV}. Kleshchev's book contains an
excellent exposition of Grojnowski's approach in the context of
degenerate affine Hecke algebras~\cite{KleBook}. The idea is to
introduce a crystal structure on categories of modules, interpreting
Kashiwara operators module theoretically. To apply this approach to the study of algebras $R(\nu)$, rather than working with projective modules, one must work with the category of finite dimensional graded $R(\nu)$-modules.  This could be done by working over an algebraically closed field $\Bbbk$ and utilizing the $\Z[q,q^{-1}]$-bilinear pairing
\begin{equation}
 (,) \maps K_0(R(\nu)\pmod) \times G_0(R(\nu)\fmod) \to \Z[q,q^{-1}],
\end{equation}
where $G_0(R(\nu)\fmod)$ denotes the Grothendieck group of the
category of finite dimensional graded $R(\nu)$-modules. Since the pairing is a perfect pairing (see \cite{KL}), it allows one to deduce that Serre relations hold on $G_0(R)$ from the corresponding result for $K_0(R)$. Here, however, we take a more direct approach giving a direct proof of Serre relations on $G_0(R)$ and a more direct identification of $G_0(R)$ with $\UmA$. This is a byproduct of our careful analysis, which additionally yields new results on the structure of simple modules.

We study the crystal graph whose nodes are the graded
simple $\Rn{\nu}$-modules (up to grading shift) taken over all $\nu
\in \N[I]$. By identifying this crystal graph with the Kashiwara
crystal $B(\infty)$ associated to $\Um$ we are able to define a
crystal structure on the set of graded simple modules for the
cyclotomic quotients $\RnuL$ and show that it is the crystal graph
$B(\Lambda)$. This allows us to view cyclotomic quotients of the
algebras $\Rn{\nu}$ as a categorification of the integrable highest
weight representation $V(\Lambda)$ of $\Up$, proving {\em part} of
the cyclotomic quotient conjecture from \cite{KL} in the general
setting.  This does not prove the entire cyclotomic quotient
conjecture as our isomorphism is only an isomorphism of
$\Up$-modules, not of $\mathbf{U}_q(\mf{g})$-modules.
%

The study of KLR algebras and their cyclotomic quotients is rapidly developing.  On the same day that this posted to the ArXiV, an article by Kleshchev and Ram \cite{KR} also appeared where they construct all irreducible representations of algebras $R(\nu)$ in finite type from Lyndon words.  Their work generalizes the fundamental work of \cite{BZ,Zel} who parameterized and constructed the simple modules for the affine Hecke algebra in type $A$ with generic parameter in terms of $\mathbf{U}^-(\mf{gl}_\infty)$.  Furthermore, some time after this article appeared alternative proofs of the full cyclotomic quotient conjecture were given by Webster~\cite{Web} and by Kang-Kashiwara~\cite{KK}.  Kang and Kashiwara show that functors lifting the action of $E_i$ and $F_i$ in $\U$ are biadjoint, showing that cyclotomic quotients categorify $V(\Lambda)$ as $\U$-modules and give a 2-representation in the sense Rouquier~\cite{Ro}.  Webster's work gives a different proof of biadjointness and also constructs an action of the 2-category $\dot{\cal{U}}$ from \cite{Lau1,KL} on categories of modules over cyclotomic quotients.

This article gives a proof of the crystal version of the cyclotomic quotient conjecture.  This work differs from the articles mentioned above in that it requires a detailed study of the fine structure simple modules for cyclotomic quotients.  We feel that this fine structure constitutes the main results obtained in this article.  These results are strong enough to give an alternative proof of the categorification theorem of \cite{KL,KL2} staying entirely in the category of finitely-generated modules, see Section~\ref{sec_genrel}.

All of the results in this paper should extend to Rouquier's version
of algebras $R(\nu)$ associated to Hermitian matrices, at least for
those Hermitian matrices leading to graded algebras.  We also
believe that these results will fit naturally within Khovanov and
Lauda's framework of categorified quantum groups \cite{Lau1,KL3}, as
well as Rouquier's 2-representations of 2-Kac-Moody
algebras~\cite{CR,Ro}.

\bigskip

We end the introduction with a brief outline of the article,
highlighting other results to be found herein.  In
Section~\ref{sec_Rnu} we review the definition and key properties of
the algebras $\Rn{\nu}$.  In Section~\ref{sec_functors} we study
various functors defined on the categories of graded modules over
the algebras $R(\nu)$.  In particular, Section~\ref{sec_coind}
introduces the co-induction functor and proves several key results.
In Section~\ref{sec_Groth} we look at the morphisms induced by these
functors on the Grothendieck rings.

Section~\ref{sec_crystals} contains a brief review of crystal
theory.  Of key importance is the result of Kashiwara and
Saito~\cite{KS}, recalled in Section~\ref{sec_Binf}, characterizing
the crystal $B(\infty)$.  In Section~\ref{sec_mod_crystals} we
introduce crystal structures on the category of modules over
algebras $R(\nu)$ and their cyclotomic quotients $\RnuL$.  After a
detailed study of this crystal data in Section~\ref{sec_modules},
these crystals are identified as the crystals $B(\infty)$ and
$B(\Lambda)$ in Section~\ref{sec_harvest}.

\bigskip

\noindent {\it Acknowledgments:}  We thank Ian Grojnowski for
suggesting this project and for many helpful discussions.  We also
thank Mikhail Khovanov for helpful discussions and comments. The
first author was partially supported by the NSF grant DMS-0739392
and DMS-0855713. The second author was partially supported by the
NSA grant \#H982300910076, and would like acknowledge Columbia's RTG
grant DMS-0739392 for supporting her visits to Columbia.

%
\subsection{The algebras $\Rn{\nu}$}
\label{sec_Rnu}
%

%
\subsubsection{Cartan datum} \label{sec_Cartan}
%

Assume we are given a Cartan data
\begin{align*}
  & \text{P - a free $\Z$-module (called the weight lattice)} \\
  & \text{I - an index set for simple roots}\\
  & \text{$\alpha_i \in P$ for $i \in I$ called simple roots}\\
& \text{$h_i\in P^\vee = \Hom_{\Z}(P,\Z)$ called simple coroots}
\\
  & \text{$(,) \maps P \times P \to \Z$ a bilinear form}
\end{align*}
where we write $\langle \cdot, \cdot \rangle \maps P^{\vee} \times P
\to \Z$ for the canonical pairing.  This data is required to satisfy
the following axioms
\begin{align}
  & (\alpha_i, \alpha_i) \in 2\Z_{>0} \;\; \text{for any $i\in I$} \\
  & \langle h_i, \lambda \rangle = 2 \frac{(\alpha_i,\lambda)}{(\alpha_i,\alpha_i)}
  \;\;\text{for $i \in I$ and $\lambda \in P$} \\
  & (\alpha_i,\alpha_j) \leq 0 \;\; \text{for $i,j\in I$ with $i \neq j$.}
\end{align}
Hence $\{\langle h_i, \alpha_j \rangle\}_{i,j\in I}$ is a symmetrizable generalized Cartan matrix. In what follows we write
\begin{equation}
 a_{ij} =-\langle i,j \rangle := -\langle h_i, \alpha_j \rangle
\end{equation}
for $i,j\in I$.

Let $\Lambda_i\in P^+$ be the fundamental weights defined by
$\langle h_j,\Lambda_i \rangle =\delta_{ij}$.

%
\subsubsection{The algebra $\Um$}
\label{sec_UZ}
%

Associated to a Cartan datum one can define an algebra $\Um$, the
quantum deformation of the universal enveloping algebra of the
``lower-triangular'' subalgebra of a symmetrizable Kac-Moody algebra
$\mf{g}$.  Our discussion here follows Lusztig~\cite{Lus4}.

Let $q_i=q^{\frac{(\alpha_i,\alpha_i)}{2}}$,
$[a]_i=q_i^{a-1}+q_i^{a-3}+\dots + q_i^{1-a}$,
$[a]_i!=[a]_i[a-1]_i\dots [1]_i$. Denote by
$'\mathbf{f}$\index{f@$'\mathbf{f}$} the free associative algebra
over $\Q(q)$ with generators $\theta_i$, $i\in I$, and introduce
q-divided powers $\theta_i^{(a)}= \theta_i^a/[a]_i!$. The algebra
$'\mathbf{f}$ is $\N[I]$-graded, with $\theta_i$ in degree $i$. The
tensor square $'\mathbf{f}\otimes {}'\mathbf{f}$ is an associative
algebra with twisted multiplication
$$ (x_1\otimes x_2) (x_1'\otimes x_2') =q^{- |x_2|\cdot |x_1'|} x_1 x_1' \otimes x_2 x_2'$$
for homogeneous $x_1, x_2, x_1', x_2'$. The assignment $r(\theta_i) = \theta_i
\otimes 1 + 1\otimes \theta_i$ extends to a unique algebra homomorphism $r:
{}'\mathbf{f}\lra {}'\mathbf{f}\otimes {}'\mathbf{f}$.

The algebra $'\mathbf{f}$ carries a $\Q(q)$-bilinear form determined by
the conditions
\begin{itemize}
\item $(1,1)=1$,
\item $(\theta_i, \theta_j) = \delta_{i,j} (1-q_i^2)^{-1} $ for $i,j\in I$,
\item $(x, y y') = (r(x), y \otimes y')$ for $x,y, y' \in {}'\mathbf{f}$,
\item $(x x', y) = (x \otimes x', r(y))$ for $x, x', y\in {}'\mathbf{f}$.
\end{itemize}
The bilinear form $(,)$ is symmetric. Its radical $\mf{I}$ is a two-sided ideal
of $'\mathbf{f}$. The form $(,)$ descends to a nondegenerate form on the
associative $\Q(q)$-algebra $\mathbf{f} = {}'\mathbf{f}/\mf{I}$.

\begin{thm} \label{thm_quantumGK} The ideal $\mf{I}$ is generated by the elements
$$ \sum_{r+s=a_{ij}+1} (-1)^r \theta_i^{(r)} \theta_j \theta_i^{(s)} $$
over all $i,j\in I, $ $i\not= j$.
\end{thm}
For a general Cartan datum, the only known proof of this theorem
requires Lusztig's geometric realization of $\mathbf{f}$ via
perverse sheaves. This proof is given in his
book~\cite[Theorem~33.1.3]{Lus4}. Less sophisticated proofs exist
when the Cartan datum is finite.

\begin{rem} \label{rem_nosmaller}
Theorem~\ref{thm_quantumGK} implies that $\mathbf{f}$ is the
quotient of $'\mathbf{f}$ by the quantum Serre relations
\begin{equation}\label{rels-serre}
 \sum_{r+s=a_{ij}+1} (-1)^r \theta_i^{(r)} \theta_j \theta_i^{(s)} =0.
\end{equation}
Furthermore, since $\mathbf{f}$ is
an $\N[I]$-graded quotient of a free algebra, it
also implies that there are no smaller degree relations in $\mathbf{f}$. In
particular, \eqref{rels-serre} can never hold for $r+s=c+1$  with
$c<a_{ij}$.
\end{rem}

Let $\U$ denote the quantum enveloping algebra of a symmetrizable
Kac-Moody algebra $\mf{g}$.  There is a pair of injective algebra
homomorphisms $\mathbf{f} \to \U$, which sends $\theta_i \mapsto
e_i$, respectively $\theta_i \mapsto f_i$. We denote the images of
these homomorphisms as $\mathbf{U}_q^+(\mf{g})$ and
$\mathbf{U}_q^{-}(\mf{g})$. Let $\cal{A} = \Z[q,q^{-1}]$. The
integral form of the algebra $\mathbf{f}$, denoted $\Af$, is the
$\Z[q,q^{-1}]$-subalgebra of $\mathbf{f}$ generated by the divided
powers $\theta_i^{(a)}$, over all $i\in I$ and $a\in \N$. We write
$\UmA$ for the corresponding integral form of the negative half of
the quantum enveloping algebra $\U$.  The algebra $\Af$ admits a
decomposition into weight spaces $\Af = \bigoplus_{\nu \in \N[I]}
\Af(\nu)$.

In the next section we introduce graded algebras $R(\nu)$ whose
Grothendieck ring was shown by Khovanov and Lauda to be isomorphic
to $\Af$ as bialgebras, see Theorem~\ref{thm_KL}.

%
\subsubsection{The definition of the algebra $\Rn{\nu}$} \label{sec_defnRnu}
%

Recall the definition from~\cite{KL,KL2} of the algebra $R$
associated to a Cartan datum. Let $\Bbbk$ be an algebraically closed
field (of arbitrary characteristic). The algebra $R$ is defined by
finite $\Bbbk$-linear combinations of braid--like diagrams in the
plane, where each strand is coloured by a vertex $i \in I$.  Strands
can intersect and can carry dots; however, triple intersections are
not allowed.  Diagrams are considered up to planar isotopy that do
not change the combinatorial type of the diagram. We recall the
local relations
\begin{eqnarray} \label{new_eq_UUzero}
   \xy   (0,0)*{\twocross{i}{j}}; \endxy
 & = & \left\{
\begin{array}{ccl}
  0 & \qquad & \text{if $i=j$, } \\ \\
  \xy (0,0)*{\sup{i}};  (8,0)*{\sup{j}};  \endxy
  & &
 \text{if $(\alpha_i, \alpha_j)=0$, }
  \\    \\
  \vcenter{ \xy
   (0,0)*{\supdot{i}};
   (8,0)*{\sup{j}};
   (-3.5,2)*{\scs a_{ij}};
  \endxy}
    \;\; + \;\;
   \vcenter{\xy  (8,0)*{\supdot{j}};(0,0)*{\sup{i}};
   (11.5,2)*{\scs a _{ji}};
   \endxy}
 & &
 \text{if $(\alpha_i, \alpha_j)\not= 0$}.
\end{array}
\right.
\end{eqnarray}

\begin{eqnarray}\label{new_eq_ijslide}
  \xy  (0,0)*{\dcrossul{i}{j}};  \endxy
 \quad  =
   \xy  (0,0)*{\dcrossdr{i}{j}};   \endxy
& \quad &
   \xy  (0,0)*{\dcrossur{i}{j}};  \endxy
 \quad = \;\;
   \xy  (0,0)*{\dcrossdl{i}{j}};  \endxy \qquad \text{for $i \neq j$}
\end{eqnarray}

\begin{eqnarray}        \label{new_eq_iislide1}
 \xy  (0,0)*{\dcrossul{i}{i}}; \endxy
    \quad - \quad
 \xy (0,0)*{\dcrossdr{i}{i}}; \endxy
  & = &
 \xy (-3,0)*{\sup{i}}; (3,0)*{\sup{i}}; \endxy \\      \label{eq_iislide2}
  \xy (0,0)*{\dcrossdl{i}{i}}; \endxy
 \quad - \quad
 \xy (0,0)*{\dcrossur{i}{i}}; \endxy
  & = &
 \xy (-3,0)*{\sup{i}}; (3,0)*{\sup{i}}; \endxy
\end{eqnarray}

\begin{eqnarray}      \label{new_eq_r3_easy}
\xy  (0,0)*{\linecrossL{i}{j}{k}}; \endxy
  &=&
\xy (0,0)*{\linecrossR{i}{j}{k}}; \endxy
 \qquad \text{unless $i=k$ and $(\alpha_i, \alpha_j)\not= 0$   \hspace{1in} }
\\
\xy (0,0)*{\linecrossL{i}{j}{i}}; \endxy
  &-&
\xy (0,0)*{\linecrossR{i}{j}{i}}; \endxy
 \quad = \quad
 \sum_{a=0}^{a_{ij}-1}
 \xy
 (-9,0)*{\updot{i}};
 (-6.5,3)*{\scs a};
 (0,0)*{\up{j}};
 (9,0)*{\updot{i}};
 (17,3)*{\scs a_{ij}-1-a};\endxy
 \qquad \text{if $(\alpha_i, \alpha_j)\neq 0 $} \nn\\ \label{eq_r3_hard}
\end{eqnarray}

Left multiplication is given by concatenating a diagram on top of another diagrams when the corresponding endpoints have the same colours, and is defined to be zero
otherwise. The algebra is graded where generators are defined to
have degrees
\begin{equation}
  \deg\left( \xy  (-3,0)*{\supdot{i}}; \endxy \right) = (\alpha_i,\alpha_i) , \qquad \deg\left(  \xy
  (0,0)*{\dcross{i}{j}};  \endxy \right) = - (\alpha_i, \alpha_j).
\end{equation}

For $\nu = \sum_{i \in I} \nu_i \cdot i \in \N[I]$ let $\seq(\nu)$
be the set of all sequences of vertices $\ii = i_1 \ldots i_m$ where
$i_r \in I$ for each $r$ and vertex $i$ appears $\nu_i$ times in the
sequence.  The length $m$ of the sequence is equal to $|\nu|=\sum_{i
\in I} \nu_i$. It is sometimes convenient to identify $\nu = \sum_{i
\in I} \nu_i \cdot i \in \N[I]$ as  $\nu \in \sum_{i \in I} \nu_i
\alpha_i \in Q_+=\oplus_{i \in I} \Z_{\geq 0} \alpha_i$. We denote $Q_{-}=-Q_{+}=\oplus_{i \in I} \Z_{\leq 0} \alpha_i$. The algebra
$R$ has a decomposition
\begin{equation}
R = \bigoplus_{\nu \in \N[I]} \Rn{\nu}
\end{equation}
where $\Rn{\nu}$ is the subalgebra generated by diagrams that
contain $\nu_i$ strands coloured $i$.

To convert from graphical to algebraic notation write
\begin{eqnarray}
  1_{\ii} \quad := \quad
  \xy (0,0)*{\sup{}};  (0,-6)*{\scs i_1};  \endxy
    \dots
    \xy (0,0)*{\sup{}};  (0,-6)*{\scs i_k}; \endxy
   \dots
    \xy (0,0)*{\sup{}};  (1,-6)*{\scs i_m}; \endxy
\end{eqnarray}
for $\ii=i_1 i_2\dots i_m\in \seq(\nu)$.  The elements $1_{\ii}$ are idempotents in the ring $\Rn{\nu}$ and when $I$ is finite, $1_{\nu}\in \Rn{\nu}$ is given by $1_{\nu}=\sum_{i \in \seq(\nu)}1_{\ii}$. For $1\le r \le m$ we denote
\begin{eqnarray} \label{eq_dot_xki}
  x_{r,\ii} \quad := \quad
  \xy (0,0)*{\sup{}};  (0,-6)*{\scs i_1};  \endxy
    \dots
    \xy (0,0)*{\supdot{}};  (0,-6)*{\scs i_r}; \endxy
   \dots
    \xy (0,0)*{\sup{}};  (1,-6)*{\scs i_m}; \endxy
\end{eqnarray}
with the dot positioned on the $r$-th strand counting from the left, and
\begin{eqnarray} \label{eq_crossing_delta}
  \psi_{r,\ii} \quad := \quad
  \xy  (0,0)*{\sup{}};  (0,-6)*{\scs i_1};   \endxy
    \dots
  \xy (0,0)*{\dcross{i_r}{\; \; i_{r+1}}}; \endxy
   \dots
  \xy (0,0)*{\sup{}}; (1,-6)*{\scs i_m}; \endxy
\end{eqnarray}
The algebra $\Rn{\nu}$ decomposes as a vector space
\begin{equation}
\Rn{\nu} = \bigoplus_{\ii,\jj \in
\seq(\nu)} 1_{\jj} \Rn{\nu} 1_{\ii}
\end{equation}
where $1_{\jj}\Rn{\nu}1_{\ii}$ is the $\Bbbk$-vector space of all
linear combinations of diagrams with sequence $\ii$ at the bottom and sequence
$\jj$ at the top modulo the above relations.

The symmetric group $S_m$ acts on $\seq(\nu)$, $m=|\nu|$ by
permutations. Transposition $s_r=(r,r+1)$ switches entries $i_r,
i_{r+1}$ of $\ii$.  Thus, $\psi_{r,\ii} \in
{1_{s_r(\ii)}\Rn{\nu}1_{\ii}}$. For each $w\in S_m$ fix once and for
all a reduced expression $\hat{w} = s_{w_1}s_{w_2}\dots s_{w_t}$.
Given $w \in S_n$ we convert its reduced expression $\hat{w}$ into
an element of $1_{w(\ii)}\Rn{\nu}1_{\ii}$ denoted
$\psi_{\hat{w},\ii}=\psi_{w_1,s_{w_2}\cdots s_{w_t}(\ii)}\cdots
\psi_{w_{t-1},s_{w_t}(\ii)} \psi_{w_t,\ii}$.  To simplify notation
we introduce elements
\begin{equation}
  x_r := \sum_{\ii \in \seq(\nu)}x_{r,\ii}, \qquad \quad \psi_{\hat{w}} = \sum_{\ii \in \seq(\nu)} \psi_{\hat{w},\ii}
\end{equation}
so that $x_r1_{\ii} = 1_{\ii}x_r= x_{r,\ii}$ and
$\psi_{\hat{w}}1_{\ii} = 1_{w(\ii)}\psi_{\hat{w}}=
\psi_{\hat{w},\ii}$. This allows us to write the definition of the
algebra $\Rn{\nu}$ as follows:

For $\nu \in \N[I]$ with $|\nu|=m$, let $\Rn{\nu}$ denote the associative, $\Bbbk$-algebra on generators
\begin{align}
 & 1_{\ii}      & \text{for $\ii \in \seq(\nu)$} \\
 & x_{r}   & \text{for $1 \leq r \leq m$} \\
 & \psi_{r}     & \text{for $1 \leq r \leq m-1$}
\end{align}
subject to the following relations for $\ii$, $\jj \in \seq(\nu)$:
\begin{align}
  & 1_{\ii}1_{\jj} = \delta_{\ii,\jj}1_{\ii}, \\
  & x_{r}1_{\ii} = 1_{\ii} x_{r}, \\
  &\psi_{r}1_{\ii}  = 1_{s_r(\ii)} \psi_{r}, \\
  & x_{r} x_{t} = x_{t} x_{r}, \\
  & \psi_{r}\psi_{t} = \psi_{t}\psi_{r} \qquad \text{if $|r-t|>1$}, \\
  &\psi_{r}\psi_{r}1_{\ii} =
  \left\{\begin{array}{lll} \label{relation_quadratic}
    0 & \quad &\text{if $i_r = i_{r+1}$} \\
   1_{\ii} & & \text{if $(\alpha_{i_r},\alpha_{i_{r+1}}) = 0$} \\
   \left(x_{r}^{-\langle i_r,i_{r+1}\rangle} + x_{r+1}^{-\langle i_{r+1}, i_{r}\rangle} \right)1_{\ii}  & & \text{if $(\alpha_{i_r},\alpha_{i_{r+1}}) \neq 0$ and $i_r \neq i_{r+1}$,}
  \end{array} \right.
   \\
   \label{relation_cubic}
& \left( \psi_{r}\psi_{r+1}\psi_{r}
 -
 \psi_{r+1} \psi_{r} \psi_{r+1}\right)1_{\ii}= \nn \\
 & \qquad = \left\{
 \begin{array}{cll}
   \xsum{t=0}{ -\langle i_r,i_{r+1}\rangle -1}
   x_{r}^{t} x_{r+2}^{-\langle i_r,i_{r+1}\rangle-1-t} 1_{\ii}
 & \quad & \text{if $i_r= i_{r+2}$ and $(\alpha_{i_r},\alpha_{i_{r+1}})\neq 0$}\\
   0 & & \text{otherwise,}
 \end{array}
 \right. \\
   \label{relation_xpsi}
 & \left(\psi_{r} x_{t} - x_{s_r(t)}\psi_{r}\right)1_{\ii} =
 \left\{
 \begin{array}{cll}
   1_{\ii} & & \text{if $t=r$ and $i_r=i_{r+1}$}\\
   -1_{\ii} & &  \text{if $t=r+1$ and $i_r=i_{r+1}$} \\
   0 & & \text{otherwise.}
 \end{array}
 \right.
\end{align}

\begin{rem} \label{rem_basis}
For $\ii, \jj\in \rm{Seq}(\nu)$ let ${_{\jj}S_{\ii}}$ be the subset of $S_m$
consisting of permutations $w$ that take $\ii$ to $\jj$ via the standard action
of permutations on sequences, defined above.
Denote the subset $\{\hat{w}\}_{w\in {_{\jj}S_{\ii}}}$
of $1_{\jj}R1_{\ii}$ by ${_{\jj}\widehat{S}_{\ii}}$.
It was shown in \cite{KL,KL2} that the vector space $1_{\jj}\Rn{\nu}1_{\ii}$ has a basis consisting of elements of the form
\begin{equation}
  \{ \psi_{\hat{w}} \cdot x_{1}^{a_1} \cdots  x_{m}^{a_m}1_{\ii}
\mid \hat{w} \in{_{\jj}\widehat{S}_{\ii}}, \;\; a_{r} \in \Z_{\geq 0}\}.
\end{equation}
\end{rem}

Rouquier has defined a generalization of the algebras $R$, where the
relations depend on Hermitian matrices~\cite{Ro}. The results of this paper will extend to these algebras whenever the Hermitian matrices give rise to graded algebras $R$.

%
\subsubsection{The involution $\sigma$} \label{sec_involution}
%

Flipping a diagram about a vertical axis and simultaneously taking
$$
 \xy    (0,0)*{\dcross{i}{i}}; \endxy \quad {\rm to} \quad
  -\xy (0,0)*{\dcross{i}{i}}; \endxy
$$
(in other words, multiplying the diagram by $(-1)^s$ where $s$ is the number of
times equally labelled strands intersect) is an involution $\sigma = \sigma_\nu$ of $\Rn{\nu}$. Let $w_0$ denote the longest element of $S_{|\nu|}$. We can specify $\sigma$ algebraically as follows:
\begin{eqnarray}
  \sigma \maps \Rn{\nu} & \to & \Rn{\nu}
   \\ 1_{\ii} &\mapsto& 1_{w_0(\ii)} \nn \\
      x_{r} & \mapsto& x_{|\nu|+1-r} \nn \\
 \psi_{r} 1_{\ii} & \mapsto& (-1)^{ \delta_{i_r i_{r+1}}}  \psi_{|\nu|-r} 1_{w_0(\ii)}. \nn
\end{eqnarray}
Given an $\Rnu$-module  $M$, we let
$\sigma^{\ast}M$ denote the $\Rnu$-module whose underlying set is $M$ but with
twisted action $r \cdot u = \sigma(r) u$.

%
\subsubsection{Graded characters} \label{sec_char}
%

Define the graded character $\chr(M)$ of a graded finitely-generated $\Rn{\nu}$-module $M$
as
$$ \chr(M) = \sum_{\ii\in \seq(\nu)} \gdim (1_{\ii} M) \cdot \ii.$$
The character is an element of the free $\Z((q))$-module with the
basis $\seq(\nu)$; when $M$ is finite dimensional, $\chr(M)$ is an
element of the free $\Z[q,q^{-1}]$-module with basis $\seq(\nu)$.

%
\section{Functors on the module category} \label{sec_functors}
%

%
\subsection{Categories of graded modules}
%

We form the direct sum
$$ R = \bigoplus_{\nu\in \N[I]} \Rn{\nu}.$$
This is a non-unital ring.  However, $R$ is an idempotented ring with the elements $1_\nu \in \Rnu$ giving a system of mutually orthogonal idempotents. Observe that the appropriate notion of unital module $M$ for idempotented rings is the requirement that $M = \bigoplus_{\nu \in \N[I]} 1_{\nu}M$.

Let $\Rn{\nu}\dmod$ be the category of
finitely-generated graded left $\Rn{\nu}$-modules, $\Rn{\nu}\fmod$
be the category of finite dimensional graded $\Rn{\nu}$-modules, and
$\Rn{\nu}\pmod$ be the category of projective objects in
$\Rn{\nu}\dmod$. The morphisms in each of these three categories are
grading-preserving module homomorphisms.

By various categories of $R$-modules we will
mean direct sums of corresponding categories of $\Rn{\nu}$-modules:
\begin{eqnarray*}
R\dmod & \define &  \bigoplus_{\nu\in \N[I]} \Rn{\nu}\dmod , \\
R\fmod & \define & \bigoplus_{\nu\in \N[I]} \Rn{\nu}\fmod , \\
R\pmod & \define & \bigoplus_{\nu\in \N[I]} \Rn{\nu}\pmod.
\end{eqnarray*}
By a simple $R(\nu)$-module we mean a simple object in the category
$R(\nu)\dmod$. In this paper we will be primarily concerned with the
category of finite dimensional $R(\nu)$-modules. Note that this
category contains all of the simples.  Henceforth, by an
$R(\nu)$-module we will mean a finite dimensional graded
$R(\nu)$-module, unless we say otherwise. We will denote the zero
module by $\0$.

For any two $R(\nu)$-modules $M$, $N$ denote by $\Hom(M,N)$ or
$\Hom_{R(\nu)}(M,N)$ the $\Bbbk$-vector space of degree preserving
homomorphisms, and by $\Hom(M\{r\},N)= \Hom(M,N\{-r\})$ the space of
homogeneous homomorphisms of degree $r$. Here $N\{r\}$ denotes $N$
with the grading shifted up by $r$, so that $\chr(N\{r\}) = q^r
\chr(N)$.  Then we write
\begin{equation}
  \HOM(M,N) := \bigoplus_{r \in \Z} \Hom(M,N\{r\}),
\end{equation}
for the $\Z$-graded $\Bbbk$-vector space of all $R(\nu)$-module morphisms.

Though it is essential to work with the degree preserving morphisms
to get the $\Z[q,q^{-1}]$-module structure for the categorification
theorems in \cite{KL,KL2}, for our purposes it will often be
convenient to work with degree homogenous morphisms, but not
necessarily degree preserving, in the various categories of graded
modules introduced above.  Since any homogenous morphism can be
interpreted as a degree preserving morphism by shifting the grading
on the source or target, all results stated using homogeneous
morphisms can be recast as degree zero morphisms for an appropriate
shift on the source or target. For this reason, throughout the paper we define $M \cong N$ to mean there exists $r \in \Z$ such that $M$ is isomorphic to
$N\{r\}$ as graded modules, and all isomorphisms will implicitly mean isomorphic up to such a grading shift unless otherwise specified.

%
\subsection{Induction and Restriction functors}
\label{sec_indres}
%

There is an inclusion of graded algebras
$$\iota_{\nu,\nu'}\ : \  \Rn{\nu}\otimes \Rn{\nu'} \hookrightarrow \Rn{\nu+\nu'}$$
given graphically by putting the diagrams next to each other. It takes the
idempotent $1_{\ii}\otimes 1_{\jj}$ to $1_{\ii\jj}$ and the unit
element $1_{\nu}\otimes 1_{\nu'}$ to an idempotent of $R(\nu+\nu')$ denoted $1_{\nu,\nu'}$. This inclusion gives rise to restriction and induction functors denoted by $\Res_{\nu,\nu'}$ and $\Ind_{\nu,\nu'}$, respectively.  When it is clear from the context, or when no confusion is likely to arise, we often simplify notation and write $\Res$ and $\Ind$.

We can also consider these notions for any tuple $\unu =
(\nu^{(1)},\nu^{(2)},\dots, \nu^{(k)})$ and sometimes refer to the
image $\Rn{\unu} \define {\mathrm {Im}}\, \iota_{\unu}  \subseteq
\Rn{\nu^{(1)}+\cdots+ \nu^{(k)}}$ as a parabolic subalgebra. This
subalgebra has identity $1_{\unu}$. Let $\mu = \nu^{(1)}+\cdots+
\nu^{(k)}$, $m = \sum_r | \nu^{(r)}|$, and $P= P_{\unu}$ be the
composition $(|\nu^{(1)}|, \ldots, |\nu^{(k)}|)$ of $m$ so that
$S_P$ is the corresponding parabolic subgroup of $S_m$. It follows
from Remark~\ref{rem_basis} that $\Rn{\mu} 1_{\unu}$ is a free right
$\Rn{\unu}$-module with basis $\{ \psi_{\hat{w}} 1_{\unu} \mid w \in
S_m / S_P \}$ and $1_{\unu} \Rn{\mu}$ is a free left
$\Rn{\unu}$-module with basis
 $\{ 1_{\unu} \psi_{\hat{w}}  \mid w \in  S_P \backslash  S_m  \}$.
By abuse of notation we will write $S_m / S_P$ to denote the
minimal length left coset representatives, i.e$.$ $\{ w \in S_m \mid
\ell(wv) = \ell(w) + \ell(v) \, \forall \, v \in S_P \}$,
and $S_P \backslash   S_m$ for the minimal length right coset representatives.

\begin{rem}\label{rem_indbasis}
It is easy to see that if $M$ is an $\Rn{\unu}$-module with basis
${\mathcal U}$ consisting of weight vectors, then $\{ \psi_{\hat{w}}
\otimes u  \mid u \in {\mathcal U},  w \in S_m / S_P \}$ is a weight
basis of $\Ind_{\unu} M \define \Rn{\mu} \otimes_{\Rn{\unu}} M $
(where for each $w$ we fix just one reduced expression $\hat{w}$).
Note $\Rn{\mu} \otimes_{\Rn{\unu}} M = \Rn{\mu} 1_{\unu}
\otimes_{\Rn{\unu}} M$ since $\psi_{\hat{w}} 1_{\unu} \otimes u =
\psi_{\hat{w}} \otimes 1_{\unu} u = \psi_{\hat{w}} \otimes u$.

Likewise, $\Indc M \define \HOM_{\Rn{\unu}}(\Rn{\mu}, M)$, which is
discussed in detail in Section~\ref{sec_coind} below,  and has basis $\{
f_{{{w}}, u}  \mid u \in {\mathcal U},  w \in S_P \backslash   S_m
\}$ where $f_{{{w}}, u} ( h \psi_{\hat{v}}) = h u \, \delta_{w,v}$
for $h \in \Rn{\unu}$ and $v \in S_P \backslash   S_m$.
Note $\Hom_{\Rn{\unu}}(\Rn{\mu}, M) = \Hom_{\Rn{\unu}}(1_{\unu}
\Rn{\mu}, M)$ since for $f \in \Hom_{\Rn{\unu}}(1_{\unu} \Rn{\mu},
M)$, $t \in \Rn{\mu}$, if $1_{\ii} \not\in \Rn{\unu}$, i.e.
$1_{\unu} 1_{\ii} = 0$, then
$$f(1_{\ii} t) = 1_{\unu} f(1_{\ii} t) = f(1_{\unu} 1_{\ii} t) = f(0) = 0.$$
In other words, we can extend the domain of $f$ to $\Rn{\mu}$ by
setting $f$ to be $0$ on $1_{\ii} \Rn{\mu}$ when $1_{\ii} \not\in
\Rn{\unu}$. Likewise any $f \in \Hom_{\Rn{\unu}}(\Rn{\mu}, M)$ must
be 0 on the above set.
\end{rem}

One extremely important property of the functor $\Ind_{\unu} - \,
\define \Rn{\mu} \otimes_{\Rn{\unu}} -$ is that it is left adjoint
to restriction. In other words, there is a functorial isomorphism
\begin{equation} \label{frob}
\HOM_{\Rn{\mu}} (\Ind_{\unu}A, B) \cong \HOM_{\Rn{\unu}}(A,
\Res_{\unu} B)
\end{equation}
where $A$, $B$ are finite dimensional $\Rn{\unu}$- and
$\Rn{\mu}$-modules, respectively. This property is called
Frobenius reciprocity and we use it repeatedly, often for deducing information about characters.

A shuffle $\kk$ of a pair of sequences $\ii\in \seq(\nu)$, $\jj\in
\seq(\nu')$ is a sequence together with a choice of subsequence
isomorphic to $\ii$ such that $\jj$ is the complementary
subsequence. Shuffles of $\ii,\jj$ are in a bijection with the
minimal length left coset representatives of $S_{|\nu|}\times
S_{|\nu'|}$ in $S_{|\nu|+|\nu'|}$. We denote by $\deg(\ii,\jj,\kk)$
the degree of the diagram in $R(\nu+\nu')$ naturally associated to
the shuffle, see an example below.
\[
 \xy
 (0,0)*{\includegraphics[scale=0.5]{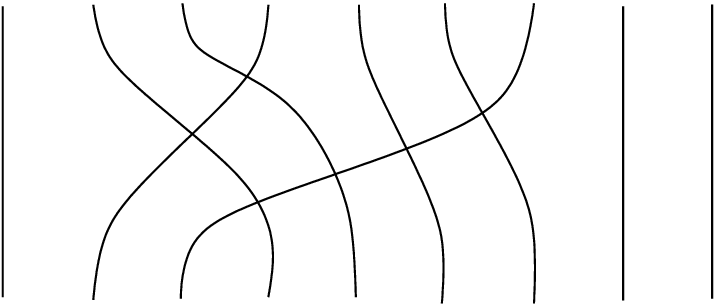}};
 (0,17)*{\overbrace{\hspace{2.5in}}};
 (-23,-17)*{\underbrace{\hspace{.75in}}};
 (11,-17)*{\underbrace{\hspace{1.6in}}};
 (0,21)*{\kk};(-23,-21)*{\ii};(11,-21)*{\jj};
 \endxy
\]
When the meaning is clear, we will also denote by $\kk$ the underlying sequence
of the shuffle $\kk$.

Given two functions $f$ and $g$ on sets $\seq(\nu)$ and $\seq(\nu')$, respectively,
with values in some commutative ring which contains $\Z[q,q^{-1}]$, we
define their (quantum) shuffle
product $f\shuffle g$ (see~\cite{Lec} and references therein)
as the function on $\seq(\nu+\nu')$ given by
$$ (f\shuffle g)(\kk) = \sum_{\ii,\jj} q^{\deg(\ii,\jj,\kk)} f(\ii) g(\jj),$$
the sum is over all ways to represent $\kk$ as a shuffle of $\ii$
and $\jj$. Given $M\in \Rn{\nu}\dmod$ and $N\in \Rn{\nu'}\dmod$ we
construct the $\Rnu \otimes \Rn{\nu'}$-module denoted by $M
\boxtimes N$  in the obvious way. It was shown in \cite{KL} that
$$\chr(\Ind_{\nu,\nu'}(M\boxtimes N) ) = \chr(M) \shuffle \chr(N).$$

A similar statement holds for characters of induced
$\Rn{\unu}$-modules by the transitivity of induction. This statement
can be seen as a special case of the Mackey formula which describes
a filtration on the restriction of an induced module (from one
parabolic to another).

More precisely, in the case of maximal parabolics, the Mackey
formula says the  graded $(\Rn{\nu}\otimes R(\nu'), R(\nu'')\otimes
R(\nu'''))$-bimodule ${1}_{\nu,\nu'}R 1_{\nu'',\nu'''}$ has a
filtration over all $\lambda\in \N[I]$ with subquotients
isomorphic to
the graded bimodules
$$  \big( {1}_{\nu}R 1_{\nu-\lambda,\lambda} \otimes
  {1}_{\nu'}R 1_{\nu'+\lambda-\nu''',\nu'''-\lambda} \big)
 \otimes_{R'}
 \big( {1}_{\nu-\lambda,\nu''+\lambda-\nu}R 1_{\nu''} \otimes
  {1}_{\lambda, \nu'''-\lambda}R 1_{\nu'''}\big) \{(-\lambda, \nu'+\lambda-\nu''')\},
$$
where $R'=R(\nu-\lambda)\otimes R(\lambda)\otimes
R(\nu'+\lambda-\nu''')\otimes R(\nu'''-\lambda)$, the bilinear form
$(\, , \, )$  is defined in Section~\ref{sec_Cartan}, and such that
every term above is in $\N[I]$. There is a natural generalization of
this statement to arbitrary parabolic subalgebras.

%
\subsection{Co-induction} \label{sec_coind}
%

In this section, we examine the right adjoint to restriction,
the co-induction functor denoted $\Indc$, and discuss the relationship between $\Ind$ and $\Indc$,
following the work of \cite{Vaz1}. Using the notation of the previous section, set $\Indc_{\Rn{\unu}} - := \HOM_{\Rn{\unu}}(\Rn{\mu},-)$ endowed with the
module structure $(r \odot f)(t)=f(tr)$ for $r,t\in \Rn{\mu}$, $f \in \Indc_{\Rn{\unu}}-$.
Now there is a functorial isomorphism
\begin{equation} \label{right_adjoint}
\HOM_{\Rn{\mu}} (B, \Indc_{\unu}A) \cong
\HOM_{\Rn{\unu}}(\Res_{\unu} B, A)
\end{equation}
where $A$, $B$ are finite dimensional modules.

Just as $w_0$ denotes the longest element of $S_m$, let $w_P \in S_P$ denote  the
longest element of the parabolic subgroup, with notation as above.
Let $y=w_Pw_0$ in the discussion below.  Note that $y$ is a minimal length right
coset representative for $S_P \backslash   S_m$ and corresponds to
 the ``longest shuffle".

Observe that for any $r$ such that $s_r \in S_P$,
$\ell(w_Ps_rw_P)=1=\ell(w_0s_rw_0)$  and further $$\ell(s_r y) = 1 +
\ell(y) = \ell(w_Ps_rw_Py)=\ell(yw_0s_rw_0)$$ as in fact
$$(w_Ps_rw_P)y= w_P s_r w_P w_P w_0=w_Pw_0w_0s_rw_0=y(w_0s_rw_0).$$

 Set
\begin{eqnarray}
 \sigma_{\unu} &:=& \sigma_{\nu^{(1)}} \otimes \sigma_{\nu^{(2)}} \otimes \dots \otimes \sigma_{\nu^{(k)}}
\end{eqnarray}
where $\sigma_{\nu} \maps R(\nu) \to R(\nu)$ is the involution defined in Section~\ref{sec_involution}.

When clear from context, let us just call $\sigma=\sigma_{\mu}$.  Then note,
$\sigma(1_{\jj})=1_{w_0(\jj)}$,
 $\sigma(x_{r})=x_{w_0(r)}$,
$\sigma(\psi_{r} 1_{\jj})=(-1)^{\delta_{j_r j_{r+1}}} \psi_{w_0s_rw_0 } 1_{w_0(\jj)}$ with similar equations
for $\sigma_{\unu}$, where $S_m$ acts on $\seq(\mu)$ in the usual fashion $w(i_1,\dots, i_m) = (i_{w^{-1}(1)}, \dots, i_{w^{-1}(m)})$.
%
In what follows, for bookkeeping purposes, we will write $u \in M$, but $\bar{u} \in
\sigma^{\ast}M$ so that the $\sigma$-twisted action can be described as $r \bar{u} = \overline{\sigma(r) u}$.

\begin{thm} \label{thm_coind}  \hfill
\begin{enumerate}
\item  \label{thm_coindp1} Let $M$ be a finite dimensional $\Rn{\unu}$-module. Then
$$\Ind_{\unu}^{\mu} M  \cong
\sigma^*_{\mu}(\Indc_{\unu}^{\mu}(\sigma^*_{\unu}M))\{\deg(y)\}$$
as graded modules.
\item \label{thm_coindp2}
Let $A$ be a finite dimensional $\Rn{\nu}$-module and $B$ a finite
dimensional $\Rn{\eta}$-module. Then there is an isomorphism $$\Ind_{\nu,\eta}^{\nu+\eta} A
\boxtimes B \cong \Indc_{\eta,\nu}^{\eta+\nu} B \boxtimes A.$$
\end{enumerate}
\end{thm}

\begin{proof}
We first note that \eqref{thm_coindp2} follows from a special case
of \eqref{thm_coindp1}. The appropriate degree shift to make it an isomorphism of graded modules
is thus $-(\eta, \nu)$.  To prove \eqref{thm_coindp1}, we first
construct a $\Rn{\unu}$-module map
\begin{equation}
  \xymatrix{M \ar[r]^-F &   \Res_{\unu}^{\mu}(\sigma_{\mu}^* \Indc_{\unu}^{\mu}(\sigma_{\unu}^*M))}
\end{equation}
with $\deg(F)=-\deg(y)$ and then the induced map
\begin{equation}
  \xymatrix{ \Ind_{\unu}^{\mu}M \ar[r]^-{\cal{F}} &   \sigma_{\mu}^{\ast}\Indc_{\unu}^{\mu}
(\sigma_{\unu}^*M))}
\end{equation}
also has $\deg(\cal{F})=-\deg(y)$ and surjective as the image of $F$
generates the target over $\Rn{\mu}$.  Since the two modules in
question have the same dimension, they are isomorphic.

Given $u \in M$
define $f_u \in \HOM_{\Rn{\unu}}(\Rn{\mu}, \sigma_{\unu}^{\ast}M)$ by
\begin{equation}
  f_u( \psi_{\hat{w}}) =
   \bar{u} \delta_{w,y}
\end{equation}
where $w \in {S_P}\backslash S_m$ ranges over the minimal length
right coset representatives, $\hat{w}$ is a fixed reduced
expression, and $y  = w_P w_0$.
Observe that $\deg(f_u)=\deg(u)-\deg(y)$. We extend $f_u$ to an
$\Rn{\unu}$-map by declaring $f_u(h\psi_{\hat{w}})=h
f_u(\psi_{\hat{w}})$ for $h \in \Rn{\unu}$ which is viable by
Remark~\ref{rem_indbasis}. Now we define
\begin{eqnarray}
  F \maps M &\to& \sigma_{\mu}^{\ast} \Indc_{\unu}^{\mu}(\sigma_{\unu}^*M) \nn \\
 u &\mapsto& \overline{f_u}
\end{eqnarray}
and check it is an $\Rn{\unu}$-map.  This map is homogeneous with
$\deg(F)=-\deg(y)$. Note that $f_{u + u'} = f_u + f_{u'}$ so it
suffices to consider only degree homogeneous weight vectors $u \in
M$, i.e.\ there exists $\ii$ such that $1_{{\ii}} \bar{u} = \bar{u}$
(and so $1_{w_P(\ii)} u=u$). In this case $f_u( 1_{\jj}
\psi_{\hat{w}}) = \bar{u} \delta_{w,y} \delta_{\ii,\jj}$, and this
holds regardless of whether $1_{\jj} \in \Rn{\unu}$ by
Remark~\ref{rem_indbasis}. In fact, by abuse of notation, we may
write $1_{\jj} \bar{u} = \bar{u} \delta_{\ii,\jj}$ even when
$1_{\jj} \not\in \Rn{\unu}$.

The following three computations show that $F(hu)= h \odot F(u)$ for $h=1_{\jj}$, $h=x_{r}$ for
all $r$, and $h=\psi_{r}1_{\jj}$ for $r$ such that $s_r \in S_P$  and $\jj$ such that $1_{\jj} \in R(\unu)$. These computations show that that $F$ is an $\Rn{\unu}$-map. In these computations note that with respect to $\psi_{\hat{w}}$, by lower terms we mean
elements of $ \{ h \psi_{\hat{v}} \mid h \in \Rn{\unu}, \; \ell(v) <
\ell(w) \}$. From now on, assume $u$ is a weight vector as above.

Case 1) We evaluate
\begin{eqnarray}
  (1_{\jj}F(u))(\psi_{\hat{w}}) &=& 1_{\jj} \odot \overline{f_u}(\psi_{\hat{w}}) = \overline{\sigma_{\mu}(1_{\jj}) \odot f_u}(\psi_{\hat{w}}) \nn \\
 &=& \overline{f_u}(\psi_{\hat{w}}1_{w_0(\jj)}) = \overline{f_u}(1_{ww_0(\jj)}\psi_{\hat{w}}) \nn \\
 &=&  \bar{u} \delta_{w,y} \delta_{\ii, ww_0(\jj)}
= \bar{u} \delta_{w,y} \delta_{\ii, yw_0(\jj)}
\nn \\
 &=&  \bar{u} \delta_{w,y} \delta_{\ii, w_P(\jj)}
 =   1_{w_P(\jj)} \bar{u} \delta_{w,y}
\nn \\
&=& \sigma_{\unu}(1_{\jj}) \bar{u} \delta_{w,y}
  = \overline{1_{\jj}u}  \delta_{w,y}
\nn \\
&=&
\overline{f_{{1_{\jj}u}}}(\psi_{\hat{w}}) = F(1_{\jj}u)(\psi_{\hat{w}})
\end{eqnarray}
so that $1_{\jj}F(u) = F(1_{\jj}u)$.

Case 2)  We compute
\begin{eqnarray}
  (x_{r} F(u))(\psi_{\hat{w}}) &=& (x_{r} \odot \overline{f_u})(\psi_{\hat{w}}) = \overline{\sigma_{\mu}(x_{r})\odot f_u}(\psi_{\hat{w}}) \nn \\
&=& \overline{f_u}(\psi_{\hat{w}} x_{w_0(r)}) \nn \\
&=& \overline{f_u}(x_{ww_0(r)}\psi_{\hat{w}} + \text{lower terms} )\nn \\
&=&
\left\{
\begin{array}{ccl}
  \overline{f_u}(x_{w_P(r)}\psi_{\hat{y}}) & \quad & \text{if $w=y$}\\
  0 & \quad & \text{else}
\end{array} \right. \nn \\
&=&
\left\{
\begin{array}{ccl}
  x_{w_P(r)}\bar{u} & \quad & \text{if $w=y$ } \\
  0 & \quad & \text{else}
\end{array} \right. \nn \\
&=&\left\{
\begin{array}{ccl}
\overline{x_{r}u} & \quad & \text{if $w=y$} \nn \\
  0 & \quad & \text{else}
\end{array} \right. \nn \\
&=& \overline{f_{x_{r}u}}(\psi_{\hat{w}}) = F(x_{r}u)(\psi_{\hat{w}})
\end{eqnarray}
so that $F(x_{r}u) = x_{r}F(u)$ for any $r$.

Case 3)
Let $r$ be such that $s_r \in S_P$, and $\jj$ be such that $\psi_r 1_{\jj} \in R(\unu)$.  Recall that then $w_Ps_rw_P \in S_P$ as well, and furthermore $\sigma_{\unu}(\psi_r 1_{\jj}) = \psi_{w_Ps_rw_P} 1_{w_P(\jj)} \in R(\unu)$.  We compute
\begin{eqnarray}
  \psi_{r} 1_{\jj} F(u)(\psi_{\hat{w}})
        &=& (\psi_{r} 1_{\jj} \odot \overline{f_u}) (\psi_{\hat{w}}) \nn \\
 &=& \overline{f_u} (\psi_{\hat{w}} \sigma_{\mu} (\psi_{r} 1_{\jj}))
        = \overline{f_u}(\psi_{\hat{w}} (-1)^{\delta_{{\jj}_r, {\jj}_{r+1}}}
        \psi_{w_0s_rw_0} 1_{w_0({\jj})} )
\nn \\
 &=&
\left\{ \begin{array}{ccl}
  (-1)^{\delta_{{\jj}_r, {\jj}_{r+1}}} \overline{f_u}(
\bigl( \psi_{w_Ps_rw_P}\psi_{\hat{y}} + \text{lower terms} \bigr) 1_{w_0({\jj})} ) &
    \quad & \text{if $w=y$} \\
(-1)^{\delta_{{\jj}_r, {\jj}_{r+1}}}
  \overline{f_u}((\text{lower terms}) 1_{w_0({\jj})}) & \quad & \text{if $w\neq y$}
\end{array} \right. \nn \\
&=&
\left\{\begin{array}{ccl}
(-1)^{\delta_{{\jj}_r, {\jj}_{r+1}}}
\overline{f_u}(\psi_{w_Ps_rw_P} 1_{y w_0({\jj})}\psi_{\hat{y}}) &\quad & \text{if $w=y$} \\
  0 & \quad & \text{else}
\end{array} \right. \nn \\
&=&
\left\{\begin{array}{ccl}
(-1)^{\delta_{{\jj}_r, {\jj}_{r+1}}}
  \psi_{w_Ps_rw_P} 1_{w_P({\jj})}\overline{f_u}(\psi_{\hat{y}}) & \quad & \text{if $w=y$} \\
  0 & \quad & \text{else}
\end{array} \right. \nn \\
&=&
\left\{\begin{array}{ccl}
 \sigma_{\unu} (\psi_{r} 1_{\jj})\bar{u}& \quad & \text{if $w=y$}\\
  0 & \quad & \text{else}
\end{array} \right. \nn \\
&=&
\left\{\begin{array}{ccl}
 \overline{\psi_{r} 1_{\jj} u}& \quad & \text{if $w=y$} \\
  0 & \quad & \text{else}
\end{array} \right. \nn \\
&=& \overline{f_{\psi_{r} 1_{\jj} u}}(\psi_{\hat{w}}) \nn \\
&=& F(\psi_{r} 1_{\jj} u)(\psi_{\hat{w}}),
\end{eqnarray}
so that $\psi_{r} 1_{\jj} F(u)=F(\psi_{r} 1_{\jj} u)$.

Note the image of $F$ contains all of
the  $\overline{f_u}$ as $u$ ranges over a weight basis of $  M$.  Hence the image of $\cal{F}\maps \Ind_{\unu}^{\mu}M \to \sigma_{\mu}^* \Indc_{\unu}^{\mu}(\sigma_{\unu}^*M)$ contains all of the $h \odot \overline{f_u}$ for $h \in \Rn{\mu}$. We shall argue this contains a basis of $\sigma_{\mu}^*
\Indc_{\unu}^{\mu}\sigma_{\unu}^*M$ which will show that $\cal{F}$
is surjective. Recall from Remark~\ref{rem_indbasis} that
$\sigma_{\mu}^* \Indc_{\unu}^{\mu}(\sigma_{\unu}^*M)$ has a basis of
``bump functions" of the form  $\overline{f_{{w},u}} $ and in this
notation $\overline{f_{u}} =\overline{f_{{y},u}}$. As in
\cite{Vaz1}, we can show the $\psi_{\hat{v}}\odot
\overline{f_{{y},u}}$ for appropriate $v$ are triangular with
respect to the $\{\overline{f_{{w},u'}} \}$ so contain a basis.
Since the dimensions of the induced and coinduced modules are the
same,
 $\cal{F}$ is in fact an isomorphism.
\end{proof}

%
\subsection[Certain simple modules]{Simple $R(mi)$-modules} \label{sec_Lim}
%

Simple modules for the algebra $R(mi)$
play a key role in this paper.  There are several constructions of these modules.

Throughout this section let $\ii=i^m$.  Consider the graded algebra
$\Bbbk[x_{1,\ii},\dots, x_{m,\ii}]$  with
$\deg(x_{t,\ii})=(\alpha_i,\alpha_i)$. Up to isomorphism
and grading shift, there is a unique graded irreducible
module $L(i^m)$ for the ring $\Rn{mi}$ given as the quotient of
$\Bbbk[x_{1,\ii},\dots, x_{m,\ii}]$ by the ideal generated by homogeneous symmetric polynomials with positive degree, see \cite[Section 2.2]{KL}. This module can alternatively be described as the induced module from the trivial $R'$-module, where
$R'$ is the subalgebra of $R(mi)$ generated by $\psi_{1,\ii}, \dots,
\psi_{m-1,\ii}$ and symmetric polynomials in $\Bbbk[x_{1,\ii},\dots,
x_{m,\ii}]$.
Note the trivial $R'$-module is its unique 1-dimensional module, on which all
$\psi_{r, \ii}$ and $\sum_{r=1}^m x_{r,\ii}^k$  act as $0$, where  $1 \le r \le m$ and $k \geq 1$.

Furthermore, this irreducible module $L(i^m)$
 is isomorphic
to the module
induced from the one-dimensional graded module $L$ $= L(i) \boxtimes
\cdots \boxtimes L(i)$ over $\Bbbk[x_{1,\ii},\dots, x_{m,\ii}]$ on
which $x_{1,\ii}, \dots, x_{m,\ii}$ all act trivially. In this paper
we fix the grading shift on this unique simple module $L(i^m)\{r\}$
so that
\begin{equation}
  \chr(L(i^m)) = [m]_i^! i^m.
\end{equation}

In \cite[Proposition 2.8]{Lau3}, it is not only shown  that
 for any $u \in L(i^m)$, $1 \leq r \leq m$, and $k \geq m$
that $ x_{r}^k u =0$, but
also that there exists $\tilde{u} \in L(i^m)$ such that $x_r^{m-1} \tilde{u}
\neq 0$ for all $r$.

See the third statement in Section~\ref{sec_properties} for some of
the important properties of $L(i^m)$, such as its behaviour under
the induction and restriction functors.

%
\subsection{Refining the restriction functor}
\label{sec_define_ei}
%

For $M$ in $\Rn{\nu}\dmod$ and $i\in I$ let
$$\Delta_{i}M= (1_{\nu-i}\otimes 1_i)M = \Res_{\nu-i,i} M,$$
and, more generally,
$$ \Delta_{i^n}M =  (1_{\nu-ni}\otimes 1_{ni})M = \Res_{\nu-ni,ni} M.$$
We view $\Delta_{i^n}$ as a functor into the category $R(\nu-ni)\otimes R(ni)\dmod$.
By Frobenius reciprocity, there are functorial isomorphisms
\begin{equation} \label{eq_funs}
\HOM_{\Rn{\nu}} (\Ind_{\nu-ni,ni}N\boxtimes L(i^n), M) \cong
\HOM_{R(\nu-ni)\otimes R(ni)}(N\boxtimes L(i^n), \Delta_{i^n}M),
\end{equation}
for $M$ as above and $N\in R(\nu-ni)\dmod$.

Define
\begin{equation}\label{def_ei}
 e_i \; := \Res_{\nu-i}^{\nu-i,i} \circ \Delta_i \maps \Rn{\nu}\fmod \to \Rn{\nu-i}\fmod
\end{equation}
and for $M \in Rnu\fmod$, set
\begin{align}
  & \et{i}M := \soc e_i M, \\
  & \ft{i}M := \cosoc \Ind_{\nu,i}^{\nu+i} M \boxtimes L(i), \\
  & \epsilon_i(M) := \max\{n \geq 0 \mid \et{i}^n M \neq \0 \}.
\end{align}
We also define their so-called $\sigma$-symmetric versions, which are indicated
with a $\vee$.
Note that $ \sigma^{\ast}(\Delta_{i}(\sigma^{\ast}M)) = \Res_{i,\nu-i} M$.
Set
\begin{align}
 &e_i^\vee
 := \Res_{\nu-i}^{i,\nu-i} \circ \Res_{i,\nu-i} \maps \Rn{\nu}\fmod \to \Rn{\nu-i}\fmod, \\
  & \ets{i}M := \sigma^{\ast}(\et{i}(\sigma^{\ast}M)) = \soc e_i^\vee M, \\
  & \fts{i}M := \sigma^{\ast}(\ft{i}(\sigma^{\ast}M))
= \cosoc \Ind_{i, \nu}^{\nu+i} L(i) \boxtimes M, \\
  & \eph(M) := \epsilon_i(\sigma^{\ast}M) 
= \max\{m \geq 0 \mid (\ets{i})^mM \neq \0 \}.
\end{align}
Observe that the functors $e_i$ and $e_i^\vee$ are exact.
Although the functors $\et{i}$ and $\ft{i}$ can be defined on any 
module,
in this paper we will only apply them to simple modules. It is a
theorem of \cite{KL} that if $M$ is irreducible, so are $\ft{i}M$
and $\et{i}M$ (as long as the latter is nonzero), and likewise for
$\fts{i}M$ and $\ets{i}M$. This is stated below along with other key
properties.

%
\subsubsection[Key properties]{Properties of the functors
$\et{i}$ and $\ft{i}$ on simple modules}
\label{sec_properties}
%

In this section we give a long list of results that were proved in
\cite{KL} about simple $\Rnu$-modules  and their behaviour under
induction and restriction. They extend to the symmetrizable case by the results in \cite{KL2}. We will use them freely throughout the
paper.

\begin{enumerate}
  \item $$\chr(\Delta_{i^n}M)= \sum_{\jj\in \seq(\nu-ni)}  \gdim(1_{{\jj} i^n}M) \cdot \jj, $$
where we view $\Delta_{i^n}M$ as a module over the subalgebra $R(\nu-ni)$
of $R(\nu-ni)\otimes R(ni)$.
 \item
Let $N\in \Rn{\nu}\dmod$ be irreducible and $M=\Ind_{\nu,ni}N\boxtimes L(i^n)$.
Let $\epsilon = \ep{N}$.
\begin{enumerate}
\item  $\Delta_{i^{\epsilon+n}}M\cong \et{i}^{\epsilon} N\boxtimes
L(i^{\epsilon + n})$.
\item
$\cosoc M$ is irreducible, and $\cosoc M \iso
\ft{i}^n N$, $\Delta_{i^{\epsilon+n}} \ft{i}^n N \cong
\et{i}^{\epsilon} N\boxtimes L(i^{\epsilon + n})$, and
$\varepsilon_i(\ft{i}^n N)= \varepsilon+n$.
\item
All other composition factors $L$ of $M$ have
$\varepsilon_i(L)<\varepsilon+n$.
\item  $\ft{i}^n N$ occurs with multiplicity one as a composition factor of
$M$.
\end{enumerate}
 \item
Let $\underline{\mu} = (i^{\mu_1}, \cdots, i^{\mu_r})$ with $\sum_{k=1}^r \mu_k = n$.
\begin{enumerate}
\item The module $L(i^n)$ over the algebra $\Rn{ni}$ is the only graded
irreducible module, up to isomorphism.
\item All composition factors of $\Res_{\underline{\mu}}L(i^n)$ are isomorphic
to $L(i^{\mu_1})\boxtimes \dots \boxtimes L(i^{\mu_r})$,
and
$\soc(\Res_{\underline{\mu}} L(i^n))$ is irreducible.
\item $\et{i} L(i^n) \cong L(i^{n-1})$.
\end{enumerate}
 \item Let $M\in \Rn{\nu}\dmod$ be irreducible with $\varepsilon_i(M)>0$.
Then $\et{i}M = {\soc}(e_i M)$ is irreducible and
$\varepsilon_i( \et{i} M)=\varepsilon_i(M)-1$.
Socles of $e_iM$ are pairwise nonisomorphic for different $i\in I$.
%
 \item  For irreducible $M\in \Rn{\nu}\dmod$ let $m=\epsilon_i(M)$. Then the socle of $e_i^m M$ is
isomorphic to $\widetilde{e}^m_iM^{\oplus [m]_i^!}$.
 \item For irreducible modules $M\in \Rn{\nu}\dmod$ and $N\in \Rn{\nu+i}\dmod$
we have $\widetilde{f}_i M \cong N$ if and only if $\widetilde{e}_i N\cong M$.
 \item Let $M,N\in \Rn{\nu}\dmod$ be irreducible. Then
$\widetilde{f}_iM \cong \widetilde{f}_iN$ if and only if $M\cong N$.
Assuming $\varepsilon_i(M), \varepsilon_i(N)>0$, $\widetilde{e}_i M \cong \widetilde{e}_i N$
if and only if $M\cong N$.
\end{enumerate}

%
\subsection{The algebras $\RnuL$}
\label{sec_RnuL}
%

For $\Lambda = \sum_{i \in I} \lambda_i \Lambda_i \in P^+$ consider the two-sided ideal $\cal{J}^{\Lambda}_{\nu}$ of $\Rn{\nu}$ generated by elements
$x_{1,\ii}^{\lambda_{i_1}}$ over all sequences $\ii \in \seq(\nu)$.
We sometimes write $\cal{J}^{\Lambda}_{\nu}=\cal{J}^{\Lambda}$ when no confusion is likely to arise. Define
\begin{equation} \label{eq_def_RLambda}
\RnuL := \Rn{\nu} /\cal{J}^{\Lambda}_{\nu}
\end{equation}
By analogy with the Ariki-Koike cyclotomic quotient of the affine
Hecke algebra~\cite{AK} (see also \cite{Ari3}) this algebra is
called the cyclotomic quotient at weight $\Lambda$ of $\Rn{\nu}$. As
above we form the non-unital ring
\begin{equation}
 R^{\Lambda} = \bigoplus_{\nu \in \N[I]} \RnuL.
\end{equation}

In type $A$ the following Proposition is essentially contained in \cite[Section 2.2]{BK1}.  Here we give the natural extension to arbitrary type.

\begin{prop} \hfill
\begin{enumerate}
  \item
For all $\ii \in \seq(\nu)$ and any $\Lambda \in P^+$ the elements $x_{r,\ii}$
are nilpotent for all $1 \leq r \leq |\nu|$.
  \item The algebra $\RnuL$ is finite dimensional.
\end{enumerate}
\end{prop}

\begin{proof}
The proof of the first claim is by induction on the length of the
sequence $\ii$.  The base case is immediate from the definition
\eqref{eq_def_RLambda} of $\RnuL$.  For the induction step we assume
the claim holds for all sequences $\ii$ of length of $m-1$ and prove
the result for sequences of length $m$. For $j \in I$ write $\ii j$
for the concatenated sequence of length $m$ obtained by adding $j$
to the end of $\ii$. Restricting  $R_{\ii j} \to R1_{\ii}\otimes
R1_{j}$ implies $x_{r,\ii j}$ is nilpotent for all $1 \leq r < m$.
Thus it suffices to prove that $x_{m,\ii j}$ is nilpotent.

Let $\ii = i_1 \dots i_{m-1}$ and assume for simplicity of notation
that $i_{m-1}=i$ for some $i \in I$. If $(\alpha_{i},\alpha_j) <0$,
then by the inductive hypothesis for some $N,N'>0$ we have
$x_{m-1,\ii}^{N}=0$ and $x_{m-1,\ii'}^{N'}=0$, where $\ii'=i_1\dots
i_{m-3}i_{m-2}j$, so that adding $j$ to the end of the sequence
$\ii$ we have $x_{m-1,\ii j}^{N}=0$, and adding an $i$ to the end of
the sequence $\ii'$ we have $x_{m-1,s_{m-1}(\ii j)}^{N'}=0$. Since
$x_{m-1,\ii}^{N}=0$ we certainly have $x_{m-1,\ii}^{N a_{ij}}=0$.
Then $x_{m,\ii j}^{N a_{ji}+N'}=0$, since by \eqref{new_eq_UUzero}
(working locally around the last two strands) we have
\begin{equation}
   \vcenter{\xy  (8,0)*{\updot{j}};(0,0)*{\up{i}};
   (16,2)*{\scs N a _{ji}+N'};
   \endxy} =
     \xy   (0,0)*{\twocross{i}{j}};
     (2,9)*{\bullet}+(10,-1)*{\scs (N-1) a _{ji}+N'};; \endxy
 \;\; - \;\;
  \vcenter{ \xy
   (0,0)*{\updot{i}};
   (8,0)*{\updot{j}};
   (-3.5,2)*{\scs a_{ij}};
   (18,2)*{\scs (N-1) a _{ji}+N'};
  \endxy}
\end{equation}
The first diagram on the right-hand-side is zero since we can slide
the dots using \eqref{new_eq_ijslide} and then use our assumption
$x_{m-1,s_{m-1}(\ii j)}^{N'}=0$.
Then either $N=1$ and
the second diagram is also zero by assumption, or $N>1$ and we can
iterate $N$-times the application of \eqref{new_eq_UUzero}  to show
\begin{equation}
   \vcenter{\xy  (8,0)*{\updot{j}};(0,0)*{\up{i}};
   (16,2)*{\scs N a _{ji}+N'};
   \endxy} =
     \xy   (0,0)*{\twocross{i}{j}};
     (2,9)*{\bullet}+(4,-1)*{\scs N'};
     (-2,9)*{\bullet}+(-7,-1)*{\scs (N-1) a _{ij}}; \endxy
 \;\; - \;\; (-1)^{N} \;\;
  \vcenter{ \xy
   (0,0)*{\updot{i}};
   (8,0)*{\updot{j}};
   (-5,2)*{\scs  Na_{ij}};
   (14,2)*{\scs N'};
  \endxy}
\end{equation}
Sliding the $N'$ dots on the first diagram on the right-hand-side
using \eqref{new_eq_ijslide}, the entire right side is zero by our
assumptions on $N$, $N'$.

If $(\alpha_i,\alpha_j)=0$ then by the inductive hypothesis there
exist an $N$ with $x_{m,s_{m-1}(\ii')}^N=0$ for $\ii'=i_1 \dots
i_{m-3}i_{m-2}j$, so that \eqref{new_eq_UUzero} and
\eqref{new_eq_ijslide} imply $x_{m,\ii j} = \psi_{m-1} x_{m-1}^N
\psi_{m-1} 1_{\ii j} =0$. If $i=j$ then an identical proof as in
\cite[Lemma 2.1]{BK1} or \cite[Proposition 2.9 (i)]{Lau3} shows that
$x_{m,\ii j}$ is nilpotent. Therefore, $x_{r,\ii j}$ is nilpotent
for all  $1 \leq r\leq m$ and we have proven the induction step.

The second claim in the Proposition follows from the first since
$\RnuL$ is spanned by $\{\psi_{\hat{w} ,\ii} x_{1,\ii}^{n_1} \cdots
x_{m,\ii}^{n_m}  \mid w \in S_m, n_1, \dots , n_m \geq 0  \}$ and by
the first claim only finitely many of the $x_{r,\ii}^{n_r}$ are
nonzero for each $1 \leq r\leq m$.
\end{proof}

In terms of the graphical calculus the cyclotomic quotient $\RnuL$ is the quotient of $\Rn{\nu}$ by the ideal generated by
\begin{eqnarray}
  \xy (0,0)*{\supdot{i_1}}; (-4,2)*{\lambda_{i_1}};
  (6,0)*{\sup{\;\; i_2}};
  (12,0)*{\cdots};
  (18,0)*{\sup{\;\;\; i_m}}
  \endxy &=& 0 \label{eq_quotient1}
\end{eqnarray}
over all sequences $\ii$ in $\seq(\nu)$.


For bookkeeping purposes we will denote $\RnuL$-modules in calligraphic font $\cal{M}$ but $\Rnu$-modules by $M$.

We introduce functors
\begin{equation}
\infL \maps \RnuL\dmod \to \Rn{\nu}\fmod \qquad \prL \maps \Rn{\nu}\fmod \to \RnuL\dmod
\end{equation}
where $\infL$ is the inflation along the epimorphism $\Rn{\nu} \to
\RnuL$, so that $\cal{M}=\infL \cal{M}$ on the level of sets.  If
$\cal{M}, \cal{N}$ are $\RnuL$-modules, then
$$\Hom_{\RnuL}(\cal{M}, \cal{N}) \iso \Hom_{\Rnu}(\infL \cal{M}, \infL \cal{N}).$$
Note $\cal{M}$ is irreducible if and only if $\infL\cal{M}$ is.
 We define $\prL {M}=M/\cal{J}^{\Lambda}M$.
If $M$ is irreducible then $\prL M$ is either irreducible or zero.
 Observe $\infL$ is an exact
functor and its left adjoint is $\prL$ which is only right exact.

\begin{prop}
\label{prop_pr}
Let  $\Lambda = \sum_{i \in I} \lambda_i \Lambda_i \in P^+$
and let $M$ be a simple $\Rnu$-module.
Then
\begin{enumerate}
\item
$\cal{J}^\Lambda M = \0$
iff
 $\prL M \neq \0$
iff
$\eph(M) \le \lambda_i$ for all $i \in I$.
When these conditions hold,
we may identify $M$ with the $\RnuL$-module $\prL M$.
\item
$\cal{J}^\Lambda M = M$ if and only if there exists some $i \in I$
such that $\eph(M) > \lambda_i$.
\end{enumerate}
\end{prop}

We omit the proof of the above proposition. It follows from a
 careful study of the simple module $L(i^m)$, as in \cite[Lemma 2.1]{KL}
combined with the properties listed in (2) of Section~\ref{sec_properties}.
The second statement follows from the first as $M$ is simple.
It also follows that  when $\Lambda$ is large enough
$\cal{J}^\Lambda M = \0$, and such $\Lambda$ always exist.
Since any simple $M$ is finite-dimensional, it suffices
to take $\lambda_i > \dim_{\Bbbk} M$ to ensure $\Lambda$ is large enough.

Let $\cal{M}$ be an irreducible $\RnuL$-module.  As in
Section~\ref{sec_define_ei} define
\begin{eqnarray}
  e_{i}^{\Lambda} \cal{M} &=& \prL \circ e_{i} \circ \infL\cal{M}
  \maps \RnuL\dmod \to R^{\Lambda}(\nu-i)\dmod
  \nn \\
  \et{i}^{\Lambda} \cal{M} &=& \prL \circ \et{i} \circ \infL\cal{M}
  \nn \\
\ft{i}^{\Lambda} \cal{M} &=& \prL \circ \ft{i} \circ \infL \cal{M}
 \nn \\
\epsilon_i^{\Lambda}( \cal{M}) &=&  \ep{\infL\cal{M}}
\nn
\end{eqnarray}

Let $\cal{M} \in \RnuL \dmod$ and $M = \infL \cal{M}$.
Then $\prL M = \cal{M}$. Since $\cal{J}^\Lambda M = \0$ then
$\cal{J}^\Lambda e_{i} M = \0$ too, so that
$ e_i^\Lambda \cal{M}$ is an $\Rn{\nu-i}^\Lambda$-module
with $\infL(e_{i}^\Lambda \cal{M})  = e_{i}M$.
In particular, $\dim_{\Bbbk} e_i^\Lambda \cal{M} = \dim_{\Bbbk} e_i M$.
If furthermore  $\cal{M}$ is irreducible, then
$\et{i}^\Lambda \cal{M} =\soc e_i^\Lambda \cal{M}$.

%
\subsection{Ungraded modules}
\label{sec_ungraded}
%

Write $\underline{R\dmod}$, $\underline{R\fmod}$, and $\underline{R\pmod}$ for the corresponding categories of ungraded modules.  There are forgetful functors
\begin{equation}
  R\dmod \to \underline{R\dmod}, \qquad  R\fmod \to \underline{R\fmod}, \qquad R\pmod \to \underline{R\pmod}
\end{equation}
given by sending a module $M$ to the module $\underline{M}$ obtained
by forgetting the gradings, and mapping $\HOM(M,N)$ to
$\underline{\Hom}(\underline{M},\underline{N})$.  Essentially not
much is lost working with the ungraded modules since given an
irreducible module $M \in R\fmod$, then $\underline{M}$ is
irreducible in $\underline{R\fmod}$ \cite[Theorem 4.4.4(v)]{NV}.
Likewise, since $R^{\Lambda}(\nu)$ is a finite dimensional
$\Bbbk$-algebra, if $K \in \underline{R^{\Lambda}(\nu)\fmod}$ is
irreducible, then there exists an irreducible $L \in
R^{\Lambda}(\nu)\fmod$ such that $\underline{L} \cong K$.
Furthermore, $L$ is unique up to isomorphism and grading shift, see
\cite[Theorem 9.6.8]{NV}. Since any finite-dimensional
$R(\nu)$-module $M$ can be identified with the $\RnuL$-module $\prL M$ for some $\Lambda$, we also have that for any irreducible $K \in \underline{R(\nu)\fmod}$ there exists a unique, up to grading shift and isomorphism, irreducible $L \in R(\nu)\fmod$
such that $\underline{L}=K$.

%
\section{Operators on the Grothendieck group} \label{sec_Groth}
%

The Grothendieck groups
\begin{eqnarray}
 K_0(R) = \bigoplus_{\nu\in \N[I]} K_0(\Rn{\nu}\pmod), &\qquad&
     G_0(R) = \bigoplus_{\nu\in \N[I]} G_0(\Rn{\nu}\fmod) \nn \\
 K_0(R^\Lambda) = \bigoplus_{\nu\in \N[I]} K_0(\RnuL\pmod), &\qquad&
    G_0(R^\Lambda) = \bigoplus_{\nu\in \N[I]} G_0(\RnuL\fmod) \nn
\end{eqnarray}
are the direct sums of Grothendieck groups $\Rn{\nu}\pmod$,
$\Rn{\nu}\fmod$, $\RnuL\pmod$, $\RnuL\fmod$ respectively.  The
Grothendieck groups have the structure of a $\Z[q,q^{-1}]$-module
given by shifting the grading, $q[M] = [M\{1\}]$.

The functor $e_i$ defined in \eqref{def_ei} is clearly exact so
descends to an operator on the Grothendieck group
\begin{equation}
  G_0(\Rn{\nu}\fmod) \longrightarrow G_0(\Rn{\nu-i}\fmod)
\end{equation}
and hence
\begin{equation}
  e_i \maps G_0(R) \longrightarrow G_0(R).
\end{equation}
By abuse of notation, we will also call this operator $e_i$.
Likewise $ e_i^\Lambda \maps G_0(R^\Lambda) \longrightarrow G_0(R^\Lambda)$.
We also define divided powers
\begin{equation}
e_i^{(r)} \maps G_0(R) \longrightarrow G_0(R)
\end{equation}
given by $e_i^{(r)}[M]=\frac{1}{[r]_i^!}[e_i^rM]$, which are
well-defined by Section~\ref{sec_Lim}.

For irreducible $M$, we define
$\et{i}[M] = [\et{i} M]$,
$\ft{i}[M] = [\ft{i} M]$,
and extend the action linearly.

The exact functors of induction and restriction induce a multiplication and comultiplication on $G_0(R)$ giving $G_0(R)$ the structure of a (twisted) bialgebra.  More precisely, for $M \in R(\nu)\fmod$ and $N \in R(\mu)\fmod$, the multiplication is given by $[M][N] = [\Ind_{\nu,\mu} M \boxtimes N]$ and the comultiplication by $\Delta[M] = \sum_{\mu_1+\mu_2=\nu} [\Res_{\mu_1,\mu_2}M]$. In that latter we used the fact that simple $R(\mu_1) \otimes R(\mu_2)$-modules have the form $N_1 \boxtimes N_2$ and identified $[N_1 \boxtimes N_2]$ with $[N_1] \otimes [N_2]$.  There is a similar bialgebra structure on $K_0(R)$.

The main categorification results from \cite{KL,KL2} include the following theorem
restated here for completeness.  Although we do not use the results here explicitly, they are mentioned throughout the paper.   The theorem below condenses those of Theorem 3.17, Proposition 3.4, Proposition 3.18 of \cite{KL} and Theorem 8 of \cite{KL2}.

\begin{thm}[Khovanov-Lauda] \label{thm_KL} \hfill
\begin{enumerate}[(1)]
  \item \label{KL-injectchar}The character map
$$ \chr: G_0(\Rn{\nu}\fmod) \lra \Z[q,q^{-1}]\seq(\nu) $$
is injective.
\item
  There is a isomorphism of twisted $\Z[q,q^{-1}]$-bialgebras
\begin{equation}
  \gamma \maps \Af \longrightarrow K_0(R)
\end{equation}
such that multiplication corresponds to the exact functor $\Ind$ and comultiplication is induced by the exact functor $\Res$.

\end{enumerate}
\end{thm}

Note that as a consequence of part \eqref{KL-injectchar} we can deduce that
for any $R(\nu)$-module $M$ its graded character $\chr(M)$ completely determines  $[M] \in G_0(R)$.

Let us consider the maximal commutative subalgebra
$$\bigoplus_{\ii \in \seq(\nu)} \Bbbk[x_{1, \ii}, \ldots, x_{m, \ii}]
\subseteq \Rnu.$$ This ring was called $\Pol_\nu$ in \cite{KL}. In
the notation of this paper, we could also denote it
$\Bbbk[x_1,\dots, x_m] 1_{\nu}$. Its irreducible submodules are one
dimensional, and
are isomorphic to $L(i_1)
\boxtimes L(i_2) \boxtimes \cdots \boxtimes L(i_m)$ and in this way
correspond to $\ii = (i_1, \ldots, i_m) \in \seq(\nu)$.
 In this way, we may identify
$G_0( \Bbbk[x_1,\dots, x_m] 1_{\nu} \fmod)$ with $\Z[q,q^{-1}] \seq(\nu)$.
Hence one may rephrase the injectivity of the character map
as saying that a module is determined
by its restriction to that maximal commutative subalgebra,
in their respective  Grothendieck groups.

Note that the isomorphism classes of simple modules, up to grading
shift, form a basis of $G_0(R)$ as a free $\Z[q,q^{-1}]$-module.
One of the main results of this
paper is that we compute the rank of $G_0(\RnuL\fmod)$ by realizing
a crystal structure on $G_0(R^\Lambda)$ and identifying it as the
highest weight crystal $B(\Lambda)$. In this language, we see the
operators $\et{i}$ and $\ft{i}$ above become crystal operators.

%
\section{Reminders on crystals} \label{sec_crystals}
%

A main result of this paper is the realization of a crystal graph
structure on $G_0(R)$ which we identify as the crystal $B(\infty)$.
Hence, we need to remind the reader of the language and notation of
crystals. For a good introduction to crystal graphs see \cite{Kas4} or \cite{HK}.

%
\subsection{Monoidal category of crystals}
%

We recall the tensor category of crystals following
Kashiwara~\cite{Kas4}, see also~\cite{Kas1,Kas2,KS}.

A {\em crystal} is a set $B$ together with maps
\begin{itemize}
  \item $\wt \maps B \longrightarrow P$,
  \item $\epsilon_i, \phi_i \maps B \longrightarrow \Z \sqcup \{ \infty\}$ \quad for $i \in I$,
  \item $\et{i}, \ft{i} \maps B \longrightarrow B \sqcup \{0\}$ \quad for $i \in I$,
\end{itemize}
such that
\begin{enumerate}[(C1)]
 \item $\phi_i(b) =\epsilon_i(b)+ \langle h_i, \wt(b) \rangle$ \quad for any $i$.
 \item If $b \in B$ satisfies $\et{i} b \neq 0$, then
  \begin{align}
 &  \epsilon_i(\et{i}b) = \epsilon_i(b)-1, & \phi_i(\et{i}b) = \phi_i(b) +1, & & \wt(\et{i}b) = \wt(b)+\alpha_i.
  \end{align}
 \item If $b \in B$ satisfies $\ft{i} b \neq 0$, then
  \begin{align}
 &  \epsilon_i(\ft{i}b) = \epsilon_i(b)+1,
 & \phi_i(\ft{i}b) = \phi_i(b)-1,
 & &\wt(\ft{i}b) = \wt(b)-\alpha_i.
  \end{align}
 \item For $b_1$, $b_2 \in B$, $b_2=\tilde{f}_ib_1$ if and only if $b_1 = \et{i}b_2$.
 \item If $\phi_i(b) = -\infty$, then $\et{i}b=\ft{i}b=0$.
\end{enumerate} \medskip

If $B_1$ and $B_2$ are two crystals, then a {\em morphism} $\psi \maps B_1 \to B_2$ of crystals is a map $$\psi \maps B_1 \sqcup \{0\} \to B_2 \sqcup \{0\}$$ satisfying the following properties:
\begin{enumerate}[(M1)]
  \item $\psi(0) = 0$.
  \item If $\psi(b) \neq 0$ for $b \in B_1$, then
\begin{align}
  & \wt(\psi(b)) = \wt(b),
  & \epsilon_i(\psi(b)) = \epsilon_i(b),
  & &\phi_i(\psi(b)) = \phi_i(b).
\end{align}
 \item For $b \in B_1$ such that $\psi(b) \neq 0$ and $\psi(\et{i}b) \neq 0$, we have $\psi(\et{i}b) = \et{i}(\psi(b))$.
 \item For $b \in B_1$ such that $\psi(b) \neq 0$ and $\psi(\ft{i}b) \neq 0$, we have $\psi(\ft{i}b) = \ft{i}(\psi(b))$.
\end{enumerate}
A morphism $\psi$ of crystals is called {\em strict} if
\begin{equation}
  \psi \et{i} = \et{i}\psi, \qquad \quad \psi \ft{i} = \ft{i}\psi,
\end{equation}
and an {\em embedding} if $\psi$ is injective. \medskip

Given two crystals $B_1$ and $B_2$ their tensor product $B_1 \otimes
B_2$ has underlying set $\{b_1 \otimes b_2; b_1 \in B_1, \;
\text{and} \; b_2 \in B_2 \}$ where we identify $b_1 \otimes 0 = 0
\otimes b_2 = 0$. The crystal structure is given as follows:
\begin{align}
  & \wt(b_1 \otimes b_2) = \wt(b_1) + \wt (b_2), \\
  & \epsilon_i(b_1 \otimes b_2) = \max\max\{ \epsilon_i(b_1), \epsilon_i(b_2)-\langle h_i, \wt(b_1) \rangle\},\\
  & \phi_i(b_1 \otimes b_2) = \max\{ \phi_i(b_1) + \langle h_i,\wt(b_2) \rangle, \phi_i(b_2)\}, \\
  &\et{i}(b_1 \otimes b_2 ) = \left\{
  \begin{array}{lll}
    \et{i}(b_1) \otimes b_2 & \quad & \text{if $\phi_i(b_1) \geq \epsilon_i(b_2)$}\\
    b_1 \otimes \et{i}b_2 & &\text{if $\phi_i(b_1) < \epsilon_i(b_2)$},
  \end{array}
  \right. \label{eq_ei_tensor}\\
  &\ft{i}(b_1 \otimes b_2 ) = \left\{
  \begin{array}{lll}
    \ft{i}(b_1) \otimes b_2 & \quad & \text{if $\phi_i(b_1) > \epsilon_i(b_2)$}\\
    b_1 \otimes \ft{i}b_2 & &\text{if $\phi_i(b_1) \leq \epsilon_i(b_2)$}.
  \end{array}
  \right.
\end{align}

\begin{example}$T_{\Lambda}$ $(\Lambda \in P)$ \hfill

\noindent Let $T_{\Lambda} = \{ t_{\Lambda} \}$ with
$\wt(t_{\Lambda})=\Lambda$, $\epsilon_i(t_{\Lambda}) =
\phi_i(t_{\Lambda})=-\infty$,
$\et{i}t_{\Lambda}=\ft{i}t_{\Lambda}=0$. Note that the underlying set of the crystal $T_{\Lambda}$ consists of a single node.  Tensoring a crystal $B$
with the crystal $T_{\Lambda}$ has the effect of shifting the weight
$\wt$ by $\Lambda$ and leaving the other data fixed.
\end{example}

\begin{example}$B_i$ $(i \in I)$ \newline
$B_i = \{b_i(n) \mid n \in \Z \}$ with $\wt(b_i(n))=n\alpha_i$,
\begin{align}
& \epsilon_j(b_i(n)) =
 \left\{ \begin{array}{lll}
   -n & \quad & \text{if $i=j$}\\
   -\infty & & \text{if $j \neq i$},
 \end{array} \right. &
 \phi_j(b_i(n)) =
 \left\{ \begin{array}{lll}
   n & \quad & \text{if $i=j$}\\
   -\infty & & \text{if $j \neq i$},
 \end{array} \right. \label{eq_Bi_eps}
\end{align}
\begin{align}
  & \et{j}b_i(n) =
  \left\{ \begin{array}{lll}
   b_i(n+1) & \quad & \text{if $i=j$}\\
   0 & & \text{if $j \neq i$},
 \end{array} \right.
 & \ft{j}b_i(n) =
  \left\{ \begin{array}{lll}
   b_i(n-1) & \quad & \text{if $i=j$}\\
   0 & & \text{if $j \neq i$}.
 \end{array} \right.
  &
\end{align}
We write $b_i$ for $b_i(0)$.
\end{example}

%
\subsection{Description of $B(\infty)$}
\label{sec_Binf}
%

$B(\infty)$ is the crystal associated with the crystal graph of
$\Um(\mf{g})$ where $\mf{g}$ is the Kac-Moody algebra defined from
the Cartan data of Section~\ref{sec_Cartan}. One can also define
$B(\infty)$ as an abstract crystal.  As such, it can be
characterized by Kashiwara-Saito's Proposition~\ref{prop_descB}
below.

\begin{prop}[\cite{KS} Proposition 3.2.3] \label{prop_descB}
Let $B$ be a crystal and $b_0$ an element of $B$ with weight zero.  Assume the following conditions.
\begin{enumerate}[(B1)]
  \item $\wt(B) \subset Q_-$.
  \item $b_0$ is the unique element of $B$ with weight zero.
  \item $\epsilon_i(b_0) = 0$ for every $i\in I$.
  \item $\epsilon_i(b) \in \Z$ for any $b\in B$ and $i\in I$.
  \item For every $i\in I$, there exists a strict embedding $\Psi_i \maps B \to B \otimes B_i$.
  \item $\Psi_i(B) \subset B \times \{ \ft{i}^nb_i; n \geq 0 \}$.
  \item For any $b \in B$ such that $b \neq b_0$, there exists $i$ such that $\Psi_i(b)= b'\otimes \ft{i}^nb_i$ with $n>0$.
\end{enumerate}
Then $B$ is isomorphic to $B(\infty)$.
\end{prop}

\section{Module theoretic realizations of certain crystals }
\label{sec_mod_crystals}
%

%
\subsection{The crystal $\cal{B}$}
%

Let $\cal{B}$ denote the set of isomorphism classes of irreducible $R$-modules.
Let $\0$ denote the zero module.

Let $M$ be an irreducible $\Rn{\nu}$-module, so that $[M]\in\cal{B}$.  By abuse of notation, we identify $M$ with $[M]$ in the following definitions. Hence, we are defining operators and functions on $\cal{B} \sqcup \{0\}$ below.

Recall from Section~\ref{sec_define_ei}  the definitions
\begin{align}
  & \et{i}M := \soc e_i M \\
  & \ft{i}M := \cosoc \Ind_{\nu,i}^{\nu+i} M \boxtimes L(i) \\
  & \epsilon_i(M) := \max\{n \geq 0 \mid \et{i}^n M \neq \0 \}
\end{align}
and similarly the $\vee$-versions
\begin{align}
  & \ets{i}M := \sigma^{\ast}(\et{i}(\sigma^{\ast}M)) \\ 
  & \fts{i}M := \sigma^{\ast}(\ft{i}(\sigma^{\ast}M)) 
= \cosoc \Ind_{i, \nu}^{\nu+i} L(i) \boxtimes M, \\
  & \eph(M) := \epsilon_i(\sigma^{\ast}M) 
= \max\{m \geq 0 \mid (\ets{i})^mM \neq \0 \}.
\end{align}

For $\nu=\sum_{i\in I}\nu_i \alpha_i$, $i \in I$ and $M \in \Rn{\nu}\fmod$ set
\begin{equation}
  \wt(M) = -\nu, \qquad \wt_i(M)= \langle h_i, \wt(M)\rangle.
\end{equation}
Set \begin{equation}
  \phi_i(M) = \epsilon_i(M) + \langle h_i, \wt(M) \rangle.
\end{equation}

\begin{prop}
The tuple $(\cal{B},\epsilon_i, \phi_i, \et{i},\ft{i}, \wt )$ defines a crystal.
\end{prop}

\begin{proof}
(C1) is the definition of $\phi_i$.  (C2)--(C4) was shown in
\cite{KL}, see Section~\ref{sec_properties}.  Property (C5) is
vacuous as $\phi_i(b)$ is always finite for $b \in \cal{B}$.
\end{proof}

We write $\1\in \cal{B}$ for the class of the trivial $\Rn{\nu}$-module where
 $\nu=\emptyset$ and  $|\nu|=0$.

One of the main theorems of this paper is Theorem~\ref{thm_Binf} that
identifies the crystal $\cal{B}$ as $B(\infty)$.
However we need the many auxiliary results that follow before we can prove this.

%
\subsection{The crystal $\cal{B} \otimes T_{\Lambda}$}
%

Let $M$ be an irreducible $\Rn{\nu}$-module, so $M \otimes
t_{\Lambda}\in \cal{B} \otimes T_\Lambda$. Then
\begin{eqnarray}
  \ep{M \otimes t_{\Lambda}} &=& \ep{M} \nn \\
  \phi_i(M \otimes t_{\Lambda}) &=& \phi_i(M)+\lambda_i \nn \\
 \et{i}(M \otimes t_{\Lambda}) &=& \et{i}M \otimes t_{\Lambda} \nn \\
 \ft{i}(M \otimes t_{\Lambda}) &=& \ft{i}M \otimes t_{\Lambda} \nn \\
 \wt(M \otimes t_{\Lambda}) &=& -\nu+\Lambda. \nn
\end{eqnarray}

%
\subsection{The crystal $\cal{B}^{\Lambda}$} \label{sec-BLambda}
%

Let $\cal{B}^{\Lambda}$ denote the set of isomorphism classes of irreducible $R^{\Lambda}$-modules.  As in the previous section, by abuse of notation we write $\cal{M}$ for $[\cal{M}]$ below.  Define
\begin{eqnarray}
  \et{i}^{\Lambda} \maps \cal{B}^{\Lambda} &\to& \cal{B}^{\Lambda}\sqcup\{\0\} \nn \\
  \cal{M} & \mapsto& \prL \circ \et{i} \circ \infL\cal{M}
  \nn \\
\ft{i}^{\Lambda} \maps \cal{B}^{\Lambda} &\to& \cal{B}^{\Lambda}\sqcup\{\0\} \nn \\
  \cal{M} & \mapsto& \prL \circ \ft{i} \circ \infL \cal{M}
 \nn \\
\epsilon_i^{\Lambda} \maps \cal{B}^{\Lambda} &\to& \Z\sqcup\{-\infty\} \nn \\
  \cal{M} & \mapsto&  \ep{\infL\cal{M}}
\nn \\
  \phi_i^{\Lambda} \maps \cal{B}^{\Lambda} &\to& \Z\sqcup\{-\infty\} \nn \\
  \cal{M} & \mapsto&  \max\{k \in \Z \mid \prL \circ \ft{i}^k \circ \infL \cal{M} \neq \0 \}
\nn \\
  \wt^{\Lambda} \maps \cal{B}^\Lambda &\to& P \nn \\
  \cal{M} &\mapsto& -\nu + \Lambda.
\end{eqnarray}
Note $\epsilon_i^{\Lambda}(\cal{M}) =
\max\{k \in \Z \mid  (\et{i}^{\Lambda})^k\cal{M} \neq \0\}$,
and $0 \leq \phiL(\cal{M})< \infty$.

It is true, but not at all obvious, that with this definition
$\phiL(\cal{M})=\epsilon_i^{\Lambda}(\cal{M})+\langle
h_i,\wt^{\Lambda}\cal{M} \rangle$; see Corollary~\ref{cor_jump}. The
proof that the data
$(\cal{B}^{\Lambda},\epsilon_i^{\Lambda},\phiL,\et{i}^{\Lambda},\ft{i}^{\Lambda},\wt^{\Lambda})$
defines a crystal is delayed until Section~\ref{sec_harvest}.

On the level of sets define a function
\begin{eqnarray}
  \Upsilon \maps \cal{B}^{\Lambda} &\to& \cal{B} \otimes T_{\Lambda} \nn \\
  \cal{M} & \mapsto& \infL \cal{M} \otimes t_{\Lambda}.
\end{eqnarray}
The function $\Upsilon$ is clearly injective and satisfies
\begin{eqnarray} \label{ups1}
  \epsilon_i^{\Lambda}(\cal{M})& =& \epsilon_i(\Upsilon\cal{M}), \\
\Upsilon \et{i}^{\Lambda} \cal{M} &=& \et{i} \Upsilon \cal{M}, \label{ups2}\\
  \Upsilon \ft{i}^{\Lambda} \cal{M} &=&\left\{
\begin{array}{ccc}
  \ft{i}\Upsilon \cal{M} & \quad & \ft{i}^{\Lambda}\cal{M} \neq \0 \\
 \0 & \quad & \ft{i}^{\Lambda}\cal{M}= \0
\end{array}
\right.  \label{ups3} \\
\wt^{\Lambda}(\cal{M})
&=&
\wt(\Upsilon \cal{M})
\label{ups4}.
\end{eqnarray}
Later we will see the relationship between $\phiL(\cal{M})$
and $\phi_i(\infL \cal{M})$.
 Once this relationship is in place (see Corollary~\ref{cor_jump}) it will imply $\Upsilon$ is an embedding of crystals and in particular
that $\cal{B}^{\Lambda}$ is a crystal.  In Section~\ref{sec_harvest} we show that $\cal{B} \cong B(\infty)$ which then identifies $\cal{B}^{\Lambda}$ as the highest weight crystal $B(\Lambda)$.

%
\section{Understanding $\Rnu$-modules and the crystal data of $\cal{B}$}
\label{sec_modules}
%
This section contains an in depth study of simple $\Rnu$-modules and the
functor $\ft{i}$.
In particular, we describe how the quantities $\ephj$, $\epsilon_i$, $\phiL$
change with the application of $\ft{j}$.

Throughout this section we assume $j \neq i$ and set $a=a_{ij}=-\langle h_i,\alpha_j \rangle$.

%
\subsection{Jump} \label{sec_jump}
%

Given an irreducible module $M$, $\prL \ft{i}M$ is either irreducible or zero.
In the following subsection, we determine exactly when the latter occurs.
More specifically, we compare $\eph(M)$ to $\eph(\ft{i}M)$
and compute when the latter quantity ``jumps" by $+1$.
In this case, we show  $\ft{i}M \cong \fts{i}M$.
Understanding exactly when this jump occurs will be  a key ingredient in
constructing the strict embedding of crystals in Section~\ref{sec_embedding}.

One very useful byproduct  of understanding co-induction is that for
irreducible $M$ if we know $\ft{i}M \cong \fts{i}M$ then we can
easily conclude $\ft{i}^mM \cong \Ind M \boxtimes L(i^m) \cong \Ind
L(i^m) \boxtimes M$, not just for $m=1$, but for all $m \geq 1$, and
in particular that the latter module is irreducible.  We will prove
this in Lemma~\ref{lem_jump_tfae} below. While for the main
results of this paper, it suffices to understand exactly when
$\ft{i}M \cong \fts{i}M$, we found it worthwhile to include
Section~\ref{sec_coind} precisely for the sake of a deeper
understanding of $\Ind M \boxtimes L(i)$.

The following proposition is a consequence of Theorem~\ref{thm_coind}.
and the properties listed in Section~\ref{sec_properties}.

\begin{prop} \label{prop_coind}
 Let $M$ be an irreducible $\Rn{\nu}$-module.  Let $n \geq 1$.
Then
\begin{enumerate}
  \item \label{cor_coindp1} $\ft{i}^nM \cong \soc \Indc M \boxtimes L(i^n) \cong \soc \Ind L(i^n) \boxtimes M$.
 \item \label{cor_coindp2} $(\fts{i})^nM \cong \soc \Indc L(i^n) \boxtimes M \cong \soc\Ind M \boxtimes L(i^n)$.
\end{enumerate}
\end{prop}

\begin{proof}
 Let $m = \ep{M}$ and $N = \et{i}^m M$.  Recall from Section~\ref{sec_properties}
\begin{equation}
  \Res_{\nu-mi,mi}M \cong N \boxtimes L(i^m).
\end{equation}
We thus have a nonzero map  $\Res_{\nu-mi,mi}M \to N \boxtimes
L(i^m)$, hence a nonzero and thus injective map
\begin{equation}
  M \to \Indc N \boxtimes L(i^m).
\end{equation}

Repeating the standard arguments from \cite{GV,KL} we see $M \cong
\soc \Indc N \boxtimes L(i^m)$ and that all other composition
factors have $\epsilon_i$ strictly smaller that $m$. Likewise we
have $\ft{i}^n M \cong \soc \Indc N \boxtimes L(i^{m+n})$ and deduce
statement \eqref{cor_coindp1}, using Theorem~\ref{thm_coind}.  The
proof of \eqref{cor_coindp2} is similar.
\end{proof}

It is necessary to understand how $\eph$ changes with application of $\ft{j}$.
\begin{prop} \label{prop_K10.1.3}
Let $M$ be an irreducible $\Rn{\nu}$-module.
\begin{enumerate}[i)]
  \item For any $i \in I$, either $\eph(\ft{i}M)=\eph(M)$ or $\eph(M)+1$.
  \item For any $i,j \in I$ with $i \neq j$, we have $\eph(\tilde{f}_jM)=\eph(M)$ and $\ep{\fts{j}M}=\ep{M}$.
\end{enumerate}
\end{prop}

\begin{proof}
Consider $\Ind M \boxtimes L(j)\twoheadrightarrow \ft{j}M$, so by
Frobenius reciprocity $\eph(\ft{j}M)\geq \eph(M)$.  On the other
hand, by the Shuffle Lemma
\begin{equation}
  \eph(\ft{j}M) \leq \eph(M)+\eph(L(j)) = \eph(M)+\delta_{ij}.
\end{equation}
In the case $i=j$ we then get $\eph(M) \leq \eph(\ft{j}M)\leq
\eph(M)+1$ and in the case $i \neq j$ $\eph(M) \leq
\eph(\ft{j}M)\leq \eph(M)$.  Applying the automorphism $\sigma$ in
the case $i \neq j$ also yields the symmetric statement
$\ep{\fts{j}M}=\ep{M}$.
\end{proof}

\begin{defn}
Let $M$ be an irreducible $\Rn{\nu}$-module and let $\Lambda \in P^+$.  Define
\begin{equation}
  \phiL(M) = \max\{k \in \Z \mid \prL \ft{i}^k M \neq \0 \}.
\end{equation}
where we take the convention that $\ft{i}^k=\et{i}^{-k}$ when $k<0$,
and that $\max \emptyset = -\infty$.
\end{defn}

Note that $\prL M \neq \0$ if and only if $\phiL(M) \geq 0$ for all
$i \in I$ by Proposition~\ref{prop_pr},
or even for a single $i \in I$ by
Proposition~\ref{prop_K10.1.3}.  Hence, by allowing $\phiL$ to take
negative values, we can use $\phiL$ to detect which irreducible
$\Rn{\nu}$-modules are in fact $\RnuL$-modules.  Thus when $\phiL(M)
\geq 0$ it agrees with $\phiL(\prL M)$ as defined in
Section~\ref{sec-BLambda} which is manifestly nonnegative.  By abuse
of notation we call both functions $\phiL$.

Observe that 
\begin{equation}
  \phiL(\ft{i}M) = \phiL(M)-1.
\end{equation}
We warn the reader that with this extended definition of $\phiL$ on $G_0(R)$, it
not only takes negative values but can be equal to $-\infty$.
For example, take $\Lambda = \Lambda_i$, and let $j \neq i$.
Then $\et{i} L(j) = \0$ and we see $\prL \ft{i}^k L(j)  = \0$ for
all $k \in \Z$ by Proposition~\ref{prop_K10.1.3}.
Hence $\phiL(L(j)) = - \infty$.
However, this is no call for alarm, as by
Proposition~\ref{prop_pr},
we can always find a larger $\Lambda$ so that $\prL M \neq \0$ for any given $M$.

\begin{defn}
Let $M$ be a simple $\Rn{\nu}$-module and let $i\in I$.  Then
\begin{equation}
  \jump_i(M) := \max\{ J \geq 0 \mid \eph(M)=\eph(\ft{i}^J M)\}.
\end{equation}
\end{defn}
While it is clear $\jump_i (M) \geq 0$, it is less clear why $\jump_i(M) < \infty$.
We show this in Proposition~\ref{prop_jump_tfae2} \eqref{jumpwt}.

In the following Lemma we collect a long list of useful characterizations of when $\jump_i(M) = 0$.  We find it convenient to be overly thorough below and furthermore  to give  this lemma the name ``Jump Lemma" because we use it repeatedly throughout the paper.

We remind the reader that the isomorphisms below are homogeneous but
not necessarily degree preserving.

\begin{lem}[Jump Lemma]\label{lem_jump_tfae}
  Let $M$ be irreducible.  The following are equivalent:
\begin{enumerate}[1)]
  \item \label{tfaep1} $\jump_i(M)=0$
  \item \label{tfaep2}$\ft{i}M\cong \fts{i}M$
  \item \label{tfaep3} $\ft{i}^mM\cong (\fts{i})^mM$ for all $m \geq 1$
  \item \label{tfaep4} $\Ind M \boxtimes L(i) \cong \Ind L(i) \boxtimes M$
\item \label{tfaep5} $\Ind M \boxtimes L(i^m) \cong \Ind L(i^m)\boxtimes M$ for all $m\geq
 1$
  \item \label{tfaep6}$\ft{i}M \cong \Ind M \boxtimes L(i)$
\makebox[3.3in][r]{
$6'$) \,
$\fts{i}M\cong \Ind L(i) \boxtimes M$
}
  \item \label{tfaep7} $\Ind M \boxtimes L(i)$ is irreducible
\makebox[3.2in][r]{
$7'$) \,
$\Ind L(i) \boxtimes M$ is irreducible
}
  \item \label{tfaep8}
\makebox[2.8in][l]{
$\Ind M \boxtimes L(i^m)$ is irreducible
}
\makebox[2.8in][l]{ $8'$) \, $\Ind L(i^m) \boxtimes M$ is irreducible }
\\
\makebox[3.2in][l]{ for all $m \geq1$  } \makebox[3in][l]{for all $m
\geq 1$ }
  \item \label{tfaep9} $\eph(\ft{i}M)=\eph(M)+1$
\makebox[3.2in][r]{
$9'$) \,
$\ep{\fts{i}M} = \ep{M}+1$
}
  \item \label{tfaep10} $\jump_i(\ft{i}^m M)=0$ for all $m \geq 0$
  \item \label{tfaep11} $\eph(\ft{i}^mM) = \eph(M)+m$ for all $m \geq 1$
\end{enumerate}
\end{lem}

\begin{proof}
Pairs of ``symmetric" conditions labelled by $X)$ and $X')$ are clearly equivalent to each
other by applying the automorphism $\sigma$, except for $ \eqref{tfaep9} \Leftrightarrow
\eqref{tfaep9}'$ which is slightly less obvious.
We will show $\eqref{tfaep2}\Leftrightarrow \eqref{tfaep9}$ which then gives
$\eqref{tfaep2} \Leftrightarrow \eqref{tfaep9}'$ by $\sigma$-symmetry.

By Proposition~\ref{prop_K10.1.3}, we have $\eph(M) \leq
\eph(\ft{i}M) \leq \eph(M)+1$. This yields $\eqref{tfaep1}
\Leftrightarrow \eqref{tfaep9}$. Suppose \eqref{tfaep9} holds, i.e.
$\eph(\ft{i}M)= \eph(M)+1=\eph(\fts{i}M)$.  By the Shuffle Lemma,
\begin{equation}
  \chr(\Ind M \boxtimes L(i))\mid_{q=1} = \chr(\Ind L(i) \boxtimes
  M)\mid_{q=1},
\end{equation}
so by the injectivity of the character map and the discussion of
Section~\ref{sec_ungraded}, they have same composition factors. But
$\fts{i}M$ is the unique composition factor of $\Ind L(i) \boxtimes
M$ with largest $\eph$, forcing $\ft{i}M \cong \fts{i}M$
which yields \eqref{tfaep2}.  The converse of
$\eqref{tfaep2} \To \eqref{tfaep9}$ is obvious.  So we have
$\eqref{tfaep2} \Leftrightarrow \eqref{tfaep9}$ and by
$\sigma$-symmetry also $\eqref{tfaep2} \Leftrightarrow
\eqref{tfaep9}'$.

Next suppose $\eqref{tfaep2}$, i.e. $\ft{i}M \cong \fts{i}M$.  This implies
\begin{equation}
 \cosoc \Ind M \boxtimes L(i) \cong \soc \Indc L(i) \boxtimes M \cong \soc \Ind M \boxtimes
L(i)
\end{equation}
by Proposition~\ref{prop_coind}.  Furthermore from
Section~\ref{sec_properties}, $\ft{i}M$ is not only the cosocle, but
occurs with multiplicity one in $\Ind M \boxtimes L(i)$. For it to
also be
 the socle forces $\Ind M \boxtimes L(i)$ to be irreducible, yielding \eqref{tfaep7}.
Clearly $\eqref{tfaep7} \Leftrightarrow  \eqref{tfaep6}$.  Further
$\eqref{tfaep7} \To \eqref{tfaep4}$ as $\chr(\Ind M \boxtimes L(i))
= \chr(\Ind L(i) \boxtimes M)$ at $q=1$.

Given \eqref{tfaep4} an inductive argument and transitivity of
induction gives \eqref{tfaep5}, that $\Ind M \boxtimes L(i^m) \cong
\Ind L(i^m) \boxtimes M$ for all $m \geq 1$.  Thus, $ \ft{i}^mM
\cong \cosoc \Ind M \boxtimes L(i^m) \cong \cosoc \Ind L(i^m)
\boxtimes M \cong (\fts{i})^mM$, yielding \eqref{tfaep3} and thus
\eqref{tfaep11} by then evaluating $\eph$.  That $\eqref{tfaep11}
\To \eqref{tfaep3}$ is an identical argument to $\eqref{tfaep9} \To
\eqref{tfaep2}$.

Now suppose \eqref{tfaep3} holds.  Again by Proposition~\ref{prop_coind}
\begin{equation}
  \cosoc \Ind M \boxtimes L(i^m) \cong \soc \Indc L(i^m) \boxtimes M \cong \soc \Ind M
\boxtimes L(i^m)
\end{equation}
so as above $\Ind M \boxtimes L(i^m)$ is irreducible, yielding \eqref{tfaep8},
and hence it is isomorphic to $\ft{i}^mM$.

It is trivial to check $\eqref{tfaep8} \To \eqref{tfaep7}
\To \eqref{tfaep4} \To \eqref{tfaep2}$ and $\eqref{tfaep6} \Leftrightarrow
\eqref{tfaep6}'$, $\eqref{tfaep7} \Leftrightarrow \eqref{tfaep7}'$, $\eqref{tfaep8}
\Leftrightarrow \eqref{tfaep8}' $.
Finally, since $\eqref{tfaep1}\Leftrightarrow \eqref{tfaep11}$ we certainly have
$\eqref{tfaep1}\Leftrightarrow \eqref{tfaep10}$.  This completes the proof.
\end{proof}

The following proposition gives alternate characterizations of $\jump_i(M)$.
Although we do not prove that all five hold at this time, it
is worth stating them all together now.

\begin{prop} \label{prop_jump_tfae}
Let $M$ be a simple $\Rn{\nu}$-module and let $i\in I$. Then the following hold.
\begin{enumerate}[i)]
  \item \label{XX2}  $\jump_i(M) = \min \{ J \geq 0 \mid \ft{i}(\ft{i}^J M) \cong \fts{i}(\ft{i}^J M)\}$
  \item \label{jumpphi}
  If $\phiL(M) > - \infty$, then $\jump_i(M)=\phiL(M) +
\eph(M)-\lambda_i$, where $\Lambda = \sum_{i} \lambda_i \Lambda_i \in P^+$.
\end{enumerate}
\end{prop}

\begin{proof}
We first prove \eqref{XX2}.  Let $J = \jump_i(M)$ and $N = \ft{i}^J
M$. Then by the maximality of $J$, $\eph(\ft{i}N) =
\eph(N)+1=\eph(M)+1$.  By the Jump Lemma~\ref{lem_jump_tfae}, $\ft{i}N \cong
\fts{i}N$, i.e. $\ft{i}(\ft{i}^JM) \cong \fts{i}(\ft{i}^JM)$.
Further, if $0 \leq m < J$ then
\begin{equation}
  \eph(\fts{i}\ft{i}^m M) = 1 + \eph(\ft{i}^m M) = 1+\eph(M) > \eph(M) = \eph(\ft{i}^{m+1}M)
\end{equation}
so $\fts{i}\ft{i}^mM \ncong \ft{i}\ft{i}^m M$.  This yields \eqref{XX2}.

Now we prove \eqref{jumpphi}.  Again let $J=\jump_i(M)$.  First,
suppose $\phiL(M) \geq 0$.  Then, as $\prL \ft{i}^{\phiL(M)}M \neq
\0$, it follows from Proposition~\ref{prop_K10.1.3} and
Proposition~\ref{prop_pr} that $\prL M \neq \0$.  Hence $\lambda_i \geq
\eph(M)=\eph(\ft{i}^JM)$.  But by
\eqref{tfaep11} of the Jump Lemma, 
$\eph(\ft{i}^{J+m}M)=\eph(M)+m$ for all $m \geq 0$.

Set $m=\lambda_i-\eph(M)$. Then by the maximality of $J$,
$\eph(\ft{i}^{J+m}M)=\lambda_i$ but
$\eph(\ft{i}^{J+m+1}M)=\lambda_i+1$.  And by
Proposition~\ref{prop_K10.1.3} $\ephj(\ft{i}^{J+m}M) = \ephj(M) \le
\lambda_j$.  In other words $\prL \ft{i}^{J+m}M \neq \0$ but $\prL
\ft{i}^{J+m+1}M =\0$, so by definition
$\phiL(M)=J+m=\jump_i(M)+\lambda_i-\eph(M)$.  Equivalently
$\jump_i(M)-\phiL(M)+\eph(M)-\lambda_i$.

Second, if  $- \infty < \phiL(M)<0$, let $k=-\phiL(M)$. Note
$\eph(\et{i}^kM)=\lambda_i$ but $\eph(\et{i}^{k-1}M)=\lambda_i+1$ so
that $\jump_i(\et{i}^kM)=0$ and hence $\jump_i(M)=0$ too, by
characterization \eqref{tfaep10} of the Jump Lemma. 
As before, $\eph(M)=\eph(\ft{i}^k\et{i}^kM)=\eph(\et{i}^kM)+k=\lambda_i-\phiL(M)$.
So again $\jump_i(M)=0=\phiL(M)+\eph(M)-\lambda_i$.
\end{proof}

It is clear from the Proposition~\ref{prop_jump_tfae} that
\begin{equation} \label{jumprecursion}
  \jump_i(\ft{i}M) = \max\{0, \jump_i(M)-1\}.
\end{equation}
We continue our list of characterizations of $\jump_i$ in a separate proposition below, whose proof is postponed to the end of this Section~\ref{sec_phi-lambda}.

\begin{prop} \label{prop_jump_tfae2}
Let $M$ be a simple $\Rn{\nu}$-module and let $i\in I$. Then the following hold.
\begin{enumerate}[i)]\addtocounter{enumi}{2}
\item
\label{XX1} $\jump_i(M) = \max\{ J \geq 0 \mid
\epsilon_i(M)=\epsilon_i((\fts{i})^J M)\}$
\item
\label{XX3}  $\jump_i(M) = \min \{ J \geq 0 \mid \ft{i}((\fts{i})^J M) \cong \fts{i}((\fts{i})^J M)\}$
\item
\label{jumpwt}  $\jump_i(M)=\epsilon_i(M) + \eph(M)+\wt_i(M)$.
\end{enumerate}
\end{prop}

We must delay the proof of \eqref{jumpwt} until we have proved
Theorem~\ref{thm_epsilonifj} and consequently
Corollary~\ref{cor_jump}.

The equivalence of Proposition~\ref{prop_jump_tfae} \eqref{XX2} to the definition of $\jump_i$ is $\sigma$-symmetric to the equivalence of $\eqref{XX1} \Leftrightarrow \eqref{XX3}$, and \eqref{XX2} is $\sigma$-symmetric to \eqref{XX3}. So once we have \eqref{jumpwt} whose right-hand side is a $\sigma$-symmetric expression, we will have all \eqref{XX1}--\eqref{jumpwt} of Proposition~\ref{prop_jump_tfae2}.

\begin{rem}\label{rem_Omega}
Given $\Lambda, \Omega \in P^+$ and irreducible modules $A$ and $B$ with $\prL A \neq \0$, $\prO A \neq \0$, $\prL B \neq \0$, $\prO B \neq \0$, then
$\phiL(A)-\phiL(B) = \phiO(A)-\phiO(B)$ since by Proposition~\ref{prop_jump_tfae}~\eqref{jumpphi} we compute
\begin{align}
  \phiL(A)-\phiL(B) &= (\jump_i(A)-\eph(A)+\lambda_i)-
  (\jump_i(B)-\eph(B)+\lambda_i) \\
  &= \jump_i(A)-\jump_i(B)+\eph(B)-\eph(A) \\
  &= \phiO(A)-\phiO(B).
\end{align}
\end{rem}

%
\subsection{Serre relations} \label{sec_Serre}
%

In this section we discuss the quantum Serre relations~\eqref{Serrea} which are certain  (minimal) relations that hold
among the operators $e_i$ on $G_0(R)$.  We refer the reader to \cite{KL2}, where
they prove similar relations (the vanishing of alternating sums in
$K_0(R)$) hold on a certain family of projective modules in their
Corollary 7. Then by the  obvious generalization to the
symmetrizable case of Corollary 2.15 of \cite{KL}
we have
\begin{equation}
  \sum_{r=0}^{a+1}(-1)^r e_i^{(a+1-r)}e_je_i^{(r)}[M]=0
\end{equation}
for all $M \in R(\nu)\dmod$ with $|\nu|=a+1$,
where $a = - \langle h_i, \alpha_j \rangle$, and hence for all
$[M]\in G_0(R)$, showing the operator
\begin{equation} \label{Serrea}
  \sum_{r=0}^{a+1}(-1)^r e_i^{(a+1-r)}e_je_i^{(r)}=0.
\end{equation}
Recall the divided powers $e_i^{(r)}$ are given by
$e_i^{(r)}[M]=\frac{1}{[r]_i^!}[e_i^rM]$.

Furthermore, when $c\leq a$ the operator
\begin{equation} \label{Serrec}
  \sum_{r=0}^{c}(-1)^r e_i^{(c-r)}e_je_i^{(r)}
\end{equation}
is never the zero operator on $G_0(R)$ by the quantum Gabber-Kac
Theorem~\cite[Theorem 33.1.3]{Lus4} and the work of \cite{KL,KL2}, which essentially computes the kernel of the map from the free algebra on the
generators $e_i$ to $G_0(R)$, see Remark~\ref{rem_nosmaller}.

In Section~\ref{sec_genrel} below, we give an alternate proof that the
quantum Serre relation~\eqref{Serrea} holds
by examining the structure of all simple $\Rn{(a+1)i + j}$-modules.
We further construct simple $\Rn{ci + j}$-modules that are witnesses to the
nonvanishing of~\eqref{Serrec} when $c \le a$.
In the following remark, we give a sample argument of how understanding
the simple $\Rnu$-modules for a fixed $\nu$ gives a relation among the operators
$e_i$ on $G_0(R)$.
Although we only give it in detail for a degree 2 relation among the $e_i$, it can
be easily extended to higher degree relations.

\begin{rem}\label{rem_zero}
Suppose we have explicitly constructed all simple $\Rn{i+j}$-modules
$M$, and have verified $(e_ie_j-e_je_i)[M]=0$ for all such $M$. (We
know this is the case whenever $\langle i,j \rangle = 0$.) We will
call this a degree 2 relation in the $e_i$'s for obvious reasons. We
easily see the operator $e_ie_j-e_je_i$ is zero on
$G_0(\Rn{\mu}\fmod)$ not just for $\mu = i+j$ but for any $\nu$ with
$|\mu|= 0, 1, 2$.
%
Now consider arbitrary $\nu$ with $|\nu| > 2$. Let $M$ be any finite
dimensional $\Rnu$-module. We can write $[\Res_{\nu - \mu, \mu} M] =
\sum_h [A_h \boxtimes B_h]$ for some simple $\Rn{\mu}$-modules $B_h$
with $|\mu|=2$, or the restriction is zero.  Then
\begin{align}
(e_ie_j-e_je_i)[M] &= \sum_{ \mu : |\mu|=2 } \sum_h [A_h \boxtimes (e_ie_j-e_je_i)B_h ]
\\ &= \sum \sum [A_h \boxtimes \0] = 0.
\end{align}
Hence  $e_ie_j-e_je_i$ is zero as an operator on $G_0(R)$. However,
this is a relation of the form~\eqref{Serrec} with $c=0$. By the
discussion above on the minimality of the quantum Serre relation,
this forces $a_{ij}=0$. Similarly, if one shows the expression
\eqref{Serrea} in the quantum Serre relation vanishes on all
irreducible $\Rn{(a+1)i +j}$-modules, the same argument shows the
relation holds on all $G_0(R)$ and that $a_{ij} \leq a$.
\end{rem}

%
\subsection[The structure theorems]{The Structure Theorems for simple
$\Rn{ci+j}$-modules} \label{sec_structure}
%

In this section we describe the structure of all simple $\Rn{ci+j}$-modules.
We will henceforth refer to Theorems~\ref{thm_simplec}, ~\ref{thm_circleL} as
the Structure  Theorems for simple  $\Rn{ci+j}$-modules.
Throughout this section we assume $j \neq i$ and set
$a=a_{ij}=-\langle h_i,\alpha_j \rangle$.

In the theorems below we introduce the notation
$$\Lj{i^{c-n}ji^n} \quad \text{and} \quad \Lc{n} \define  \Lj{i^{a-n}ji^n}$$
for the irreducible $\Rn{ci+j}$-modules
when $c \le a$.
They are characterized by $\epsilon_i\left(\Lj{i^{c-n}ji^n}\right)=n$.

\begin{thm} \label{thm_simplec}
Let $c\leq a$ and let $\nu= ci+j$. Up to isomorphism,
there exists a unique irreducible $\Rn{\nu}$-module denoted $\Lj{i^{c-n}ji^n}$ with
\begin{equation} \label{Ln}
\epsilon_i\left(\Lj{i^{c-n}ji^n}\right)=n
\end{equation}
for each $n$ with $0 \leq n \leq c$.  Furthermore,
\begin{equation}\label{Lcn}
  \eph(\Lj{i^{c-n}ji^n})=c-n
\end{equation}
and
\begin{equation} \label{charL}
 \chr(\Lj{i^{c-n}ji^n}) = [c-n]_i![n]_i!i^{c-n}ji^n .
\end{equation}
In particular, in the Grothendieck group $e_{i}^{(c-s)}e_je_i^{(s)} [ \Lj{i^{c-n}ji^n} ] =0$ unless $s=n$.
\end{thm}

\begin{proof}
The proof is by induction on $c$.  The case $c=0$ is obvious;
there exists a unique irreducible $R(j)$-module $L(j)$ and it
obviously satisfies \eqref{Ln}--\eqref{charL}.

The case $c=1$ is also straightforward.  Since $c \leq a$, and so $a\neq 0$, we compute $\Ind L(i)\boxtimes L(j)$ is reducible, but has irreducible cosocle.  Let
\begin{align}
  \Lj{ij} &= \cosoc \Ind L(i) \boxtimes L(j) \\
  \Lj{ji} &= \cosoc \Ind L(j) \boxtimes L(i).
\end{align}
Note each of the above modules is one-dimensional and satisfies \eqref{Ln}--\eqref{charL}.  Observe if \eqref{Ln} did not hold for either module, then
by the Jump Lemma~\ref{lem_jump_tfae}
\begin{equation}
  \Ind L(i) \boxtimes L(j) \cong \Ind L(j) \boxtimes L(i)
\end{equation}
and this module would be irreducible.  Hence for all $\Rn{i+j}$-modules $M$ we would have
\begin{equation}
  (e_ie_j-e_je_i)[M]=0
\end{equation}
and in fact this relation would then hold for any $\nu$ and any irreducible $\Rn{\nu}$-module $M$ via Remark~\ref{rem_zero}.
But by \eqref{Serrec} this would imply $a=0$, a contradiction.

Now assume the theorem holds for some fixed $c \leq a$ and we will show it also holds for $c+1$ so long as $c+1 \leq a$.  Let $N$ be an irreducible $\Rn{(c+1)i+j}$-module with $\ep{N}=n$.

Suppose $n>0$.  If in fact $n=0$ consider instead ${n}^\vee=\eph{N}$
which cannot also be 0 and perform the following argument applying
the automorphism $\sigma$ everywhere.
Observe any other
module $N'$ such that $\ep{N'} = n$ has $\et{i}N' \iso \et{i} N$,
forcing $N' \iso N$, which gives us the uniqueness.
Note $\et{i}N$ is an
$\Rn{ci+j}$-module with $\ep{\et{i}N}=n-1$ so by the inductive
hypothesis $\et{i}N= \Lj{i^{c+1-n}ji^{n-1}}$.
We have a surjection (up to grading shift)
\begin{equation}
  \Ind \Lj{i^{c+1-n}ji^{n-1}} \boxtimes L(i) \twoheadrightarrow N.
\end{equation}
Since $N=\cosoc \Ind \Lj{i^{c+1-n}ji^{n-1}} \boxtimes L(i)$, by
Frobenius reciprocity, the Shuffle Lemma, and the fact that $L(i^m)$
is irreducible with character $[m]_i^!i^m$, either we have
\begin{equation}
  \chr(N) = [c+1-n]_i^![n]_i^!i^{c+1-n}ji^{n}
\end{equation}
or
\begin{eqnarray}
  \chr(N) &=& [c+1-n]_i^![n]_i^!i^{c+1-n}ji^{n}+q^{-(\alpha_i,\alpha_j)}[c+2-n]_i^![n-1]_i^!i^{c+2-n}ji^{n-1} \\
&=&\chr(\Ind \Lj{i^{c+1-n}ji^{n-1}} \boxtimes L(i)).
\end{eqnarray}
In the former case, $N$ satisfies \eqref{charL} and of course also \eqref{Lcn}.  In the latter case, by the injectivity of the character map,
 we must have isomorphisms $N \cong \Ind \Lj{i^{c+1-n}ji^{n-1}} \boxtimes L(i)$ and in fact
\begin{equation} \label{badL}
  \Ind \Lj{i^{c+1-n}ji^{n-1}} \boxtimes L(i)
\cong \Ind L(i) \boxtimes \Lj{i^{c+1-n}ji^{n-1}}.
\end{equation}
Next we will show that if \eqref{badL} holds for this $n$, then it holds for all $1 \leq n \leq c$.

Let $M = \cosoc \Ind L(i) \boxtimes \Lj{i^{c-n}ji^{n}}$ which is irreducible.  By the Shuffle Lemma, either $\ep{M}=n$ or $\ep{M}=n+1$.  If $\ep{M}=n$, then by uniqueness
part of the inductive hypothesis  $\et{i}M\cong \et{i}N$ and so $M \cong N$.  But this is impossible as $i^{c+2-n}ji^{n-1}$ can never be a constituent of $\chr(M)$.  So we must have $\ep{M}=n+1$.  Repeating the same analysis of characters as above we must have
\begin{equation}
  M \cong \Ind L(i) \boxtimes \Lj{i^{c-n}ji^{n}} \cong \Ind \Lj{i^{c-n}ji^{n}} \boxtimes L(i).
\end{equation}
Continuing in this manner, we deduce
\begin{equation}\label{worseL}
  \Ind L(i) \boxtimes \Lj{i^{c-g}ji^{g}} \cong \Ind \Lj{i^{c-g}ji^{g}}\boxtimes L(i)
\end{equation}
for all $n-1\leq g \leq c$.

We may repeat the same argument applying the automorphism $\sigma$ everywhere.  In other words consider $\eph(N)=c+2-n$ and start with
\begin{equation}
  M'= \cosoc \Ind \Lj{i^{c+2-n}ji^{n-2}} \boxtimes L(i)
\end{equation}
which will force $\eph(M')= c+3-n$ and
\begin{equation}
\Ind
  \Lj{i^{c+2-n}ji^{n-2}} \boxtimes L(i) \cong \Ind L(i) \boxtimes \Lj{i^{c+2-n}ji^{n-2}}.
\end{equation}
Continuing  as before yields isomorphisms \eqref{worseL} for $n-1 >g \geq0$,
in other words for all $g$.

Under the original assumption that the $\Rn{(c+1)i+j}$-module $N$
does not satisfy \eqref{charL}, we have shown that every irreducible
$\Rn{(c+1)i+j}$-module $A$ satisfies
\begin{equation}
  A \cong \Ind L(i) \boxtimes B \cong \Ind B \boxtimes L(i)
\end{equation}
for some irreducible $\Rn{ci+j}$-module $B$, and furthermore we have computed $\chr(A)$.

On closer examination of these characters, we see
\begin{equation}
  \sum_{s=0}^{c+1}(-1)^s e_i^{(c+1-s)}e_je_i^{(s)}[A] =0
\end{equation}
for all such $A$.
But then an argument similar to that in Remark~\ref{rem_zero} shows
\begin{equation}
  \sum_{s=0}^{c+1}(-1)^s e_i^{(c+1-s)}e_je_i^{(s)}[C] =0
\end{equation}
for all irreducible $\Rn{\nu}$-modules $C$ for any $\nu \in \N[I]$.
So by \eqref{Serrec}, \eqref{Serrea} we would get $a\leq c$,
contradicting $c+1\leq a$.

So it must be that all irreducible $\Rn{(c+1)i+j}$-modules satisfy
\eqref{Ln}, \eqref{Lcn}, and \eqref{charL}.
\end{proof}

In the previous theorem we introduced the notation $\Lj{i^{c-n}ji^{n}}$ for the unique
 simple $\Rn{ci+j}$-module with $\epsilon_i=n$ when $c\leq a$.
 Theorem~\ref{thm_circleL} below extends this uniqueness to $c \ge a$.
Recall that in the special case that $c=a$, we denote
$$\Lc{n} =\Lj{i^{a-n}ji^{n}}.$$
The following theorem motivates why we distinguish the special case $c=a$.

\begin{thm}\label{thm_circleL}
 Let $0 \leq n \leq a$.
\begin{enumerate}[i)]
\item \label{circleL1}The module
\begin{equation} \label{icommutesL}
  \Ind L(i^m) \boxtimes \Lc{n} \cong \Ind \Lc{n} \boxtimes L(i^m)
\end{equation}
is irreducible for all $m \geq 0$.

\item \label{circleL2}Let $c \geq a$.  Let $N$ be an irreducible $\Rn{ci+j}$-module with $\ep{N}=n$. Then $c-a \leq n \leq c$ and
\begin{equation}
  N \cong \Ind \Lc{n-(c-a)}\boxtimes L(i^{c-a}).
\end{equation}
\end{enumerate}
\end{thm}

\begin{proof}
We first prove \eqref{icommutesL} for $m=1$, from which it will
follow for all $m$ by the Jump Lemma~\ref{lem_jump_tfae}.  Let
$M=\ft{i}\Lc{n} = \cosoc \Ind \Lc{n} \boxtimes L(i)$, which is
irreducible. Note $\ep{M}=n+1$ and by the Shuffle Lemma
\begin{equation}  \label{Serre1}
e_i^{(a-n)}e_je_i^{(n+1)}[M] \neq 0
\end{equation}
but
\begin{equation}  \label{Serre2}
e_i^{(a+1-s)}e_je_i^{(s)}[M] = 0
\end{equation}
unless $s=n+1$ or $s=n$.  But the Serre relations \eqref{Serrea} imply the following operator is identically zero:
\begin{equation}
\sum_{s=0}^{a+1}(-1)^s e_i^{(a+1-s)}e_je_i^{(s)} = 0.
\end{equation}
In particular,
\begin{eqnarray}
0&=&\sum_{s=0}^{a+1}(-1)^s e_i^{(a+1-s)}e_je_i^{(s)}[M] \nn \\
&\refequal{\eqref{Serre2}}& (-1)^n e_i^{(a+1-n)}e_je_i^{(n)}[M]
+ (-1)^{n+1}e_i^{(a-n)}e_je_i^{(n+1)}[M],
\end{eqnarray}
from which we conclude, by \eqref{Serre1}, that
\begin{equation}
  e_i^{(a+1-n)}e_je_i^{(n)}[M] \neq 0.
\end{equation}

This implies
\begin{equation}
a-n+1= \eph{M}=\eph(\ft{i}\Lc{n})=\eph(\Lc{n})+1
\end{equation}
so that by the Jump Lemma
$\ft{i}\Lc{n} \cong \fts{i}\Lc{n}$,  and
consequently part \eqref{circleL1} of the theorem also holds for all $m\geq 1$.
(The case $m=0$ is vacuously true.)


For part \eqref{circleL2}, we induct on $c\geq a$, the case $c=a$ following directly from Theorem~\ref{thm_simplec}.  Now assume the statement for general $c>a$ and consider an irreducible $\Rn{(c+1)i+j}$-module $N$ such that $\ep{N}=n$. If $n=0$, then clearly $e_i^{(c+1)}e_j[N] \neq 0$ so also $e_{i}^{(a+1)}e_j[N] \neq 0$, which by the Serre relations \eqref{Serrea} implies there exists an $n'\neq 0$ with $e_{i}^{(a+1-n')}e_je_{i}^{(n')}[N]\neq0$.  But then $\ep{N} \geq n' >0$, which is a contradiction.

Let $M \cong \et{i}N \neq \0$, so that $\ep{M}=n-1$ and by the inductive hypothesis $$M \cong \Ind \Lc{n-1-(c-a)}\boxtimes L(i^{c-a}).$$  Hence,
by part \eqref{circleL1} and the Jump Lemma 
\begin{equation}
  N \cong \ft{i}M \cong \Ind \Lc{n-((c+1)-a)}\boxtimes L(i^{c+1-a}).
\end{equation}
Consequently $n\geq c+1-a$.  As $N$ is an irreducible $\Rn{(c+1)i+j}$-module, clearly $c+1 \geq n$.
\end{proof}

Observe that from Theorems~\ref{thm_simplec}, \ref{thm_circleL} and the Shuffle Lemma, we have computed the character (up to grading shift) of all irreducible $\Rn{ci+j}$-modules.

%
\subsubsection{A generators and relations proof} \label{sec_genrel}
%
In this section, we give alternative proofs of the Structure
Theorems~\ref{thm_simplec} and \ref{thm_circleL} using the
description of $\Rn{\nu}$ via generators and relations.  In
particular, we do not use the Serre relations \eqref{Serrea} and in
fact one could instead deduce that the Serre relations hold from
these theorems.

We first set up some useful notation.  For this section let
\[
\xy (0,0)*{\ii(b,c)= \underbrace{i \dots i}j\underbrace{i \dots i}};
 (2.5,-5)*{b}; (13.5,-5)*{c};
\endxy
\]
Let $\{ u_r \mid 1 \leq r \leq m!\}$ be a (weight) basis of $L(i^m)$, $\{y_s \mid 1 \leq s
\le n!  \}$ be a basis of $L(i^n)$, and $\{v\}$ be a basis of $L(j)$.
Recall the following fact about the irreducible module $L(i^m)$. For any $u \in L(i^m)$
\begin{equation} \label{eq_nilpotent}
  x_{r}^k u =0,
\end{equation}
for all $k \geq m$, and  $1 \leq r \leq m$.  Further if $u \neq 0$
then $L(i^m)=\Rn{mi}u$, and $1_{\jj} u =0$ if $\jj \neq i^m$. Also
there exists $\tilde{u} \in L(i^m)$ such that $x_r^{m-1} \tilde{u}
\neq 0$ for all $r$.  (We note that it is from these properties we
may deduce Proposition~\ref{prop_pr}.)

The induced module $\Ind L(i^m) \boxtimes L(j) \boxtimes L(i^n)$ has a weight basis
\begin{equation}
  B = \{\psi_{\hat{w}} \otimes (u_r \otimes v \otimes y_s) \mid
 1 \leq r \leq m!, \; 1 \leq s \leq n!, \; w \in S_{m+1+n}/ {S_m\times S_1 \times S_n} \}
\end{equation}
as in Remark~\ref{rem_indbasis}.

\begin{prop}
 Let $K = {\rm span} \{ \psi_{\hat{w}} \otimes (u_r \otimes v \otimes y_s) \in  B \mid \ell(w) \neq 0\}$.  Suppose $c=m+n \leq a$.  Then
\begin{enumerate}
  \item $K$ is a proper submodule of $\Ind L(i^m) \boxtimes L(j) \boxtimes L(i^n)$.
  \item \label{prop_newsimplep2}
The quotient module $\Ind L(i^m) \boxtimes L(j) \boxtimes L(i^n)/K$
is irreducible with character $[m]_i^![n]_i^! i^mji^n.$
\end{enumerate}
\end{prop}

\begin{proof}
It suffices to show
\begin{equation} \label{hpsiK}
h \psi_{\hat{w}} \otimes (u_r \otimes v \otimes y_s) \in K
\end{equation}
where $\ell(w)>0$ as $h$ ranges over the generators $1_{\jj}$, $x_{r}$, $\psi_{r}$ of $\Rn{\nu}$.

Considering the relations in Section~\ref{sec_defnRnu}, $h \psi_{\hat{w}} \otimes (u_r \otimes v \otimes y_s)$ is 0 or a sum of terms of the form $\psi_{\hat{w'}} \otimes (u' \otimes v \otimes y')$ with $\ell(w') \geq \ell(w)-2$, so in other words, we reduce to the case $\ell(w) = 1$ or $\ell(w)=2$ (or else the terms are obviously in $K$).  In fact, it is only in considering relation \eqref{relation_cubic} we examine $\ell(w)=2$, and otherwise we examine $\ell(w)=1$.

To make this reduction valid, we first examine the case $h = x_t$.
Let $\ii = \ii(m,n)$.
We first observe that for $w \in S_{m+1+n}/ S_m \times S_1 \times S_n$,
$w(m+1) = r+1$ if and only if $w(\ii) = \ii(r, c-r)$.
In this case, we can factor $w = \tau \gamma$ with $\ell(w) = \ell(\tau)+\ell(\gamma)$
where $\gamma$ is minimal such that $\gamma(\ii) = \ii(r,c-r)$.
In particular $\gamma = s_{r+1} \cdots s_{m-1} s_m$ or $\gamma = s_{r} \cdots s_{m+2}
s_{m+1}$, which has length $|m-r|$.
By relation \eqref{relation_xpsi}
$$x_t \psi_{\hat{w}} 1_{\ii}=  1_{{\ii}(r, c-r)} x_t \psi_{\hat{w}}
= 1_{{\ii}(r, c-r)} \psi_{\hat{w}} x_{w^{-1}(t)}
+  1_{{\ii}(r, c-r)} \sum \psi_{i_1} \cdots \psi_{i_k}
$$
where the sum is over some subset of (not necessarily reduced)
 subwords $s_{i_1} \cdots s_{i_k}$ of $\hat{w}$,
all satisfying that if $z = s_{i_1} \cdots s_{i_k}$  then $z(m+1) = r+1$.
In particular $\ell(z) \ge |m-r|$.
This shows \eqref{hpsiK} holds for $h = x_t$ when $\ell(w) > 0$.

For $h=1_{\jj}$,  either $h \psi_{\hat{w}} 1_{\ii(m,n)} = 0$ or $h
\psi_{\hat{w}} 1_{\ii(m,n)} =\psi_{\hat{w}} 1_{\ii(m,n)}$,
so clearly \eqref{hpsiK} holds. 

For $h = \psi_b$,
when employing relation \eqref{relation_cubic}, we see some terms in
$h \psi_{\hat{w}} 1_{{\ii}}$ may involve terms of the form $f(x_1, \ldots, x_{c+1})
\psi_{\hat{w'}}$ with $\ell(w') = \ell(w) - 2$.
However from the case completed above regarding relation
\eqref{relation_xpsi}, these terms still have length $>0$ so long as $\ell(w') > 0$.
In other words, we need only consider  the case $\ell(w)=2$,
for which either $w=s_{m \pm 1}s_m$ or $w=s_{m+1 \pm 1}s_{m+1}$.
 However,  the only cases that are
potentially ``length-decreasing" by 2 are for $w = s_{m+1} s_m$ and $h = \psi_m$,
or $w = s_m s_{m+1}$ and $h = \psi_{m+1}$, for which we compute
\begin{equation}
 (\psi_{m }\psi_{m+1 } \psi_{m } - \psi_{m+1 }\psi_{m } \psi_{m+1 }) 1_{\ii}
 =
\sum_{k=0}^{a+1} x_{m}^k x_{m+2}^{a+1-k} 1_{\ii}.
\end{equation}
By \eqref{eq_nilpotent}
\begin{eqnarray}
  x_{m}^k x_{m+2}^{a+1-k} \otimes
(u \otimes v \otimes y) = 1_{\ii} \otimes
(x_{m}^ku) \otimes v \otimes (x_{1}^{a+1-k}y) = 0
\end{eqnarray}
since either $k \geq m$ or $a+1-k > a+1-m \geq n$ as we assumed $m+n
\leq a$. This yields
$$\psi_{m }\psi_{m+1 } \psi_{m } \otimes (u \otimes v \otimes y) =
\psi_{m+1 }\psi_{m } \psi_{m+1 } \otimes (u \otimes v \otimes y).$$
In fact, we also have $\psi_{m}\psi_{m-1}\psi_{m} \otimes (u \otimes
v \otimes y) = \psi_{m-1}\psi_{m}\psi_{m-1}\otimes (u \otimes v
\otimes y)$, as for instance $\ii_{m-1} \neq \ii_{m+1}$,
and similarly $\psi_{m+1}\psi_{m+2}\psi_{m+1} \otimes (u \otimes v
\otimes y) = \psi_{m+2}\psi_{m+1}\psi_{m+2} \otimes (u \otimes v
\otimes y)$. Thus in all cases, this braid relation honestly holds.
This then reduces us to the case $\ell(w)=1$ as such relations
decrease length by at most 1. For example,
\begin{equation}
  \psi_{m}\psi_{m-1}\psi_{m} \otimes (u \otimes v \otimes y) =
  \psi_{m-1}\psi_{m}\psi_{m-1} \otimes (u \otimes v \otimes y)
  = \psi_{m-1}\psi_{m} \otimes (u' \otimes v \otimes y).
\end{equation}

When $\ell(w)=1$ either $w=s_m$ or $w=s_{m+1}$.
 For $h=\psi_{b}$ the only remaining relation that is length decreasing is
\eqref{relation_quadratic} (which decreases length by at most one,
when $b = m$ or $m+1$), for which  we compute
\begin{eqnarray}
  \psi_{m}\psi_{m} \otimes (u \otimes v \otimes y) &=& (x_{m}^a+x_{m+1}^{-\langle j,i\rangle})1_{\ii} \otimes (u \otimes v \otimes y) \nn \\ &=&
1_{\ii} \otimes (x_{m}^a u) \otimes v \otimes y + 1_{\ii} \otimes u \otimes (x_{1}^{-\langle j,i\rangle} v) \otimes y \nn  \\
&=& 0 \in K
\end{eqnarray}
by \eqref{eq_nilpotent} since $a \geq m$, and $-\langle j, i \rangle
\geq 1$. Similarly,
\begin{align}
  \psi_{m+1}\psi_{m+1} \otimes (u \otimes v \otimes y)  &= 1_{\ii} \otimes u \otimes (x_{1}^{- \langle j,i\rangle }v) \otimes y
+ 1_{\ii} \otimes u \otimes v \otimes (x_{1}^ay)\nn\\ &= 0 \in K
\end{align}
as $a \geq n$.

In conclusion, $K$ is indeed a submodule and in fact generated by
\begin{equation}
  \psi_{m+1} \otimes (u_r \otimes v \otimes y_s), \qquad {\rm and} \quad \psi_{m} \otimes (u_r \otimes v \otimes y_s).
\end{equation}

For part \eqref{prop_newsimplep2} note $w(\ii)=\ii(c-r,r)$ for
some $r$, but $r\neq n$ when $\ell(w)>0$ for minimal length $w \in
S_{m+1+n}/S_{m} \times S_1 \times S_{n}$. In other words,
$\psi_{\hat{w}} \otimes (u_r \otimes v \otimes y_s)$ is a weight
vector and $1_{\ii} \psi_{\hat{w}} \otimes (u_r \otimes v
\otimes y_s)=0$ when $\ell(w)>0$. That is, for all $z \in Q = \Ind
L(i^m) \boxtimes L(j) \boxtimes L(i^n)/K$,
 $1_{\ii}z=z$, but $1_{\ii(c-r,r)}z=0$ when $r \neq n$.  Hence all constituents of $\chr(Q)$ have the form $i^{m}ji^n$.

By Frobenius reciprocity, and the irreducibility of $L(i^m)$, we
have an injection
\begin{equation}
  L(i^m) \boxtimes L(j) \boxtimes L(i^n)\hookrightarrow \Res_{mi,j,ni}Q
\end{equation}
which is also a surjection by the above arguments.  Hence
\begin{equation}
\chr(Q)= [m]_i^! [n]_i^! i^m j i^n.
\end{equation}
Note that, up to grading shift, $Q$ is none other than
$\Lj{i^{m}ji^n}$ and we have shown this is the unique simple
quotient of $\Ind L(i^m) \boxtimes L(j) \boxtimes L(i^n)$. The
uniqueness statements of Theorem~\ref{thm_simplec} follow by
Frobenius reciprocity.
\end{proof}

Next we will give a generators and relations proof that
\begin{equation}
  \ft{i} \Lc{n} \cong \fts{i} \Lc{n} \cong \Ind \Lc{n} \boxtimes L(i).
\end{equation}
Just as in the proof of Theorem~\ref{thm_simplec},
\begin{equation}
  \chr(\Ind \Lc{n} \boxtimes L(i)) = [a-n]_i^![n+1]_i^! i^{a-n}ji^{n+1}
+ q^{-(\alpha_i,\alpha_j)}[a-n+1]_i^! [n]_i^! i^{a+n+1} j i^n,
\end{equation}
and since $L(i^m)$ is irreducible with dimension $m!$, either $\chr( \ft{i}\Lc{n})=[a-n]_i^![n+1]_i^! i^{a-n}ji^{n+1}$ or
$\chr(\ft{i} \Lc{n}) = \chr(\Ind \Lc{n} \boxtimes L(i))$.

In the latter case, $\Ind \Lc{n} \boxtimes L(i)$ is isomorphic to
$\ft{i}\Lc{n}$, so by the Jump Lemma~\ref{lem_jump_tfae} it is irreducible
and isomorphic to $\fts{i} \Lc{n}$.
In the former case, we clearly have
\begin{equation}
  0 \to K \to \Ind L(i^{a-n}) \boxtimes L(j) \boxtimes L(i^{n+1}) \to \ft{i} \Lc{n}
\end{equation}
by Frobenius reciprocity.

The $\Rn{(a+1)i + j}$-module $\Ind L(i^{a-n}) \boxtimes L(j)
\boxtimes L(i^{n+1})$ has a weight basis given by
\begin{equation}
  \{ \psi_{\hat{w}} \otimes (u_r \otimes v \otimes y_s) \mid w \in S_{a+2}/S_{a-n}\times S_1 \times S_{n+1}, \quad 1 \leq r \leq (a-n)!, \; 1 \leq s \leq (n+1)! \}.
\end{equation}
Let $\ii=\ii(a-n,n+1)$.
Note, for all minimal left coset representatives $w \in
S_{a+2}/S_{a-n}\times S_1 \times S_{n+1}$ that $w(\ii)\neq \ii$
unless $w = \id$, i.e. unless $\ell(w)=0$. (In fact
$w(\ii)=\ii(a-r+1,r)$ for some $r$.)  Since
$1_{\ii(a-r+1,r)}\ft{i}\Lc{n}=0$ if $r \neq n+1$ by assumption, we
must have
\begin{equation}\label{uvy1}
  K = {\rm span} \{ \psi_{\hat{w}}\otimes(u_r \otimes v \otimes y_s) \mid \ell(w) >0\}.
\end{equation}
We will show that $K$ is not a proper submodule.

Pick $u \in L(i^{a-n})$, $y \in L(i^{n+1})$ so that $x_{a-n}^{a-n-1}u = u' \neq 0$, $x_{1}^ny = y' \neq 0$ so that
\begin{equation} \label{uvy2}
  x_{a-n}^{a-n-1} \cdot x_{a-n+2}^{n}\left( 1_{\ii} \otimes (u \otimes v \otimes y) \right) = 1_{\ii} \otimes (u' \otimes v \otimes y') \neq 0,
\end{equation}
but
\begin{equation}
  x_{a-n}^{a-1-k} u = 0 \qquad \text{ if $k < n$}
\end{equation}
and
\begin{equation}\label{uvy3}
  x_{1}^k y =0 \qquad \text{if $k>n$}.
\end{equation}
Also recall $u'$ generates $L(i^{a-n})$ and $y'$ generates $L(i^{n+1})$ so $1_{\ii} \otimes (u' \otimes v \otimes y')$ generates the module  $\Ind L(i^{a-n})\boxtimes L(j) \boxtimes L(i^{n+1})$.  By assumption, $K \ni \psi_{a-n+1}\otimes (u \otimes v \otimes y)$ and $K \ni \psi_{a-n} \otimes (u \otimes v \otimes y)$.

If $K$ is a $\Rn{(a+1)i+j}$-submodule, $K$ also contains
\begin{eqnarray}
  \left( \psi_{a-n+1} \psi_{a-n}\psi_{a-n+1}
- \psi_{a-n}\psi_{a-n+1}\psi_{a-n} \right) \otimes (u \otimes v
\otimes y) \hspace{1.7in}\nn \\
  \refequal{\eqref{relation_cubic}}
 \left(
\sum_{k=0}^{a-1}x_{a-n}^{a-1-k} x_{a-n+2}^k
\right)\otimes (u \otimes v \otimes y)
\refequal{\eqref{uvy1},\eqref{uvy2},\eqref{uvy3}} 0 +
1_{\ii}\otimes(u' \otimes v \otimes y') \neq 0. \nn
\end{eqnarray}
Therefore $K \ni 1_{\ii} \otimes (u' \otimes v \otimes y')$, hence
$K$ contains all of $\Ind L(i^{a-n}) \boxtimes L(j) \boxtimes
L(i^{n+1})$ contradicting that $K$ is a proper submodule.  We
must have $\ft{i}\Lc{n} \cong \Ind \Lc{n} \boxtimes L(i)$.  Now
\eqref{icommutesL} in Theorem~\ref{thm_circleL} follows for general
$m$ from the $m=1$ case as before.

Note that the Structure Theorems do not depend on the characteristic of $\Bbbk$.
Just as the dimensions of  simple $\Rn{mi}$-modules are independent of ${\mathrm {char}} \Bbbk$, so are the dimensions of simple $\Rn{ci+j}$-modules. In fact, Kleshchev and Ram have conjectured~\cite{KR} that the dimensions of all simple $\Rnu$-modules are independent of ${\mathrm {char}} \Bbbk$ for finite Cartan datum.

%
\subsection[Understanding $\phi^{\Lambda}$]{Understanding $\phiL$} \label{sec_phi-lambda}
%

The main theorems in this section measure how the crystal data differs for $M$ and $\ft{j} M$. In particular, Theorem~\ref{thm_epsilonifj} below is equivalent to
\begin{equation}
  \phiL(\ft{j}M) - \ep{\ft{j}M} = a + (\phiL(M)-\ep{M})
\end{equation}
where $a= -\langle h_i, \alpha_j \rangle$.

First we introduce several lemmas that will be needed.

\begin{lem}\label{lem_epsilonpastcircle}
Suppose $c+d \leq a$.
\begin{enumerate}[i)]
\item \label{Li} $\Ind \Lj{i^cji^d} \boxtimes L(i^m)$ has irreducible cosocle equal to
\begin{equation}
\ft{i}^m \Lj{i^cji^d}=\ft{i}^{m+d}\cal{L}(i^cj) = \left\{
\begin{array}{ccc}
  \Ind \Lc{a-c}\boxtimes L(i^{m-a+c+d}) & \quad & m \geq a-(c+d) \\
  \Lj{i^cji^{d+m}} & \quad & m < a-(c+d).
\end{array}
\right.
\end{equation}

\item \label{computeepsiloncircle} Suppose there is a nonzero map
\begin{equation}
  \Ind \Lc{c_1} \boxtimes \Lc{c_2}  \boxtimes \dots \boxtimes \Lc{c_r}  \boxtimes L(i^m) \longrightarrow Q
\end{equation}
where $Q$ is irreducible.  Then $\ep{Q}=m+\sum_{t=1}^{r}c_t$ and $\eph(Q)=m+\sum_{t=1}^r(a-c_t)$.

\item \label{epcircle3} Let $B$ and $Q$ be irreducible and suppose there is a nonzero map $\Ind B \boxtimes \cal{L}(c) \to Q$.  Then $\ep{Q}=\ep{B}+c$.
\end{enumerate}
\end{lem}

\begin{proof}
Part \eqref{Li} follows from the Structure Theorems
\ref{thm_simplec}, \ref{thm_circleL}  for  irreducible
$\Rn{(c+d+m)i+j}$-modules.  For part \eqref{computeepsiloncircle}
recall $\Ind \Lc{c} \boxtimes L(i^m)$ is irreducible and is
isomorphic to $\Ind L(i^m) \boxtimes \Lc{c}$
by Part~\eqref{circleL1} of Theorem~\ref{thm_circleL}.
Consider the chain of
homogeneous surjections
\begin{equation}
\xy (0,12)*+{ \Ind \Lj{i^{a-c_1}j} \boxtimes \Lc{c_2} \boxtimes
\dots \boxtimes \Lc{c_r}\boxtimes L(i^{c_1+m})}="1"; (0,0)*+{\Ind
\cal{L}(i^{a-c_1}j) \boxtimes L(i^{c_1})\boxtimes\cal{L}(c_2)
\boxtimes \dots \boxtimes \cal{L}(c_r)\boxtimes L(i^{m})}="2";
(0,-12)*+{\Ind \Lc{c_1} \boxtimes\cal{L}(c_2) \boxtimes \dots
\boxtimes \cal{L}(c_r)\boxtimes L(i^{m})}="3"; (0,-24)*+{Q}="4";
 {\ar^{\iso} "1";"2"};
 {\ar@{->>} "2";"3"};
 {\ar@{->>} "3";"4"};
\endxy
\end{equation}
Iterating this process we get a surjection
\begin{equation}
  \Ind \cal{L}(i^{a-c_1}j)\boxtimes \cal{L}(i^{a-c_2}j)\boxtimes \dots \boxtimes
  \cal{L}(i^{a-c_r}j)\boxtimes L(i^{h}) \twoheadrightarrow Q
\end{equation}
where $h=m+\sum_{t=1}^rc_t$.  This shows that $\ep{Q}=m+\sum_{t=1}^rc_t$.  The computation of $\eph(Q)$ is similar.

For part \eqref{epcircle3} let $b=\ep{B}$.  By the Shuffle Lemma $\ep{Q}\leq b+ c$.  Further there exists an irreducible module $C$ such that $\ep{C}=0$ and $\Ind C \boxtimes L(i^b)\twoheadrightarrow B$.  By the exactness of induction, we have a surjection
\begin{equation}
  \xy
   (-20,0)*+{\Ind C \boxtimes \cal{L}(c)\boxtimes L(i^b) \cong \Ind C \boxtimes L(i^b) \boxtimes \cal{L}(c)}="1";
   (28,0)*+{Q}="3";{\ar@{->>} "1";"3"};
  \endxy
\end{equation}
and so by Frobenius reciprocity $\ep{Q}\geq \ep{\cal{L}(c)}+\ep{L(i^b)}=c+b$.
\end{proof}

\begin{lem}\label{lem_Nj}
Let $N$ be an irreducible $\Rn{ci+dj}$-module with $\ep{N}=0$.
Suppose $c+d > 0$.
\begin{enumerate}[i)]
  \item \label{Nbar} There exists irreducible $\overline{N}$ with $\ep{\overline{N}}=0$ and a surjection
  \begin{equation}
    \Ind \overline{N}\boxtimes \cal{L}(i^bj) \twoheadrightarrow N
  \end{equation}
with $b \leq a$.

 \item \label{btN}There exists an $r\in \N$ and $b_t \leq a$ for $1\leq t \leq r$ such that
 \begin{equation}
   \Ind \cal{L}(i^{b_1}j) \boxtimes \cal{L}(i^{b_2}j) \boxtimes \dots \boxtimes \cal{L}(i^{b_r}j) \twoheadrightarrow N.
 \end{equation}
\end{enumerate}
\end{lem}

\begin{proof}
First, we may assume $\et{j}N \neq \0$ or else $N$ would be the trivial module
$\1$, i.e.\ $c=d=0$.
Let $b=\ep{\et{j}N}$ and let $\overline{N}=\et{i}^b\et{j}N$ so that $\ep{\overline{N}}=0$.  There exists a surjection
\begin{equation}
  \Ind \overline{N}\boxtimes L(i^b) \boxtimes L(j) \twoheadrightarrow N.
\end{equation}
Recall $\ep{N}=0$ and by the Structure Theorems, $\Ind
L(i^b)\boxtimes L(j)$ has at most one composition factor with
$\epsilon_i=0$, namely $\cal{L}(i^bj)$ in the case $b \leq a$.  In
the case $b>a$ it has no such composition factors, contradicting
$\ep{N}=0$. Hence $b \le a$ and the above map must factor through
\begin{equation}
  \Ind \overline{N} \boxtimes \cal{L}(i^bj) \twoheadrightarrow N.
\end{equation}

For part \eqref{btN} we merely repeat the argument from part
\eqref{Nbar} using the exactness of induction.
\end{proof}

\begin{lem} \label{lem_iterate}
Suppose $Q$ is irreducible and we have a surjection
\begin{equation}
  \Ind \cal{L}(i^{b_1}j)\boxtimes \cal{L}(i^{b_2}j) \boxtimes \dots \boxtimes \cal{L}(i^{b_r}j) \boxtimes L(i^h) \twoheadrightarrow Q.
\end{equation}
\begin{enumerate}[i)]
  \item \label{iterate1} Then for $h \gg 0$ we have a surjection
\begin{equation}
  \Ind \cal{L}(a-b_1) \boxtimes \cal{L}(a-b_2) \boxtimes \dots \boxtimes \cal{L}(a-b_r) \boxtimes L(i^g) \twoheadrightarrow Q
\end{equation}
where $g=h-\sum_{t=1}^r(a-b_t)$.

\item \label{iterate2} In the case $h < ar-\sum_{t=1}^rb_t$, we have
\begin{equation}
\Ind \cal{L}(i^{b_1}j)\boxtimes \dots \boxtimes \cal{L}(i^{b_{s-1}}j)\boxtimes \cal{L}(i^{b_s}ji^{g'})\boxtimes \cal{L}(a-b_{s+1})\boxtimes \dots \boxtimes \cal{L}(a-b_r)
\twoheadrightarrow Q
\end{equation}
\end{enumerate}
where $g'=h-\sum_{t=s+1}^r(a-b_t)$ and $s$ is such that
\begin{equation}
  \sum_{t=s+1}^r(a-b_t) \leq h < \sum_{t=s}^r(a-b_t).
\end{equation}
\end{lem}

\begin{proof}
Observe that $\ep{Q}=h$.  Similar to
Lemma~\ref{lem_epsilonpastcircle}~\eqref{Li} when $d=0$, $\Ind
\Lj{i^{b_r}j} \boxtimes L(i^h)$ has a unique composition factor with
$\epsilon_i=h$, namely $\Ind L(i^{h-(a-b_r)})\boxtimes \Lc{a-b_r}$
in the case $h \geq a-b_r$ and $\Lj{i^{b_r}ji^h}$ otherwise.
In the latter case, we are done, and note we fall into case \eqref{iterate2} with $s=r$.  In the former case, we get a surjection
\begin{equation}
\Ind \Lj{i^{b_1}j} \boxtimes \dots \boxtimes \Lj{i^{b_{r-1}}j} \boxtimes L(i^{h-(a-b_r)})\boxtimes \Lc{a-b_r} \twoheadrightarrow Q.
\end{equation}
We apply the same reasoning to $\Ind \Lj{i^{b_{r-1}}j} \boxtimes
L(i^{h-(a-b_r)})$ noting that by
Lemma~\ref{lem_epsilonpastcircle}~\eqref{epcircle3}, since
$\ep{\Lc{a-b_r}}=a-b_r=\ep{Q}-(h-(a-b_r))$ we want to pick out the
unique composition factor with $\epsilon_i=h-(a-b_r)$.   As above,
this is $\Ind L(i^{h-\sum_{t=r-1}^rb_t}) \boxtimes \Lc{a-{b_{r-1}}}$
for $h$ large enough and $\Lj{i^{b_{r-1}}ji^{h-(a-b_r)}}$ otherwise.
Continuing in this vein the lemma follows.
\end{proof}

\begin{lem}\label{lem_arguement_star}
Let $M$ be an irreducible $\Rn{\nu}$-module and suppose we have a nonzero map
\begin{equation}
 \xy
 {\ar@{->>}^-f (-20,0)*+{\Ind A \boxtimes B \boxtimes L(i^h)}; (10,0)*+{M}}
 \endxy
\end{equation}
where $\ep{A}=0$ and $B$ is irreducible.  Then there exists a surjective map
\begin{equation}
 \xy
 {\ar@{->>} (-14,0)*+{\Ind A \boxtimes \ft{i}^hB}; (10,0)*+{M}}.
 \endxy
\end{equation}
\end{lem}

\begin{proof}
First note $\ep{M}=\ep{B}+h$ since by Frobenius reciprocity $\ep{M}
\geq \ep{B}+h$, but by the Shuffle Lemma $\ep{M}\leq \ep{B}+h$ since
$\ep{A}=0$.  Consider $\Ind B \boxtimes L(i^h)$.  This has  unique
irreducible quotient $\ft{i}^hB$ with $\ep{\ft{i}^hB}=\ep{B}+h$ and
has all other composition factors $U$ with $\ep{U} <
\ep{B}+h=\ep{M}$, by Section~\ref{sec_properties}. Hence, for any such $U$
there does not exist a nonzero map $\Ind A \boxtimes U \to M$.  In
particular, letting $K$ be the maximal submodule such that
\begin{equation}
  \xymatrix{
  0 \ar[r] & K \ar[r] & \Ind B \boxtimes L(i^h) \ar[r] & \ft{i}^hB \ar[r] & 0
  }
\end{equation}
is exact, the above map $f$ must restrict to zero on the submodule $\Ind A \boxtimes K$ and hence $f$ factors through $\Ind A \boxtimes \ft{i}^hB  \twoheadrightarrow M$, which is
nonzero and thus surjective.
\end{proof}

\begin{lem} \label{lem_pr0}
Let $A$ be an irreducible $\Rn{\nu}$-module with $\prL A \neq \0$ and $k=\phiL(A)$.
\begin{enumerate}[i)]
\item \label{pr01}  Let $U$ be an irreducible $R(\mu)$-module
and let $t\geq 1$.  Then $\prL \Ind A \boxtimes L(i^{k+t})\boxtimes U
=\0$.

\item \label{pr02} Let $B$ be irreducible with $\eph(B) >k$.  Then $\prL \Ind A \boxtimes B=\0$.  In particular, if $Q$ is any irreducible quotient of $\Ind A \boxtimes B$, then $\prL Q=\0$.
\end{enumerate}
\end{lem}

\begin{proof}
Recall for a module $B$, $\prL B = B/\cal{J}^{\Lambda}B$ and so
$\prL B=\0$ if and only if $B=\cal{J}^{\Lambda}B$.  Since $A$,
$L(i^{k+t})$, and $U$ are all irreducible, each is
generated by any single nonzero element.  Let us pick nonzero $w\in A$, $v
\in L(i^{k+t})$, $u \in U$.  Further $\Ind A \boxtimes L(i^{k+t})$
is cyclically generated as an $\Rn{\nu+(k+t)i}$-module by
$1_{\nu+(k+t)i}\otimes w \otimes v$ and likewise $\Ind A \boxtimes
L(i^{k+t})\boxtimes U$ is generated as an
$\Rn{\nu+(k+t)i+\mu}$-module by $1_{\nu+(k+t)i+\mu}\otimes w \otimes
v \otimes u$.

Recall that $\Ind A \boxtimes L(i^{k+t})$ has a unique simple quotient $\ft{i}^{k+t}A$ and that $\prL \ft{i}^{k+t}A =\0$ because $\phiL(A)=k$.  Since $\prL$ is right exact, $\prL \Ind A \boxtimes L(i^{k+t})=\0$.  Consequently $\cal{J}^{\Lambda}_{\nu+(k+t)i}\Ind A \boxtimes L(i^{k+t})=\Ind A \boxtimes L(i^{k+t})$.  In particular, there exists an $\eta \in \cal{J}^{\Lambda}_{\nu+(k+t)i}$ such that
\begin{equation}
  \eta 1_{\nu+(k+t)i}\otimes w \otimes v =1_{\nu+(k+t)i}\otimes w \otimes v.
\end{equation}
But then
\begin{equation}
 \eta 1_{\nu+(k+t)i+ \mu}\otimes w \otimes v \otimes u = 1_{\nu+(k+t)i+\mu}\otimes w \otimes v \otimes u.
\end{equation}
Note we can consider $\eta$ as an element of $\cal{J}^{\Lambda}_{\nu+(k+t)i+\mu}$ as well via the canonical inclusion $\Rn{\nu+(k+t)i}\hookrightarrow \Rn{\nu+(k+t)i+\mu}$.  Hence
\begin{equation}
  \cal{J}^{\Lambda}_{\nu+(k+t)i+\mu} \Ind A \boxtimes L(i^{k+t}) \boxtimes U = \Ind A \boxtimes L(i^{k+t}) \boxtimes U
\end{equation}
and so $\prL \Ind A \boxtimes L(i^{k+t}) \boxtimes U =\0$.

For part \eqref{pr02}, let $b= \eph(B)$ and $C=(\ets{i})^b B$ so we have $\Ind L(i^b)\boxtimes C \twoheadrightarrow B$.  Thus by the exactness of induction we also have a surjection $\Ind A \boxtimes L(i^b)\boxtimes C \twoheadrightarrow \Ind A \boxtimes B$.  By part \eqref{pr01} and the right exactness of $\prL$, $\prL \Ind A \boxtimes B =\0$.  Likewise $\prL Q=\0$ for any quotient of $\Ind A \boxtimes B$.
\end{proof}

\begin{lem} \label{lem_pr}
  Let $A$ be an irreducible $\Rn{\nu}$-module with $\prL A \neq \0$ and $k=\phiL(A)$.  Further suppose $\ep{A}=\epsilon_j(A)=0$ and that $B$ is an irreducible $\Rn{ci+dj}$-module with $\eph(B) \leq k$.  Let $Q$ be irreducible such that $\Ind A \boxtimes B \twoheadrightarrow Q$ is nonzero.  Then $\eph(Q) \leq \lambda_i$.  Further, if $\ephj(B)\leq \phi^{\Lambda}_j(A)$ (or if $\lambda_j \gg 0$) then $\prL Q \neq \0$.
\end{lem}

\begin{proof}
Let $b=\eph(B)$ and $C=(\ets{i})^bB$ so that $\eph(C)=0$.  We thus have surjections
\begin{equation}
  \Ind A \boxtimes L(i^b) \boxtimes C \twoheadrightarrow \Ind A \boxtimes B \twoheadrightarrow Q.
\end{equation}
Observe by Frobenius reciprocity
\begin{equation}
  (1_{\nu}\otimes 1_{bi} \otimes 1_{(c-b)i+dj}) Q \neq \0.
\end{equation}
Let $U$ be any composition factor of $\Ind A \boxtimes L(i^b)$ other
than $\ft{i}^bA$, so that $\ep{U}<b$.  By the Shuffle Lemma
$1_{\nu}\otimes 1_{bi} \otimes 1_{(c-b)i+dj}(\Ind U \boxtimes
C)=\0$, so there cannot be a nonzero homomorphism $\Ind U \boxtimes
C\twoheadrightarrow Q$. (More precisely, for every constituent
$\ii=i_1\dots i_{|\nu|+b}$ of $\chr(U)$ there exists a $y$, $|\nu|<y
\leq |\nu|+b$ with $i_y\neq i$ and $i_y\neq j$. Hence by the Shuffle
Lemma, for every constituent $\ii'=i'_1\dots i'_{|\nu|+c+d}$ of
$\chr(\Ind U \boxtimes C)$ there exists a $z$, $|\nu| <z \leq
|\nu|+c+d$ with $i'_z \neq i$ and $i'_z\neq j$.)

Thus we must have a nonzero map
\begin{equation}
  \Ind \ft{i}^bA \boxtimes C \twoheadrightarrow Q.
\end{equation}
By the Shuffle Lemma, $\eph(Q)\leq \eph(\ft{i}^bA)+\eph(C)\leq
\lambda_i$ since $b \leq k = \phiL(A)$ and $\eph(C)=0$.  Note $\ephl(Q) \leq
\ephl(A)+\ephl(B)$, so for $\ell\neq i$, $\ell \neq j$ clearly
$\ephl(Q)\leq \lambda_{\ell}$ and hence $\prL Q \neq \0$ so long as
$\ephj(B)\leq \phi^{\Lambda}_j(A)$, which will for instance be assured if
$\lambda_j \gg 0$.
\end{proof}

In the following theorem and its proof all modules have support $\nu=ci+dj$ for some $c,d\in \N$.

\begin{thm}\label{thm_ij}
Let $M$ be an irreducible $\Rn{ci+dj}$-module and let $\Lambda \in P^+$ be such that $\prL M \neq \0$ and $\prL \ft{j}M \neq \0$.  Let $m=\epsilon_i(M)$, $k=\phi_i^{\Lambda}(M)$ .   Then there exists an $n$ with $0 \leq n \leq a$ such that $\epsilon_i(\ft{j}M)=m-(a-n)$ and $\phi_i^{\Lambda}(\ft{j}M)=k+n$.
\end{thm}

\begin{proof}
Let $N=\et{i}^m M$ so that $\ep{N}=0$ and we have a surjection
\begin{equation}
  \Ind N \boxtimes L(i^m)\twoheadrightarrow M.
\end{equation}
Thus, we also have
\begin{equation}
  \Ind N \boxtimes L(i^m)\boxtimes L(j) \twoheadrightarrow \ft{j}M.
\end{equation}
By the Structure Theorems~\ref{thm_simplec}, ~\ref{thm_circleL} for
simple  $\Rn{mi+j}$-modules, for each $m-a\leq\gamma\leq m$ there
exists a composition factor $U_{\gamma}$ of $\Ind L(i^m)\boxtimes
L(j)$ with $\ep{U_{\gamma}}=\gamma$.  In particular, there is a
unique $\gamma$ such that the above map induces
\begin{equation}
  \Ind N \boxtimes U_{\gamma} \twoheadrightarrow \ft{j}M
\end{equation}
as we must have $\ep{U_\gamma}=\ep{\ft{j}M}$, since $\ep{N}=0$.  Choose $n$ so that $\gamma=m-(a-n)=\ep{\ft{j}M}$.  Note that by the Structure Theorems
\begin{equation}
  U_{\gamma} \cong \left\{
  \begin{array}{ccl}
    \Ind \Lc{n}\boxtimes L(i^{m-a}) & \quad & m \geq a \\
    \Lj{i^{a-n}ji^{m-(a-n)}} &\quad & m < a ,
  \end{array}
  \right.
\end{equation}
and furthermore
\begin{equation} \label{gamma}
\ft{i}^aU_{\gamma}\cong \Ind \Lc{n}\boxtimes L(i^m)
\end{equation}
in both cases.

By Lemma~\ref{lem_Nj} there exist $0 \leq b_t \leq a$ such that
\begin{equation}
  \Ind \Lj{i^{b_1}j}\boxtimes \Lj{i^{b_2}j}\boxtimes \dots \boxtimes \Lj{i^{b_r}j}
  \twoheadrightarrow N
\end{equation}
and hence we obtain the following surjections
\begin{align}
& \Ind \Lj{i^{b_1}j}\boxtimes \Lj{i^{b_2}j}\boxtimes \dots \boxtimes \Lj{i^{b_r}j} \boxtimes L(i^m)
   \twoheadrightarrow M \\
& \Ind \Lj{i^{b_1}j}\boxtimes \Lj{i^{b_2}j}\boxtimes \dots \boxtimes \Lj{i^{b_r}j} \boxtimes L(i^{m+h})
   \twoheadrightarrow \ft{i}^hM \label{ijMh}\\
& \Ind \Lj{i^{b_1}j}\boxtimes \Lj{i^{b_2}j}\boxtimes \dots \boxtimes \Lj{i^{b_r}j} \boxtimes U_{m-a+n}
   \twoheadrightarrow \ft{j}M \\
 & \Ind \Lj{i^{b_1}j}\boxtimes \Lj{i^{b_2}j}\boxtimes \dots \boxtimes \Lj{i^{b_r}j} \boxtimes U_{m-a+n}\boxtimes L(i^h)
   \twoheadrightarrow \ft{i}^h\ft{j}M  \label{ijMhj}
   \end{align}
We first apply Lemma~\ref{lem_iterate} to \eqref{ijMh} to obtain,
for $h \gg 0$ (in fact $h \ge \sum_{t=1}^r(a-b_t)-m$)
\begin{equation}
   \Ind \Lc{a-b_1} \boxtimes \Lc{a-b_2} \boxtimes \dots \boxtimes \Lc{a-b_r}\boxtimes L(i^g)\twoheadrightarrow \ft{i}^hM
\end{equation}
where $g=m+h-\sum_{t=1}^r(a-b_t)$.  Hence, by Lemma~\ref{lem_epsilonpastcircle}~\eqref{computeepsiloncircle}
\begin{equation}
  \eph(\ft{i}^hM) = g+\sum_{t=1}^rb_t=h+m-ar+2\sum_{t=1}^rb_t.
\end{equation}
Further, it is clear that $\eph(\ft{i}^{h+1})=1+\eph(\ft{i}^h(M))$.

Applying Lemma~\ref{lem_iterate} to \eqref{ijMhj} we obtain for $h \gg 0$
\begin{equation}
  \Ind \Lc{a-b_1} \boxtimes \dots \boxtimes \Lc{a-b_r} \boxtimes \Lc{n} \boxtimes
  L(i^m) \boxtimes L(i^{g'})\twoheadrightarrow\ft{i}^h\ft{j}M
\end{equation}
where $g'=h-a-\sum_{t=1}^r(a-b_t)$.
Note we have used \eqref{gamma} above, and in the case $m<a$ we have also employed Lemma~\ref{lem_arguement_star}.  As above, by Lemma~\ref{lem_epsilonpastcircle}~\eqref{computeepsiloncircle}
\begin{align}
  \eph(\ft{i}^h \ft{j}M) &= g'+m+a-n+\sum_{t=1}^rb_t \\
   &= h+m-n-ar+2\sum_{t=1}^rb_t \\
   &=\eph(\ft{i}^hM)-n.
\end{align}
Further, it is clear that $\eph(\ft{i}^{h+1}\ft{j}M)=1+\eph(\ft{i}^h\ft{j}M)$.

For $h \gg 0$ we have shown that $\eph(\ft{i}^h\ft{j}M)=\eph(\ft{i}^hM)-n$.
Now fix such an $h$ and let $\omega_i=h+(m-ar+2\sum_{t=1}^rb_t)$, which we may assume is positive.  Let $\omega_{\ell}=\lambda_{\ell}$ for $\ell \neq i$ and set $\Omega=\sum_{i\in I}\omega_i \Lambda_i \in P^+$.  Given these choices, we have shown
$\eph(\ft{i}^hM) = \omega_i$, but $\eph(\ft{i}^{h+1}M)=\omega_i+1$.  Hence $\phiO(M)=h$.   Likewise $\eph(\ft{i}^h\ft{j}M)=\omega_i-n$, so that $\eph(\ft{i}^{h+n}\ft{j}M)=\omega_i$, but $\eph(\ft{i}^{h+n+1}\ft{j}M)=\omega_i+1$
yielding $\phiO(\ft{j}M)=h+n$.  Observe then that
\begin{equation}
  \phiO(\ft{j}M)-\phiO(M)=n.
\end{equation}

By our hypotheses and the choice of $\Omega$, we know $\prL$ and $\prO$
are nonzero for both modules.  Hence by Remark~\ref{rem_Omega},
$$\phiL(\ft{j}M)-\phiL(M) = \phiO(\ft{j}M)-\phiO(M) = n.
$$
\end{proof}

We have just shown in Theorem~\ref{thm_ij} that Theorem~\ref{thm_epsilonifj} holds for all $\Rn{ci+dj}$-modules.  Next we show that to deduce the theorem for $\Rn{\nu}$-modules for arbitrary $\nu$ it suffices to know the result for $\nu=ci+dj$.

\begin{prop}\label{prop_reduction}
Let $\Lambda \in P^+$ and let $M$ be an irreducible
$\Rn{\nu}$-module such that $\prL M\neq \0$ and $\prL \ft{j}M\neq \0$.
Suppose $\ep{M}=m$ and $\ep{\ft{j}M}=m-(a-n)$ for some $0 \leq n
\leq a$.  Then there exists $c$, $d$ and an irreducible
$\Rn{ci+dj}$-module $B$ such that $\ep{B}=m$, $\ep{\ft{j}B}=m-(a-n)$
and there exists $\Omega \in P^+$ with $\prO(B)\neq \0$,
$\prO(\ft{j}B)\neq \0$, $\prO(M)\neq \0$, $\prO(\ft{j}M) \neq \0$, and
furthermore
\begin{equation}
  \phiO(\ft{j}M)- \phiO(M)= \phiO(\ft{j}B)-\phiO(B).
\end{equation}
\end{prop}

Note that by Remark~\ref{rem_Omega} $\phiL(\ft{j}M)- \phiL(M)=\phiO(\ft{j}M)- \phiO(M)$, so once we prove this proposition, it together with Theorem~\ref{thm_ij} proves Theorem~\ref{thm_epsilonifj}.

\begin{proof}
Let $N=\et{i}^mM$, so that $\ep{N}=0$.  Then there exists
irreducible modules $A$ and $\overline{B}$ with a surjection $\Ind A
\boxtimes \overline{B}\twoheadrightarrow N$ such that
$\ep{A}=\epsilon_j(A)=0$ and $\overline{B}$ is an
$\Rn{\bar{c}i+dj}$-module for some $\bar{c}$, $d$.  (For instance,
one may construct $A$ by setting
\begin{equation}
  A_1 =N, \qquad A_{2r}=\et{j}^{\epsilon_j(A_{2r-1})} A_{2r-1},
\qquad A_{2r+1}=\et{i}^{\ep{A_{2r}}}A_{2r}
\end{equation}
which eventually stabilizes.  So we may set $A=A_r$ for $r\gg 0$.)

Observe, as $\ep{A}=\epsilon_j(A)=0$, we must have
$\ep{\overline{B}}=\ep{N}=0$ and
$\epsilon_j(\overline{B})=\epsilon_j(N)$.  Hence we also have a
surjection
\begin{equation}
  \Ind A \boxtimes \overline{B} \boxtimes L(i^m) \twoheadrightarrow M
\end{equation}
which by Lemma~\ref{lem_arguement_star} produces a map
\begin{equation}\label{ABM}
  \Ind A \boxtimes B \twoheadrightarrow M
\end{equation}
where $B=\ft{i}^m\overline{B}$.
Observe $\ep{B}=\ep{M}=m$.
We have a surjection
\begin{equation}
  \Ind A \boxtimes B \boxtimes L(j) \twoheadrightarrow \ft{j}M
\end{equation}
and since $\epsilon_j(B)=\epsilon_j(M)$,
Lemma~\ref{lem_arguement_star} again produces a map
\begin{equation}\label{AjBjM}
  \Ind A \boxtimes \ft{j}B \twoheadrightarrow \ft{j}M.
\end{equation}

Again observe $\ep{\ft{j}B}= \ep{\ft{j}M}=m-(a-n)$.  From
\eqref{ABM} and \eqref{AjBjM}
 we also have nonzero maps
\begin{equation}
  \Ind A \boxtimes B \boxtimes L(i^h) \twoheadrightarrow \ft{i}^hM, \qquad \quad
  \Ind A \boxtimes \ft{j}B \boxtimes L(i^{h'})\twoheadrightarrow \ft{i}^{h'}\ft{j}M
\end{equation}
so applying Lemma~\ref{lem_arguement_star}, there exist surjections
\begin{equation}
  \Ind A \boxtimes \ft{i}^h B\twoheadrightarrow \ft{i}^hM, \qquad \quad
  \Ind A \boxtimes \ft{i}^{h'}\ft{j}B \twoheadrightarrow \ft{i}^{h'}\ft{j}M. \label{AhBhM}
\end{equation}

Let $\Omega=\sum_{i\in I}\omega_i\Lambda_i \in P^+$ be such that $\omega_{\ell}=\max\{ \lambda_{\ell}, \epsilon^{\vee}_{\ell}B\}$ for all $\ell \in I$.  Recall $B$ is an $\Rn{ci+dj}$-module, where $c = \bar{c}+m$, so for $\ell \neq i,j$, $\epsilon^{\vee}_{\ell}B=0$.
Take $h=\phiO(M)$ and $h'=\phiO(\ft{j}M)$ so that $\prO(\ft{i}^hM)\neq \0$, $\prO(\ft{i}^{h'}\ft{j}M)\neq \0$, but $\prO(\ft{i}^{h+1}M)=\prO(\ft{i}^{h'+1}\ft{j}M)=\0$.

From the contrapositive to Lemma~\ref{lem_pr0}~\eqref{pr02} applied to \eqref{AhBhM} we deduce
\begin{equation}
  \eph(\ft{i}^hB)\leq \phiO(A), \qquad \quad \eph(\ft{i}^{h'}\ft{j}B)\leq \phiO(A).
\end{equation}
However, applying the contrapositive of Lemma~\ref{lem_pr}
\begin{equation}
  \eph(\ft{i}^{h+1}B)> \phiO(A), \qquad \quad \eph(\ft{i}^{h'+1}\ft{j}B)> \phiO(A).
\end{equation}
We thus conclude
\begin{equation}\label{hatOmega}
  \eph(\ft{i}^hB)=\phiO(A) = \eph(\ft{i}^{h'}\ft{j}B)
\end{equation}
and furthermore $\jump_i(\ft{i}^hB)=\jump_i(\ft{i}^{h'}\ft{j}B)=0$.

Recall that
$\phiO(C) = 1+\phiO(\ft{i}C)$ for any irreducible module $C$.  Hence, we compute
\begin{eqnarray}
  \phiO(\ft{j}B)- \phiO(B) &=&
   (h'+ \phiO(\ft{i}^{h'}\ft{j}B))- (h+ \phiO(\ft{i}^hB)) \nn \\
  &=& (h'-h)+ \phiO(\ft{i}^{h'}\ft{j}B)- \phiO(\ft{i}^hB) \nn \\
  &\refequal{{\rm Prop}~\ref{prop_jump_tfae}~\eqref{jumpphi}}&
  (h'-h)+(\jump_i(\ft{i}^{h'}\ft{j}B)-\eph(\ft{i}^{h'}\ft{j}B)+\omega_i) \nn \\
 & &\quad -(\jump_i(\ft{i}^{h}B)-\eph(\ft{i}^{h}B)+\omega_i) \nn \\
  &=& (h'-h)+(0-\phiO(A)+\omega_i)-(0-\phiO(A)+\omega_i) \nn \\
  &=& h'-h \nn \\
  &=& \phiO(\ft{j}M)- \phiO(M). \nn
\end{eqnarray}
\end{proof}

\begin{thm}\label{thm_epsilonifj}
Let $M$ be an irreducible $\Rn{\nu}$-module $\Lambda \in P^+$ such that $\prL M \neq \0$
and $\prL \ft{j}M \neq \0$.  Let $m=\epsilon_i(M)$, $k=\phi_i^{\Lambda}(M)$.   Then there exists an $n$ with $0 \leq n \leq a$ such that $\epsilon_i(\ft{j}M)=m-(a-n)$ and $\phi_i^{\Lambda}(\ft{j}M)=k+n$.
\end{thm}

\begin{proof}
This follows from Theorem~\ref{thm_ij} which proves the theorem in
the case $\nu=ci+dj$ and from Proposition~\ref{prop_reduction} which
reduces it to this case.
\end{proof}

One important rephrasing of the Theorem is
\begin{equation} \label{eq_rephrase}
  \phiL(\ft{j}M) - \ep{\ft{j}M} = a + (\phiL(M)-\ep{M}) =
-\langle h_i, \alpha_j \rangle + (\phiL(M)-\ep{M}).
\end{equation}

\begin{cor} \label{cor_jump}
Let $\Lambda=\sum_{i \in I}\lambda_i \Lambda_i \in P^+$ and
let $M$ an irreducible $\Rnu$-module such that $\prL M \neq \0$.  Then
$$\phiL(M)=\lambda_i+\epsilon_i(M)+\wt_i(M).$$
\end{cor}

\begin{proof}
The proof is by induction on the length $|\nu|$.  For $|\nu|=0$ we
have $M=\1$ and $\wt(M)=0$.  For all $i \in I$ observe that
$\phiL(\1)=\lambda_i$,
$\epsilon_i(\1)=0$,
and $\wt_i(M)=0$, so that the claim clearly holds for $M=\1$.
 Fix $\nu$
with $|\nu| > 0$ and an irreducible $\Rn{\nu}$-module $M$. Let $j\in
I$ be such that $\epsilon_j(M)\neq0$, noting such $j$ exists since
$|\nu| > 0$.

Consider $N=\et{j}M$. By induction we may assume the claim holds for
$N$.  Note $M=\ft{j}N$.  By Theorem \ref{thm_epsilonifj} and its rephrasing \eqref{eq_rephrase}, for any $i \in I$
\begin{align}
\phiL(M) &=\phiL(\ft{j}N)=\phiL(N)+\epsilon_i(\ft{j}N)-\epsilon_i(N)+a_{ij} \nn\\
 &=(\lambda_i+\epsilon_i(N)+\wt_i(N))
+\epsilon_i(\ft{j}N)-\epsilon_i(N)+a_{ij} \nn \\
&= \lambda_i+\epsilon_i(\ft{j}N)+\wt_i(N)- \langle h_i, \alpha_j \rangle \nn \\
&= \lambda_i+\epsilon_i(M)+\wt_i(M). \nn
\end{align}
\end{proof}

Note  that we have finally proved
Proposition~\ref{prop_jump_tfae2}~\eqref{jumpwt}. By
Proposition~\ref{prop_pr}, given an irreducible  module $M$ we can always
take $\Lambda$ large enough so that $\prL M \neq \0$, and then
Proposition~\ref{prop_jump_tfae}~\eqref{jumpphi} combined with the
above corollary gives
\begin{eqnarray}
  \jump_i(M) &=& \phiL(M) + \eph(M) + \lambda_i \nn \\
&=&  (\lambda_i+\ep{M}+\wt_i(M))+\eph(M)-\lambda_i \nn \\
 &=& \ep{M}+\eph(M)+\wt_i(M).
\end{eqnarray}
As mentioned in the discussion below
Proposition~\ref{prop_jump_tfae2}, the $\sigma$-symmetry of
this characterization of $\jump_i(M) $ now implies the remaining
parts \eqref{XX1}, \eqref{XX3} of that proposition. In the next section, we will use
all characterizations of $\jump_i(M) $ from Propositions~\ref{prop_jump_tfae} and \ref{prop_jump_tfae2}.

%
\section{Identification of crystals -- ``Reaping the Harvest"} \label{sec_harvest}
%

Now that we have built up the machinery of Section~\ref{sec_modules},
we can prove the module theoretic crystal $\cal{B}$ is isomorphic
to $B(\infty)$.
Once we have completed this step, it is not much harder to show
$\cal{B}^\Lambda \iso B(\Lambda)$.

While the methods used in Section~\ref{sec_modules} differ from those of Grojnowski, the propositions and their proofs in Section~\ref{sec_harvest} follow \cite[Section 13]{Groj} extremely closely.

%
\subsection{Constructing the strict embedding $\Psi$} \label{sec_embedding}
%

Recall Proposition~\ref{prop_K10.1.3} that said $\eph(\tilde{f}_jM)=\eph(M)$ when
$i \neq j$ but when $i=j$ either $\eph(\ft{i}M)=\eph(M)$ or $\eph(M)+1$.

\begin{prop} \label{lem13.1}
Let $M$ be a simple $\Rn{\nu}$-module, and write $c=\eph(M)$.
\begin{enumerate}[i)]
  \item \label{13.1p1} Suppose $\eph(\ft{i} M) = \eph(M) +1$.  Then
  \begin{equation}
    \ets{i}\ft{i}M \cong M.
  \end{equation}
  \item \label{13.1p2} Suppose $\eph(\ft{j} M) = \eph(M)$ where $i$ and $j$ are not necessarily distinct.  Then
\end{enumerate}
\begin{equation}
  (\ets{i})^c (\ft{j}M) \cong \ft{j} (\ets{i}{}^cM).
\end{equation}
\end{prop}

\begin{proof}
For part \eqref{13.1p1}, the Jump Lemma~\ref{lem_jump_tfae} gives us $\ft{i} M
\iso \fts{i} M$. Therefore, $\ets{i}\ft{i}M \cong \ets{i}\fts{i}M
\cong M$.

For part \eqref{13.1p2} let $\overline{M}=(\ets{i})^cM$ so that $\eph(\barM)=0$ and we have a surjection $\Ind L(i^c)\boxtimes \barM \twoheadrightarrow M$ as well as
\begin{equation} \label{13.1.1}
  \Ind L(i^c)\boxtimes \barM \boxtimes L(j)\twoheadrightarrow \ft{j}M.
\end{equation}
Note that as $c=\eph(\ft{j}M)$, all composition factors of $(\es{i})^c\ft{j}M$ are
isomorphic to $(\ets{i})^c\ft{j}M$, so there exists a surjection $(\es{i})^c\ft{j}M\twoheadrightarrow (\ets{i})^c\ft{j}M$.  As $(\es{i})^c$ is exact, we may apply it to \eqref{13.1.1} and compose with the map above yielding
\begin{equation} \label{13.1.2}
  (\es{i})^c(\Ind L(i^c) \boxtimes \barM \boxtimes L(j)) \twoheadrightarrow (\ets{i})^c\ft{j}M.
\end{equation}
In the case $j\neq i$, by the Mackey Theorem~\cite[Proposition
2.8]{KL} $(\es{i})^c(\Ind L(i^c) \boxtimes \barM \boxtimes L(j))$
has a filtration whose subquotients are isomorphic
to $\Ind \barM \boxtimes L(j)$.  So \eqref{13.1.2} yields a
map
\begin{equation}
  \Ind \barM \boxtimes L(j)  \twoheadrightarrow (\ets{i})^c \ft{j}M,
\end{equation}
 which implies
\begin{equation}
  (\ets{i})^c \ft{j}M \cong \ft{j}\overline{M} \cong \ft{j}(\ets{i})^c M.
\end{equation}

In the case $j = i$, the subquotients are isomorphic to $\Ind \barM
\boxtimes L(i)$ or $\Ind L(i) \boxtimes \barM$.
But, by assumption
$\eph((\ets{i})^c\ft{i}M)=0$,
so by Frobenius reciprocity we cannot
have a nonzero map from $\Ind L(i) \boxtimes \barM$ to
$(\ets{i})^c\ft{i}M$.
As before, we must have
\begin{equation}
  \Ind \barM \boxtimes L(i) \twoheadrightarrow 
(\ets{i})^c \ft{i}M
\end{equation}
and so $(\ets{i})^c \ft{j}M = (\ets{i})^c \ft{i}M
 \cong
\ft{i}\barM =
\ft{i}(\ets{i})^cM=
\ft{j}(\ets{i})^cM$.
\end{proof}

\begin{prop} \label{lem13.2}
Let $M$ be an irreducible $\Rn{\nu}$-module, and write $c=\eph(M)$, $\barM = (\ets{i})^{c}(M)$.
\begin{enumerate}[i)]
  \item \label{13.2.1} $\epsilon_i(M) = \max\left\{ \epsilon_i(\barM), c -\wt_i(\barM) \right\}$.

  \item \label{13.2.2} Suppose $\epsilon_i(M)>0$.  Then
 \begin{equation}
   \eph(\et{i}M) =
 \left\{
 \begin{array}{lll}
    c & \quad & \text{if $\epsilon_i(\barM) \geq c -\wt_i(\barM)$,}\\
    c-1 & & \text{if $\epsilon_i(\barM) < c -\wt_i(\barM)$.}
 \end{array}
 \right.
 \end{equation}

 \item \label{13.2.3} Suppose $\epsilon_i(M)>0$. Then
 \begin{equation} (\ets{i})^{\eph(\et{i}M)}(\et{i}M) =
 \left\{
 \begin{array}{lll}
     \et{i}(\barM)& \quad & \text{if $\epsilon_i(\barM) \geq c -\wt_i(\barM)$,}\\
    \barM & & \text{if $\epsilon_i(\barM) < c -\wt_i(\barM)$.}
 \end{array}
 \right.
\end{equation}
\end{enumerate}
\end{prop}

\begin{proof}
Suppose $\epsilon_i(M)>\epsilon_i(\barM)$.  Then $\jump_i(M)=0$ and by Proposition~\ref{prop_jump_tfae2}~\eqref{jumpwt}
\begin{equation}
  0=\jump_i(M)=\epsilon_i(M)+\eph(M)+\wt_i(M)
  =\epsilon_i(M)+c+\wt_i(\barM)-2c
\end{equation}
so that $\epsilon_i(M)=c-\wt_i(\barM)$, and clearly
$\epsilon_i(M)=\max\left\{ \epsilon_i(\barM), c -\wt_i(\barM)
\right\}$. It is always the case that $\jump_i(M) \geq 0$. If
$\epsilon_i(M)=\ep{\barM}$, then as above
$\epsilon_i(M)=(c-\wt_i(\barM))+\jump_i(M) \geq c-\wt_i(\barM)$.  So
again $\epsilon_i(M)=\max\left\{ \epsilon_i(\barM), c -\wt_i(\barM)
\right\}$.

For part \eqref{13.2.2} consider two cases.

Case 1 ($\ep{\barM} < c -\wt_i(\barM)$):
Recall by Proposition~\ref{prop_jump_tfae2}~\eqref{jumpwt}, $\jump_i(\barM)=\eph(\barM)+\epsilon_i(\barM)+ \wt_i(\barM)=0 +\ep{\barM}+\wt_i(\barM)$ so $\jump_i{\barM}<c$ if and only if $\ep{\barM} < c -\wt_i(\barM)$.  Since $\jump_i{\barM}<c$ then $0=\jump_i((\fts{i})^{c-1}\barM)=\jump_i(\ets{i}M)$ by \eqref{jumprecursion}.
By the Jump Lemma~\ref{lem_jump_tfae}, $\ft{i}(\ets{i}M) \cong \fts{i}(\ets{i}M) \cong M$.  Hence $\ets{i}M=\et{i}M$ and so $\eph(\et{i}M)=\eph(\ets{i}M)=c-1$.

Case 2 ($\ep{\barM} \geq c -\wt_i(\barM)$):
As above this case is equivalent to
$\jump_i{\barM}\geq c$.
Note if $c=0$ then \eqref{13.2.2} obviously holds by Proposition~\ref{prop_K10.1.3}.
If $c>0$  by \eqref{jumprecursion}, we must have $0 < \jump_i((\fts{i})^{c-1}\barM)=\jump_i(\ets{i}M)$.  Suppose that $\jump_i(\et{i}M)=0$.  Then as above $\fts{i}\et{i}M \cong \ft{i}\et{i}M \cong M$ and so $\et{i}M\cong\ets{i}M$ yielding $\jump_i(\ets{i}M)=0$
which is a contradiction.   So we must have $\jump_i(\et{i}M)>0$.  Then by
the definition of $\jump_i$, $\eph(\et{i}M)=\eph(\ft{i}\et{i}M)=\eph(M)=c$.

For part \eqref{13.2.3}, first suppose $\ep{\barM} \geq c-\wt_i(\barM)$.  Then by part~\eqref{13.2.2} $\eph(\et{i}M)=c=\ep{M}$.  In other words $\eph(\et{i}M)=\eph(\ft{i}\et{i}M)$ so by Proposition~\ref{lem13.1} applied to $\et{i}M$,
\begin{equation}
  \ft{i}(\ets{i})^c\et{i}M \cong (\ets{i})^c \ft{i}\et{i}M \cong (\ets{i})^cM =\barM.
\end{equation}
Hence $(\ets{i})^c\et{i}M \cong \et{i}\barM$.

Next suppose $\ep{\barM}< c-\wt_i(\bar{M})$.  Then by part~\eqref{13.2.2}
\begin{equation}
  \eph(\et{i}M)=c-1=\eph(M)-1.
\end{equation}
In other words $\eph(\ft{i}\et{i}M) = \eph(\et{i}M)+1$, so by Proposition~\ref{lem13.1} applied to $\et{i}M$,
\begin{equation}
 \ets{i}M \cong \ets{i}\ft{i}\et{i}M \cong \et{i}M,
\end{equation}
hence $(\ets{i})^{c-1}\et{i}M \cong (\ets{i})^{c-1}\ets{i}M \cong (\ets{i})^cM \cong \barM$.
\end{proof}

\begin{prop}
For each $i\in I$ define a map
\begin{eqnarray}
  \Psi_i \maps \cal{B} &\to& \cal{B} \otimes B_{i} \nn\\
    M         &\mapsto& (\ets{i})^c(M) \otimes b_i(-c), \nn
\end{eqnarray}
where $c =\eph(M)$.  Then $\Psi_i$ is a strict embedding of crystals.
\end{prop}

\begin{proof}
First we show that $\Psi_i$ is a morphism of crystals.  (M1) is obvious. For (M2) let $\barM=(\ets{i})^cM$.  We compute
\begin{equation}
  \wt(\psi_i(M)) = \wt(\barM \otimes b_i(-c)) = \wt(\barM)+\wt(b_i(-c)) = \wt(M)+c\alpha_i - c \alpha_i = \wt(M).
\end{equation}
Consider first the case $j \neq i$.  By Proposition~\ref{prop_K10.1.3}
\begin{eqnarray}
  \epsilon_j(\Psi_i(M)) &=& \epsilon_j(\barM\otimes b_i(-c)) \nn \\
  &=&  \max\{\epsilon_j(\barM),\epsilon_j(b_i(-c))-
  \langle h_j,\wt(\barM)\rangle \} \nn \\
  &=& \max\{\epsilon_j(\barM),-\infty\} =  \epsilon_j(\barM) \nn \\
  &=& \epsilon_j(M). \nn
\end{eqnarray}
In the case $j=i$,  Proposition~\ref{lem13.2}~\eqref{13.2.1} implies
\begin{eqnarray}
  \epsilon_i(\Psi_i(M)) &=& \ep{\barM\otimes b_i(-c)} \nn \\
  &=& \max\{\epsilon_i(\barM),\ep{b_i(-c)}-\langle h_i,\wt(\barM)\rangle \}
  = \max\{\epsilon_i(\barM),c-\wt_i(\barM)\} \nn \\
  &=& \epsilon_i(M).
\end{eqnarray}
Since for both crystals, $\phi_j(b)=\epsilon_j(b)+\langle h_j, \wt(b) \rangle$ it follows $\phi_j(M)=\phi_j(\Psi_i(M))$ for all $j \in I$.

It is clear that $\Psi_i$ is injective. We will prove a stronger statement than (M3) and (M4), namely $\Psi_i(\et{j}M)=\et{j}(\Psi_i(M))$ and $\Psi_i(\ft{j}M)=\ft{j}(\Psi_i(M))$ which will show $\Psi_i$ is not just a morphism of crystals, but since
it
is injective,
$\Psi_i$
is a strict embedding of crystals.

Observe
\begin{equation} \label{epsi}
 \et{j}(\Psi_i(M)) = \et{j}\left(\barM \otimes b_i(-c)\right) =
\left\{
 \begin{array}{lll}
   \et{j}\barM\otimes b_i(-c) & \quad & \text{if $\phi_j(\barM)\geq \ep{b_i(-c)}=c$} \\
  \barM \otimes b_i(-c+1) &  & \text{if $\phi_j(\barM)< c$.}
 \end{array}
\right.
\end{equation}
We first consider the case when $j=i$.
If $\ep{M}=0$, then clearly $\ep{\barM}=0$ and further $\et{i}M=\et{i}\barM=\0$.  By Proposition~\ref{lem13.2}~\eqref{13.2.1}
\begin{equation}
  \ep{\barM} = 0 = \ep{M} = \max \{ \ep{\barM}, c -\wt_i(\barM) \} \geq c-\wt_i(\barM),
\end{equation}
yielding
$\phi_i(\barM) = \ep{\barM} + \wt_i(\barM) \ge (c  - \wt_i(\barM)) + \wt_i(\barM) = c$,
so by \eqref{eq_ei_tensor}, \eqref{eq_Bi_eps} we get
\begin{equation}
  \et{i}\Psi_i(M) = \et{i} \barM \otimes b_i(-c) = 0 = \Psi_i(0)=\Psi_i(\et{i}M).
\end{equation}
Now suppose $\ep{M} > 0$.
Using that $\phi_i(\barM):= \epsilon_i(\barM) + \wt_i(\barM)$, \eqref{eq_ei_tensor}, and \eqref{eq_Bi_eps}, Proposition~\ref{lem13.2} implies we can rewrite
\begin{eqnarray}
\et{i}\Psi_i(M) &=& \left\{
 \begin{array}{lll}
    (\ets{i})^{c}\et{i}M \otimes b_i(-c) & \quad & \text{if $\epsilon_i(\barM) \geq c-\wt_i(\barM)$} \\
   (\ets{i})^{c-1}\et{i}M \otimes b_i(-c+1) &  & \text{if $\epsilon_i(\barM) < c-\wt_i(\barM)$}
 \end{array}
\right. \\
& =& (\ets{i})^{\eph(\et{i}M)}\et{i}M \otimes b_i(\eph(\et{i}M))\\ &=& \Psi_i(\et{i}M).
\end{eqnarray}

When $j\neq i$ note that $\eph(\et{j}M)=\eph(M)=c$ so long as
$\et{j}M \neq \0$, by Proposition~\ref{prop_K10.1.3} applied to
$\et{j}M$. Equation \eqref{13.1p2} of Proposition~\ref{lem13.1}
implies $\barM = (\ets{i})^cM = \ft{j}(\ets{i})^c \et{j}M$, so
$\et{j}\barM =(\ets{i})^{c}\et{j}M$. Therefore,  by \eqref{epsi} as
$ \epsilon_j(b_i(-c))=-\infty$,
\begin{equation}
\et{j}(\Psi_i(M))= \et{j}\barM\otimes b_i(-c) = (\ets{i})^{c}\et{j}M\otimes b_i(-c)=\Psi_i(\et{j}M).
\end{equation}
In the case $\et{j}M = \0$, Proposition~\ref{prop_K10.1.3} implies $\et{j}\barM  = \0$
as well, so we compute
$$
\et{j}(\Psi_i(M))= \et{j}\barM\otimes b_i(-c) =
0 = \Psi_i(0)
=\Psi_i(\et{j}M).
$$
The proof that $\Psi_i(\ft{j}M)=\ft{j}(\Psi_i(M))$ is similar.
\end{proof}

%
\subsection{Main Theorems}
%

In the following we use the characterization of $B(\infty)$ from Section~\ref{sec_Binf} to implicitly prove $\cal{B}$ is isomorphic to $B(\infty)$.

\begin{thm} \label{thm_Binf}
The crystal $\cal{B}$ is isomorphic to $B(\infty)$.
\end{thm}

\begin{proof}
Recall that by abuse of notation, for irreducible modules $M$, we
write $M \in \cal{B}$ as shorthand for $[M] \in \cal{B}$. We show
that the crystal $\cal{B}$ satisfies the characterizing properties
of $B(\infty)$ given in Proposition~\ref{prop_descB}.  Properties
(B1)-(B4) are  obvious with $\1$ the unique node with weight zero.
The embedding $\Psi_i \maps \cal{B} \to \cal{B} \otimes B_i$ for
(B5) was constructed in the previous section.  (B6) follows from the
definition of $\Psi_i$ as $\ephj(M)\geq 0$ for all $M \in \cal{B}$,
$j \in I$.  For (B7) we must show that for $M \in \cal{B}$ other
than $\1$, then there exists $i \in I$ such that $\Psi_i(M)=N
\otimes \ft{i}^nb_i$ for some $N \in \cal{B}$ and $n>0$.  But every
such $M$ has $\eph(M) >0$ for at least one $i \in I$, so that $N$
can be taken to be $\ets{i}{}^n(M)$ for
 $n = \eph(M) > 0$.
\end{proof}

Now we will show the data $(\cal{B}^{\Lambda},\epsilon_i^{\Lambda},\phiL,\et{i}^{\Lambda},\et{i}^{\Lambda},\wt^{\Lambda})$
of Section~\ref{sec-BLambda} defines a crystal graph and identify it as the highest weight crystal $B(\Lambda)$.

\begin{thm}
$\cal{B}^\Lambda$ is a crystal; furthermore the crystal $\cal{B}^\Lambda$ is isomorphic to $B(\Lambda)$.
\end{thm}

\begin{proof}
Proposition 8.2 of Kashiwara~\cite{Kas4} gives us an embedding
\begin{equation} \label{upsilon}
  \Upsilon^{\infty} \maps B(\Lambda) \to B(\infty) \otimes T_{\Lambda}
\end{equation}
which identifies $B(\Lambda)$ as a subcrystal of $B(\infty) \otimes T_{\Lambda}$.
The nodes of $B(\Lambda)$ are associated with the nodes of the image
\begin{equation}
  {\rm Im} \Upsilon^{\infty} = \{ b \otimes t_{\Lambda}\mid \epsilon_i^*(b) \leq \langle h_i,\Lambda\rangle,\quad \text{for all $i \in I$} \}
\end{equation}
where $c=\epsilon_i^*(b)$ is defined via $\Psi_i b = b' \otimes b_i(-c)$ for the strict embedding $\Psi_i \maps B(\infty) \to B(\infty) \otimes B_i$.
The crystal data for $B(\Lambda)$ is thus inherited from that of $B(\infty) \otimes T_{\Lambda}$.
Via our isomorphism $B(\infty) \otimes T_{\Lambda} \cong \cal{B} \otimes T_{\Lambda}$ of Theorem~\ref{thm_Binf} and the description of
\begin{eqnarray}
  \Psi_i\maps \cal{B} &\to& \cal{B} \otimes B_i \nn \\
 M & \mapsto& (\ets{i})^{\eph(M)}M \otimes b_i(-\eph(M))
\end{eqnarray}
the set
\begin{equation}
  \{ M \otimes t_{\Lambda} \in \cal{B} \otimes t_{\Lambda} \mid \eph(M) \leq \lambda_i, \quad \text{for all $i \in I$} \}
\end{equation}
endowed with the crystal data of $\cal{B}\otimes T_{\Lambda}$ is thus isomorphic to $B(\Lambda)$.

Recall from Section~\ref{sec-BLambda} this is precisely ${\rm Im} \Upsilon$, as $\eph(M)\leq \lambda_i$ for all $i \in I$ if and only if $\prL M \neq \0$ which happens if and only if $M =\infL \cal{M}$
for some $\cal{M}\in \cal{B}^{\Lambda}$. By Kashiwara's Proposition, we know ${\rm Im} \Upsilon \cong B(\Lambda)$ as crystals.

What remains is to check that the crystal data ${\rm Im} \Upsilon$ inherits from $\cal{B}\otimes T_{\Lambda}$ agrees with the data defined in Section~\ref{sec-BLambda} for $\cal{B}^{\Lambda}$.
Once we verify this, we will have shown $\cal{B}^{\Lambda}$ is a crystal, $\cal{B}^{\Lambda}\cong B(\Lambda)$, and $\Upsilon$ is an embedding of crystals.

Let $\cal{M} \in \cal{B}^{\Lambda}$. Recall,
since $\prL\infL \cal{M} \neq \0$,
then
 $0 \leq \phiL(\infL \cal{M})= \phiL(\cal{M})$ which was defined as
$\max\{ k \mid \prL\ft{i}^k(\infL\cal{M})\neq \0\}$.
We verify
\begin{eqnarray}
  \phi_i(\Upsilon \cal{M}) &=& \phi_i(\infL\cal{M} \otimes t_{\Lambda}) \nn \\
&=& \phi_i(\infL\cal{M})+\lambda_i \nn \\
&=& \ep{\infL{\cal{M}}}+\wt_i(\infL\cal{M})+\lambda_i \nn \\
&\refequal{{\rm Cor}~\ref{cor_jump}}&
 \phiL(\infL\cal{M}) = \phiL(\cal{M}).
\end{eqnarray}

This computation, along with \eqref{ups1}--\eqref{ups4} completes the check that $(\cal{B}^{\Lambda},\epsilon_i^{\Lambda},\phiL,\et{i}^{\Lambda},\et{i}^{\Lambda},\wt^{\Lambda})$
is a crystal and isomorphic to $B(\Lambda)$.
\end{proof}

%
\subsection[$\mathbf{U}^+$-module structure]{$\Up$-module structures}
%

Set
$$\Gdual  = \bigoplus_\nu \Gnu \qquad \GL = \bigoplus_\nu \GnuL$$
where, by $V^\ast$ we mean the restricted linear dual $\Hom_{\mc{A}}(V, \mc{A})$.
Because $G_0(R)$ and $G_0(R^\Lambda)$ are $\UpA$-modules, we can endow
$\Gdual$, $\GL$ with a left $\UpA$-module structure in several ways,
via a choice of anti-automorphism.
Here we denote by $\ast$ the $\mc{A}$-linear anti-automorphism
defined by
$$e_i^\ast = e_i \, \text{ for all } i \in I.$$
Specifically, for $y \in \UpA$, $\gamma \in \Gdual$ or $\GL$, and $N$ simple,
set
$$(y \mydot \gamma)\left( [N] \right)  = \gamma \left( y^\ast [N] \right)$$
where we will identify $\eL{i}$ with $e_i$.

$\Gnu$
has basis given by $\{ \delM{M} \mid M \in \cal{B}, \wt(M) = -\nu\}$
defined by
$$\delM{M} ([N])  = \begin{cases} q^{-r} & M \iso N\{r\} \\
                        0 & \text{otherwise,}
\end{cases}
$$
where $N$ ranges over simple $\Rnu$-modules.  We set $\wt(\delta_M) = - \wt(M)$.
Likewise $\GnuL$ has basis
$\{ \deLM{\cal{M}} \mid \cal{M} \in \cal{B}^\Lambda, \wt(\cal{M}) = -\nu+\Lambda\}$
defined similarly.
Note that if $\delM{M}$ has degree $d$ then $\delM{{M\{1\}}} = q^{-1} \delM{M}$
has degree $d-1$.
Recall $\1 \in \cal{B}$ denotes the trivial $\Rn{0}$-module
and we will also write $\1 \in \cal{B}^\Lambda$  for the trivial $\RnL{0}$-module.

\begin{lem} \label{lem_eidelta} \hfill
\begin{enumerate}[i)]
\item \label{eidelta}
$e_i^{(m)} \mydot \deltriv = \delM{L(i^m)} \in G_0(\Rn{mi})^\ast$; 
$e_i^{(m)} \mydot \deLtriv =
0 \in G_0(R^\Lambda(mi))^\ast \subseteq \GL$ if $m \ge \lambda_i +1$.
\item \label{generates}
$\Gdual$ is generated by $\deltriv$ as a $\UpA$-module; 
$\GL$ is generated by $\deLtriv$ as a $\UpA$-module.
\end{enumerate}
\end{lem}
\begin{proof}
The first part follows since $e_i^{(m)} L(i^m) \iso \1$ and the only
irreducible module $N$ for which $e_i^{(m)} N$ is  a nonzero $\Rn{0}$-module
is $N \iso L(i^m) \{ r\}$ for some $r \in \Z$.
Recall $\prL L(i^m) = \0$ if and only if $m\ge \lambda_i +1$.

For the second part, recall $\1$ co-generates $G_0(R)$
(resp.
$G_0(R^\Lambda)$)
in the sense that for any irreducible $M$, there exist $i_t \in I$ such that
$$e_{i_k}^{(m_k)} \cdots  e_{i_2}^{(m_2)} e_{i_1}^{(m_1)} M \iso \mathfrak{a} \1,$$
where $m_t = \epany{i_t}{ \et{i_{t-1}}^{m_{t-1}}\cdots  \et{i_{1}}^{m_{1}} M }$
and $\mathfrak{a} \in \mc{A}$ (in fact $\mathfrak{a} = q^r$ for some $r \in \Z$).
So certainly $\deltriv$ generates $\Gdual$
(resp. $\deLtriv$  generates $G_0(R^\Lambda)$).

More specifically, an inductive argument relying on ``triangularity" with respect
to $\epsilon_i$ gives $\delM{M} \in \UpA \mydot \deltriv$ and
$\deLM{\cal{M}} \in \UpA \mydot \deLtriv$.
\end{proof}

\begin{lem} \label{lem_Umaps} \hfill
\begin{enumerate}[i)]
\item \label{Umaps}
The maps
\begin{align}
\UpA &\xrightarrow{F} \Gdual & \UpA &\xrightarrow{\cal{F}} \GL \\
y &\mapsto y \mydot \deltriv   &        y &\mapsto y \mydot \deLtriv
\end{align}
are $\UpA$-module homomorphisms.
\item \label{Usurj}
$F$ and $\cal{F}$ are surjective.
\item \label{Uker}
$\ker \cal{F} \ni e_i^{(\lambda_i +1)}$ for all $i \in I$.
\end{enumerate}
\end{lem}
\begin{proof}
To show $F$, $\cal{F}$ are $\UpA$-maps, we need only check the Serre relations
\eqref{Serrea} vanish on $\Gdual$, $\GL$. But as the corresponding operators
are invariant under $\ast$ and vanish on any $[N]$, they certainly kill any
$\delM{M}$, $\,\deLM{\cal{M}}$.

Now $F$ (resp. $\cal{F}$) is clearly surjective as it contains
the generator $\deltriv$  (resp. $\deLtriv$) in its image.

The third statement follows from part~\eqref{eidelta} of Lemma~\ref{lem_eidelta}.
\end{proof}

  If $\VL$ is the irreducible highest weight $\U$-module with highest weight
$\Lambda$ and highest weight vector $v_{\Lambda}$ then its $\mc{A}$-form,
or integral form, $\AV$ is the $U_{\mc{A}}$-submodule of $\VL$
generated by $v_{\Lambda}$.
In particular, $\AV= \UmA \mydot v_{\Lambda}$.
We let $\Vdual$ denote the graded dual of $\VL$, whose elements are
sums of $\delta_v, v \in \VL$.
If $v \in V(\Lambda) $ has weight $\mu$ then
$\delta_v \in V(\Lambda)^\ast$ has weight $-\mu$
and
$e_i v$, if nonzero, has weight $\mu + i$ in the notation of this paper.
We set
$$\AVdual = \UpA \mydot \delta_{v_{\Lambda}}$$
  endowed  with the left $\UpA$-module structure
$$y \mydot \delta_v(w) = \delta_v( y^\ast w).$$
Note that the $-\mu$ weight space of the dual is the dual of the $\mu$
weight space, and that both weight spaces are free $\mc{A}$-modules
of finite rank.

As a left $\UpA$-module
\begin{gather}\label{UV}
\AVdual \iso \UpA / \sum_{i \in I} \UpA \mydot e_i^{(\lambda_i +1)}.
\end{gather}

We emphasize that parts \eqref{VGdual} and \eqref{VG} of the theorem below are new and settle part of the cyclotomic quotient conjecture in arbitrary type.  While part \eqref{UG} follows from \cite[Theorem 8]{KL2}, here we have given a new proof of it modeled after the work of Grojnowski~\cite{Groj} using crystals to verify the rank of $G_0(R(\nu))$.

\begin{thm} \label{thm_iso}
As $\UpA$ modules
\begin{enumerate}
\item \label{UG}
$ \UpA \iso   \Gdual$,
\item \label{VGdual}
$\AVdual \iso \GL$,
\item \label{VG}
$\AV \iso G_0(R^\Lambda)$.
\end{enumerate}
\end{thm}

\begin{proof}
Note that both $F$ and $\cal{F}$
 are surjective and preserve weight in the
sense that $\wt(e_i) = i$ in the notation of this paper.
We know the dimension
 of the $\nu$-weight space of $\Up$ is
$$|\{ b \in B(\infty) \mid \wt(b) = -\nu\}|
=|\{ M \in \mc{B} \mid \wt(M) = -\nu\}|
= \rank_{\mc{A}} G_0(\Rnu) = \rank_{\mc{A}} \Gnu.$$
Because $\mc{A}$ is an integral domain, a surjection between two
free $\mc{A}$-modules of the same (finite) rank must be an injection.
Hence
$ F$ must also be injective and
hence an isomorphism.

Since the left ideal $\sum_{i \in I} \UpA \mydot e_i^{(\lambda_i +1)}$
is contained in the kernel of $\cal{F}$ by part~\eqref{Uker} of Lemma~\ref{lem_Umaps},
$\cal{F}$ induces a surjection
$${}_{\mc{A}} V^\ast(\Lambda) \twoheadrightarrow \GL.$$
The dimension  of the $-\Lambda + \nu$ weight space of
$ V(\Lambda)^\ast$
is the same as
\begin{gather}
\dim V(\Lambda)_{\Lambda - \nu} =
|\{ \mathfrak{b} \in B(\Lambda) \mid \wt(\mathfrak{b}) = \Lambda -\nu\}|
=|\{ \mathcal{M} \in \cal{B}^\Lambda \mid \wt(\mc{M}) = \Lambda -\nu\}|
\\
= \rank_{\mc{A}} G_0(\RnuL) = \rank_{\mc{A}} \GnuL,
\end{gather}
so 
as above, $\cal{F}$  must in fact be an isomorphism.

The third statement follows from dualizing with respect to the antiautomorphism $\ast$.
\end{proof}

We note that \cite{KL} proves a stronger statement than part~\eqref{UG} of
Theorem~\ref{thm_iso}, namely that $ \Af  \iso K_0(R)$ as $\mc{A}$-bialgebras. So in
particular,  as $\UpA$-modules, $\UpA \iso K_0(R)$.  Using their result yields
another proof that
$\UpA \iso G_0(R)$ as $\UpA$-modules.



\begin{thebibliography}{BKW09}

\bibitem[AK94]{AK}
S.~Ariki and K.~Koike.
\newblock A {H}ecke algebra of {$({\bf Z}/r{\bf Z})\wr{\cal{S}}\sb n$} and
  construction of its irreducible representations.
\newblock {\em Adv. Math.}, 106(2):216--243, 1994.

\bibitem[AM00]{AM}
S.~Ariki and A.~Mathas.
\newblock The number of simple modules of the {H}ecke algebras of type
  {$G(r,1,n)$}.
\newblock {\em Math. Z.}, 233(3):601--623, 2000.

\bibitem[Ari96]{Ari1}
S.~Ariki.
\newblock On the decomposition numbers of the {H}ecke algebra of {$G(m,1,n)$}.
\newblock {\em J. Math. Kyoto Univ.}, 36(4):789--808, 1996.

\bibitem[Ari99]{Ari2}
S.~Ariki.
\newblock Lectures on cyclotomic {H}ecke algebras, 1999.
\newblock math.QA/9908005.

\bibitem[Ari02]{Ari3}
S.~Ariki.
\newblock {\em Representations of quantum algebras and combinatorics of {Y}oung
  tableaux}, volume~26 of {\em University Lecture Series}.
\newblock AMS, Providence, RI, 2002.

\bibitem[Ari09]{Ari4}
S.~Ariki.
\newblock Graded $q$-{S}chur algebras, 2009.
\newblock math.QA/0903.3453.

\bibitem[BK09a]{BK1}
J.~Brundan and A.~Kleshchev.
\newblock Blocks of cyclotomic {H}ecke algebras and {K}hovanov-{L}auda
  algebras.
\newblock {\em Invent. Math.}, 178(3):451--484, 2009.
\newblock arXiv:0808.2032.

\bibitem[BK09b]{BK2}
J.~Brundan and A.~Kleshchev.
\newblock Graded decomposition numbers for cyclotomic {H}ecke algebras, 2009.
\newblock arXiv:0901.4450.

\bibitem[BKW09]{BKW}
J.~Brundan, A.~Kleshchev, and W.~Wang.
\newblock Graded {S}pecht modules, 2009.
\newblock arXiv:0901.0218.

\bibitem[BM93]{BM}
M.~Brou{\'e} and G.~Malle.
\newblock Zyklotomische {H}eckealgebren.
\newblock {\em Ast\'erisque}, (212):119--189, 1993.
\newblock Repr{\'e}sentations unipotentes g{\'e}n{\'e}riques et blocs des
  groupes r{\'e}ductifs finis.

\bibitem[BS08]{BS}
J.~Brundan and C.~Stroppel.
\newblock Highest weight categories arising from {K}hovanov's diagram algebra
  {III}: category {$\cal{O}$}, 2008.
\newblock arXiv:0812.1090.

\bibitem[BZ77]{BZ}
I.~N. Bernstein and A.~V. Zelevinsky.
\newblock Induced representations of reductive {${p}$}-adic groups. {I}.
\newblock {\em Ann. Sci. \'Ecole Norm. Sup. (4)}, 10(4):441--472, 1977.

\bibitem[Che87]{Ch}
I.~Cherednik.
\newblock A new interpretation of {G}el\cprime fand-{T}zetlin bases.
\newblock {\em Duke Math. J.}, 54(2):563--577, 1987.

\bibitem[CR08]{CR}
J.~Chuang and R.~Rouquier.
\newblock Derived equivalences for symmetric groups and sl\_2-categorification.
\newblock {\em Ann. of Math.}, 167:245--298, 2008.
\newblock math.RT/0407205.

\bibitem[EK06]{EK}
N.~Enomoto and M.~Kashiwara.
\newblock Symmetric crystals and affine {H}ecke algebras of type {B}.
\newblock {\em Proc. Japan Acad. Ser. A Math. Sci.}, 82(8):131--136, 2006.
\newblock arXiv:math/0608079.

\bibitem[GL93]{GL2}
I.~Grojnowski and G.~Lusztig.
\newblock A comparison of bases of quantized enveloping algebras.
\newblock In {\em Linear algebraic groups and their representations (Los
  Angeles, CA, 1992)}, volume 153 of {\em Contemp. Math.}, pages 11--19. 1993.

\bibitem[Gro94]{Groj2}
I.~Grojnowski.
\newblock Representations of affine {H}ecke algebras (and affine quantum {${\rm
  GL}\sb n$}) at roots of unity.
\newblock {\em Internat. Math. Res. Notices}, (5):215 ff., approx.\ 3 pp.\
  (electronic), 1994.

\bibitem[Gro99]{Groj}
I.~Grojnowski.
\newblock Affine $sl_p$ controls the representation theory of the symmetric
  group and related {H}ecke algebras, 1999.
\newblock math.RT/9907129.

\bibitem[GV01]{GV}
I.~Grojnowski and M.~Vazirani.
\newblock Strong multiplicity one theorems for affine {H}ecke algebras of type
  {A}.
\newblock {\em Transform. Groups}, 6(2):143--155, 2001.

\bibitem[HK02]{HK}
J.~Hong and S.J. Kang.
\newblock {\em Introduction to quantum groups and crystal bases}, volume~42 of
  {\em Graduate Studies in Mathematics}.
\newblock American Mathematical Society, Providence, RI, 2002.

\bibitem[HL10]{Lau3}
A.~Hoffnung and A.~D. Lauda.
\newblock Nilpotency in type {A} cyclotomic quotients.
\newblock {\em Journal of Algebraic Combinatorics}, 32:533--555, 2010.
\newblock math.RT/0903.2992.

\bibitem[HM09]{HM}
J.~Hu and A.~Mathas.
\newblock Graded cellular bases for the cyclotomic
  {K}hovanov-{L}auda-{R}ouquier algebras of type {A}, 2009.
\newblock arXiv:0907.2985.

\bibitem[Kas90a]{Kas5}
M.~Kashiwara.
\newblock Bases cristallines.
\newblock {\em C. R. Acad. Sci. Paris S\'er. I Math.}, 311(6):277--280, 1990.

\bibitem[Kas90b]{Kas1}
M.~Kashiwara.
\newblock Crystalizing the {$q$}-analogue of universal enveloping algebras.
\newblock {\em Comm. Math. Phys.}, 133(2):249--260, 1990.

\bibitem[Kas91]{Kas2}
M.~Kashiwara.
\newblock On crystal bases of the {$Q$}-analogue of universal enveloping
  algebras.
\newblock {\em Duke Math. J.}, 63(2):465--516, 1991.

\bibitem[Kas93]{Kas3}
M.~Kashiwara.
\newblock Global crystal bases of quantum groups.
\newblock {\em Duke Math. J.}, 69(2):455--485, 1993.

\bibitem[Kas95]{Kas4}
M.~Kashiwara.
\newblock On crystal bases.
\newblock In {\em Representations of groups ({B}anff, {AB}, 1994)}, volume~16
  of {\em CMS Conf. Proc.}, pages 155--197. Amer. Math. Soc., Providence, RI,
  1995.

\bibitem[KK11]{KK}
S.J. Kang and M.~Kashiwara.
\newblock Categorification of highest weight modules via
  {K}hovanov-{L}auda-{R}ouquier algebras, 2011.
\newblock arXiv:1102.4677.

\bibitem[KL09]{KL}
M.~Khovanov and A.~Lauda.
\newblock A diagrammatic approach to categorification of quantum groups {I}.
\newblock {\em Represent. Theory}, 13:309--347, 2009.
\newblock math.QA/0803.4121.

\bibitem[KL10]{KL3}
M.~Khovanov and A.~Lauda.
\newblock A diagrammatic approach to categorification of quantum groups {III}.
\newblock {\em Quantum Topology}, 1:1--92, 2010.
\newblock math.QA/0807.3250.

\bibitem[KL11]{KL2}
M.~Khovanov and A.~Lauda.
\newblock A diagrammatic approach to categorification of quantum groups {II}.
\newblock {\em Trans. Amer. Math. Soc.}, 363:2685--2700, 2011.
\newblock math.QA/0804.2080.

\bibitem[Kle95]{Kle1}
A.~Kleshchev.
\newblock Branching rules for modular representations of symmetric groups.
  {II}.
\newblock {\em J. Reine Angew. Math.}, 459:163--212, 1995.

\bibitem[Kle96]{Kle2}
A.~Kleshchev.
\newblock Branching rules for modular representations of symmetric groups.
  {III}. {S}ome corollaries and a problem of {M}ullineux.
\newblock {\em J. London Math. Soc. (2)}, 54(1):25--38, 1996.

\bibitem[Kle97]{Kle3}
A.~Kleshchev.
\newblock On decomposition numbers and branching coefficients for symmetric and
  special linear groups.
\newblock {\em Proc. London Math. Soc. (3)}, 75(3):497--558, 1997.

\bibitem[Kle05]{KleBook}
A.~Kleshchev.
\newblock {\em Linear and projective representations of symmetric groups},
  volume 163 of {\em Cambridge Tracts in Mathematics}.
\newblock Cambridge U. Press, 2005.

\bibitem[KM07]{KM}
M.~Kashiwara and V.~Miemietz.
\newblock Crystals and affine {H}ecke algebras of type {$D$}.
\newblock {\em Proc. Japan Acad. Ser. A Math. Sci.}, 83(7):135--139, 2007.
\newblock arXiv:math/0703281.

\bibitem[KR09]{KR}
A.~Kleshchev and A.~Ram.
\newblock Representations of {K}hovanov-{L}auda-{R}ouquier algebras and
  combinatorics of {L}yndon words, 2009.
\newblock arXiv:0909.1984.

\bibitem[KR10]{KR2}
A.~Kleshchev and A.~Ram.
\newblock Homogeneous representations of {K}hovanov-{L}auda algebras.
\newblock {\em J. Eur. Math. Soc. (JEMS)}, 12(5):1293--1306, 2010.
\newblock arXiv:0809.0557.

\bibitem[KS97]{KS}
M.~Kashiwara and Y.~Saito.
\newblock Geometric construction of crystal bases.
\newblock {\em Duke Math. J.}, 89(1):9--36, 1997.

\bibitem[Lau08]{Lau1}
A.~D. Lauda.
\newblock A categorification of quantum sl(2).
\newblock {\em Adv. Math.}, 225:3327--3424, 2008.
\newblock math.QA/0803.3652.

\bibitem[Lec04]{Lec}
Bernard Leclerc.
\newblock Dual canonical bases, quantum shuffles and {$q$}-characters.
\newblock {\em Math. Z.}, 246(4):691--732, 2004.

\bibitem[LLT96]{LLT}
A.~Lascoux, B.~Leclerc, and J.-Y. Thibon.
\newblock Hecke algebras at roots of unity and crystal bases of quantum affine
  algebras.
\newblock {\em Comm. Math. Phys.}, 181(1):205--263, 1996.

\bibitem[Lus90a]{Lus1}
G.~Lusztig.
\newblock Canonical bases arising from quantized enveloping algebras.
\newblock {\em J. Amer. Math. Soc.}, 3(2):447--498, 1990.

\bibitem[Lus90b]{Lus2}
G.~Lusztig.
\newblock Canonical bases arising from quantized enveloping algebras. {II}.
\newblock {\em Progr. Theoret. Phys. Suppl.}, 102:175--201 (1991), 1990.
\newblock Common trends in mathematics and quantum field theories (Kyoto,
  1990).

\bibitem[Lus91]{Lus3}
G.~Lusztig.
\newblock Quivers, perverse sheaves, and quantized enveloping algebras.
\newblock {\em J. Amer. Math. Soc.}, 4(2):365--421, 1991.

\bibitem[Lus93]{Lus4}
G.~Lusztig.
\newblock {\em Introduction to quantum groups}, volume 110 of {\em Progress in
  Mathematics}.
\newblock Birkh\"auser Boston Inc., Boston, MA, 1993.

\bibitem[Lus98]{Lus6}
G.~Lusztig.
\newblock Canonical bases and {H}all algebras.
\newblock In {\em Representation theories and algebraic geometry (Montreal, PQ,
  1997)}, volume 514 of {\em NATO Adv. Sci. Inst. Ser. C Math. Phys. Sci.},
  pages 365--399. Kluwer Acad. Publ., Dordrecht, 1998.

\bibitem[Mat99]{Mathas}
A.~Mathas.
\newblock {\em Iwahori-{H}ecke algebras and {S}chur algebras of the symmetric
  group}, volume~15 of {\em University Lecture Series}.
\newblock AMS, Providence, RI, 1999.

\bibitem[MM90]{MT}
K.~Misra and T.~Miwa.
\newblock Crystal base for the basic representation of {$U\sb q(\mf{sl}(n))$}.
\newblock {\em Comm. Math. Phys.}, 134(1):79--88, 1990.

\bibitem[Nak94]{Na}
H.~Nakajima.
\newblock Instantons on {ALE} spaces, quiver varieties, and {K}ac-{M}oody
  algebras.
\newblock {\em Duke Math. J.}, 76(2):365--416, 1994.

\bibitem[NO04]{NV}
C.~N{\u{a}}st{\u{a}}sescu and F.~Van Oystaeyen.
\newblock {\em Methods of graded rings}, volume 1836 of {\em Lecture Notes in
  Mathematics}.
\newblock Springer-Verlag, Berlin, 2004.

\bibitem[Rin90]{Ringel}
C.~Ringel.
\newblock Hall algebras and quantum groups.
\newblock {\em Invent. Math.}, 101(3):583--591, 1990.

\bibitem[Rou08]{Ro}
R.~Rouquier.
\newblock 2-{K}ac-{M}oody algebras, 2008.
\newblock arXiv:0812.5023.

\bibitem[RTG]{RTG}
Columbia {U}niversity {RTG} {U}ndergraduate {R}esearch.
\newblock G. Kim, A. Kontaxis and D. Xia (2008). C. Bregman, R.~Legg,
  J.~McIvor, and D. Wang (2009).

\bibitem[SVV09]{SVV}
P.~Shan, M.~Varagnolo, and E.~Vasserot.
\newblock Canonical bases and affine {H}ecke algebras of type {D}, 2009.
\newblock arXiv:0912.4245.

\bibitem[Vaz99]{Vaz1}
M.~Vazirani.
\newblock Irreducible modules over the affine {H}ecke algebra: a strong
  multiplicity one result, 1999.
\newblock Ph.D. thesis, UC Berkeley.

\bibitem[Vaz02]{Vaz2}
M.~Vazirani.
\newblock Parameterizing {H}ecke algebra modules: {B}ernstein-{Z}elevinsky
  multisegments, {K}leshchev multipartitions, and crystal graphs, 2002.
\newblock math.RT/0107052.

\bibitem[VV09a]{VV2}
M.~Varagnolo and E.~Vasserot.
\newblock Canonical bases and affine {H}ecke algebras of type {B}, 2009.
\newblock arXiv:0911.5209.

\bibitem[VV09b]{VV}
M.~Varagnolo and E.~Vasserot.
\newblock Canonical bases and {K}hovanov-{L}auda algebras, 2009.
\newblock arXiv:0901.3992.

\bibitem[Web10]{Web}
B.~Webster.
\newblock Knot invariants and higher representation theory {I}: diagrammatic
  and geometric categorification of tensor products, 2010.
\newblock arXiv:1001.2020.

\bibitem[Zel80]{Zel}
A.~Zelevinsky.
\newblock Induced representations of reductive p-adic groups. {II}. {O}n
  irreducible representations of {${\rm GL}(n)$}.
\newblock {\em Ann. Sci. \'Ecole Norm. Sup. (4)}, 13(2):165--210, 1980.

\end{thebibliography}

\def\cprime{$'$}

%

\vspace{0.1in}

\noindent A.L:  { \sl \small Department of Mathematics, Columbia University, New
York, NY 10027} \newline \noindent
  {\tt \small email: lauda@math.columbia.edu}

\noindent M.V:  { \sl \small Department of Mathematics, University of California, Davis, Davis, California 95616-8633\newline \noindent
  {\tt \small email: vazirani@math.ucdavis.edu}

%
\end{document}
%